\documentclass[11pt]{article}
\usepackage[dvipsnames]{xcolor}
\usepackage{ulem}
\usepackage{pgfplots}
\usepackage{booktabs}
\usepackage{graphicx,wrapfig,lipsum}
\usepackage{float}    %
\usepackage{verbatim} %
\usepackage{amsmath, amssymb, amsthm, mathtools}  %
\usepackage{subfigure}   %
\usepackage[colorlinks=true,citecolor=blue,linkcolor=blue,urlcolor=blue]{hyperref}
\usepackage{cleveref}
\usepackage{bookmark}
\usepackage{fullpage}
\usepackage{enumerate}
\usepackage{paralist}
\usepackage{xspace}
\usepackage{multirow}
\usepackage{caption}
\usepackage{bbm}
\usepackage{bm}
\usepackage{url}
\usepackage{lipsum}
\usepackage{algorithm}%
\usepackage{algorithmicx}%
\usepackage{algpseudocode}%
\usepackage{color}
\usepackage{microtype}
\usepackage[section]{placeins}
\usepackage{lscape}
\usepackage{multirow}
\usepackage{todonotes}

\usepackage{tcolorbox}

\usepackage[numbers]{natbib}

\usepackage{mathtools}
\DeclarePairedDelimiter{\norm}{\lVert}{\rVert}

\newcommand{\parent}{\mathrm{par}}
\newcommand{\child}{\mathrm{child}}

\newcommand{\R}{\mathbb{R}}

\newcommand{\bbbone}{\mathbb{I}}
\newcommand{\supp}{\mathrm{supp}}
\newcommand{\dist}{\operatorname{dist}}

\newcommand{\B}{\mathcal{B}}
\newcommand{\I}{\mathcal{I}}
\newcommand{\J}{\mathcal{J}}
\newcommand{\setA}{\mathcal{A}}
\newcommand{\setV}{\mathcal{V}}
\newcommand{\setE}{\mathcal{E}}
\newcommand{\setW}{\mathcal{W}}
\newcommand{\setR}{\mathcal{R}}
\newcommand{\C}{\mathcal{C}}
\newcommand{\trim}{\mathrm{prune}}
\newcommand{\bw}{w}

\newcommand{\vx}{\bm{x}}
\newcommand{\vy}{\bm{y}}
\newcommand{\vz}{\bm{z}}

\newcommand{\vo}{\bm{o}}
\newcommand{\vc}{\bm{c}}
\newcommand{\vb}{\bm{b}}
\newcommand{\bs}{\bm{s}}

\newcommand{\vlambda}{\bm{\lambda}}
\newcommand{\valpha}{\bm{\alpha}}
\newcommand{\vxi}{\bm{\xi}}
\newcommand{\vzeta}{\bm{\zeta}}

\newcommand{\vzero}{\bm{0}}

\newcommand{\mQ}{\bm{Q}}
\newcommand{\mY}{\bm{Y}}
\newcommand{\mA}{\bm{A}}

\usepackage{amsmath, amssymb, amsthm}
\newtheorem{theorem}{Theorem}
\newtheorem{lemma}{Lemma}
\newtheorem{proposition}{Proposition}
\newtheorem{corollary}{Corollary}
\newtheorem{claim}{Claim}

\DeclareMathOperator*{\argmin}{argmin}

\theoremstyle{definition}
\newtheorem{definition}{Definition}
\newtheorem{assumption}{Assumption}

\theoremstyle{remark}

\usepackage{algorithm}
\usepackage{algpseudocode}
\usepackage{tikz}
\usepackage{multirow}
\usepackage{hyperref}

\title{Solving Convex Quadratic Optimization with Indicators Over Structured Graphs}

\author{
Aaresh Bhathena\thanks{Department of Industrial and Operations Engineering, 
University of Michigan, Ann Arbor, MI, USA. 
\texttt{aareshfb@umich.edu}}
\and
Salar Fattahi\thanks{Department of Industrial and Operations Engineering, 
University of Michigan, Ann Arbor, MI, USA. 
\texttt{fattahi@umich.edu}}
\and
Andr\'es G\'omez\thanks{Daniel J. Epstein Department of Industrial and Systems Engineering, 
University of Southern California, Los Angeles, CA, USA. 
\texttt{gomezand@usc.edu}}
\and
Simge K\"u\c{c}\"ukyavuz\thanks{Department of Industrial Engineering and Management Sciences, 
Northwestern University, Evanston, IL, USA. 
\texttt{simge@northwestern.edu}}
}

\date{}

\usepackage{comment}
\begin{document}
\maketitle
\begin{abstract}
    This paper studies convex quadratic minimization problems in which each continuous variable is coupled with a binary indicator variable. We focus on the structured setting where the Hessian matrix of the quadratic term is positive definite and exhibits sparsity. We develop an exact parametric dynamic programming algorithm whose computational complexity depends explicitly on the treewidth of the Hessian’s support graph, its volume growth, and an appropriate margin parameter. Under suitable structural conditions, the overall complexity scales linearly with the problem dimension.
    To demonstrate the practical impact of our approach, we introduce a novel framework for joint forecasting and outlier detection by extending exponential smoothing to time series with outliers. Computational experiments on both synthetic and real data sets show that our method significantly outperforms state-of-the-art solvers.
\end{abstract}

\section{Introduction}\label{sec: Intro}

We consider the following mixed-integer quadratic program (MIQP), defined by a symmetric and positive definite matrix $\mQ\in\R^{n\times n}$ and vectors $\vlambda,\vc\in\R^n$:
\begin{subequations}\label{eq: MIQP}
	\begin{align}
		\min_{\vx\in\R^n,\vz\in\{0,1\}^n}\qquad& \dfrac{1}{2}\vx^\top \mQ \vx+\vc^\top \vx+\vlambda^\top \vz\label{eq: MIQP obj}\\ 
		\text{s.t.}\qquad &\vx_i(1-\vz_i)=0&  i=1,2,\ldots, n. \label{eq: MIQP_const}
	\end{align}
\end{subequations}
In this problem, the binary vector $\vz\in\{0,1\}^n$ encodes the support of the continuous vector $\vx\in\R^n$. Specifically the constraint $\vx_i(1-\vz_i)=0$ enforces that $\vx_i=0$ whenever $\vz_i=0$, and $\vz_i=1$ allows $\vx_i\in \R$ to be unconstrained. The vector $\vlambda \in \R^n$ acts as the component-wise regularization parameter that promotes sparsity in $\vx$. We assume throughout that $\vlambda_i>0$ for every $i=1,\ldots, n$ as $\vlambda_i\le 0$ implies that $z_i=1$ at optimality.  Without loss of generality, we also normalize the diagonal entries of $\mQ$ to one by rescaling each variable $\vx_i$ as $\vx_i\rightarrow \vx_i/\sqrt{\mQ_{i,i}}$. 

This work focuses on instances of Problem~\eqref{eq: MIQP} in which the sparsity pattern of the Hessian matrix $\mQ \in \mathbb{R}^{n \times n}$ coincides with the adjacency matrix of a graph—hereafter referred to as the \textit{support graph}—with certain sparsity structures. Specifically, we consider support graphs characterized by a \textit{bounded treewidth} and a \textit{polynomial volume growth} property. The former captures the extent to which the graph resembles a tree, while the latter imposes an upper bound on the number of nodes contained within any fixed graph distance from a given node. While treewidth is a classical concept in graph theory, polynomial volume growth is a less common but equally meaningful structural property; both notions are formally introduced in Section~\ref{sec: preliminaries}.

Problem~\eqref{eq: MIQP} arises in several domains, including sparse regression \citep{bertsimas2016best,bertsimas2024slowly,del2020subset}, probabilistic graphical models~\citep{han2022polynomial,liu2025polyhedral, fattahi2021scalable, Manzour21,kucukyavuz2022consistent,xu2024integer,xu2025}, outlier detection in time series~\citep{atamturk2021sparse}, and network inference~\citep{fattahi2021scalable,ravikumar2025efficient}.
In this paper, we focus on the problem of forecasting with outlier correction in time series, where the proposed formulation arises naturally and proves particularly effective.

\subsection{An overview of our contributions}

At the core of our proposed method lies a \textit{pruning} technique that efficiently eliminates suboptimal choices of the binary vector $\vz \in \{0,1\}^n$, thereby reducing the search space from exponential to polynomial size. Consider the following equivalent formulation of Problem \eqref{eq: MIQP}:
\begin{align}
    \min_{\substack{\vx\in\R^n\\\vz\in\{0,1\}^n\\ \vx\circ(1-\vz)=0}}\! \left\{\dfrac{1}{2}\vx^\top \mQ \vx+\vc^\top \vx+\vlambda^\top \vz\right\} = \min_{\vx\in \mathbb{R}^n}\left\{\min_{\vz\in \{0,1\}^n}\left\{\dfrac{1}{2}\vx^\top\left(\mQ\circ \vz\vz^\top\right)\vx+(\vc\circ \vz)^\top \vx+\vlambda^\top \vz\right\}\right\},
\end{align}
where $\circ$ denotes the entry-wise product. Upon defining a convex quadratic function $p_{\vz}(\vx) := \dfrac{1}{2}\vx^\top\left(\mQ\circ \vz\vz^\top\right)\vx+(\vc\circ \vz)^\top \vx+\vlambda^\top \vz$ for every fixed $\vz\in \{0,1\}^n$, the above problem reduces to the following two-stage optimization problem:
\begin{align}
    \min_{\vx\in \mathbb{R}^n} f(\vx), \quad \text{where}\quad f(\vx) = \min_{\vz\in \{0,1\}^n}\ p_{\vz}(\vx).
\end{align}
The above two-stage formulation induces a corresponding two-stage solution strategy for the original problem~\eqref{eq: MIQP}: first, characterize the \textit{parametric cost} $f:\mathbb{R}^n\to\mathbb{R}$ by projecting out the binary variables $\vz$, and then optimize $f$ directly with respect to $\vx$. Evidently, the first stage constitutes the computational bottleneck: efficient solution of the problem hinges on effectively characterizing $f$, which, as implied by the above reformulation, is a piecewise function composed of up to $2^n$ convex quadratic pieces. In isolation, this approach offers no apparent advantage over exhaustively enumerating all $\vz$ configurations. However, we show that when the Hessian matrix $\mQ$ exhibits a specific graph structure, this enumeration can be dramatically accelerated by systematically {\it pruning} choices of $\vz$ that are provably suboptimal for all possible values of $\vx$.

Our pruning strategy builds upon two key ideas. First, since $\mQ$ is positive definite, any optimal solution $\vx^\star$ must have a bounded norm; that is, $\|\vx^\star\|_{\infty} \le U$ for some constant $U>0$. Consequently, the parametric cost $f$ needs to be characterized only within the bounded region $\mathcal{D} = \{\vx : \|\vx\|_{\infty} \le U\}$ containing the optimal solution. Second, we show that, under certain structural conditions, only a polynomial number of quadratic pieces $p_{\vz}$ with $\vz \in \{0,1\}^n$ are required to represent $f$ within this region. The key insight underlying this result is that, for any two sparsity patterns $\vz_1, \vz_2 \in \{0,1\}^n$ with significant overlap (formally characterized by the notion of \textit{$m$-similarity}; see Definition~\ref{def: msimilar}), the roots of the polynomial function $p_{\vz_1} - p_{\vz_2}$ grow exponentially with the length of the overlap. Hence, for a sufficiently long overlap, these roots lie outside $\mathcal{D}$, implying that either $p_{\vz_1}(\vx) > p_{\vz_2}(\vx)$ or $p_{\vz_1}(\vx) < p_{\vz_2}(\vx)$ within this region. In such cases, one of these quadratic pieces can be safely discarded (pruned), along with its corresponding sparsity pattern, without affecting the characterization of the parametric cost $f$.

While the existence of an efficient representation of the parametric cost $f$ is encouraging, it does not by itself guarantee an efficient procedure for constructing it. To this end, we develop an efficient algorithm, called the \textit{parametric algorithm}, for characterizing $f$ in settings where the matrix $\mQ$ admits a tree decomposition of small width. Our proposed parametric algorithm constructs $f$ efficiently by {dynamic programming} (DP) operating over the tree decomposition of $\mQ$. The overall computational complexity of the proposed method depends on the width of the tree decomposition, the volume growth of the support graph, and a suitable notion of margin for the problem, each of which will be discussed in detail in subsequent sections.

In practice, our parametric algorithm outperforms off-the-shelf solvers by orders of magnitude, solving problems with up to 20,000 variables in seconds to minutes, well beyond the reach of existing off-the-shelf solvers. We further demonstrate the practical relevance of the framework through an application to time-series forecasting. Leveraging a new MIQP formulation, we develop an extension of exponential smoothing that jointly performs forecasting and outlier detection. Computational results show that this integration enhances predictive performance and robustness against outliers on real-world data.

\subsection{Related work}
For a general positive definite matrix $\mQ$, Problem~\eqref{eq: MIQP} is known to be NP-hard~\citep{chen2014NPhard}. Consequently, existing methods typically rely either on sequential convex relaxations, often strengthened through cutting-plane techniques, or on exploiting specific structural properties that render special cases tractable. \medskip

\paragraph{Methods based on sequential convex relaxation.} Classical approaches based on convex relaxation employ \textit{Big-$M$} formulations to encode indicator variables~\citep{glover1975improved}. These formulations have been widely adopted for mixed-integer quadratic and regression-type problems, often incorporated within branch-and-bound techniques~\citep{bertsimas2016best,bertsimas2020sparse,dedieu2021learning}. While effective for small- to medium-scale instances, Big-$M$ formulations generally yield weak relaxations and suffer from numerical instability, leading to poor scalability on large problems~\citep{hazimeh2022sparse}.  
A major breakthrough occurred with the introduction of the \textit{perspective reformulation} technique, which provides significantly tighter convex relaxations for separable mixed-integer quadratic programs. Originally proposed by~\citet{stubbs1996branch} and further developed in a series of works~\citep{akturk2009strong,frangioni2009computational,gunluk2010perspective}, perspective reformulations have since become a cornerstone of modern algorithms for sparse and structured optimization~\citep{bertsimas2020sparse,xie2020scalable,gomez2024note,wei2022ideal,wei2024convex}. 
In practice, these techniques form the backbone of state-of-the-art commercial solvers such as \textsc{Gurobi}, which integrate them into MIQP solvers via specialized cutting planes, presolve reductions, and perspective-based convexification modules. However, despite their effectiveness, such relaxation-based methods must still be embedded within branch-and-bound or branch-and-cut frameworks, whose exponential worst-case complexity continues to render them computationally prohibitive for large-scale problems.\medskip

\paragraph{Methods for structured $\mQ$.} Given the problem’s exponential worst-case complexity, another line of work has studied instances in which $\mQ$ possesses structural properties that enable more efficient algorithms. Examples include cases where $\mQ$ is diagonal~\citep{ceria1999convex}, Stieltjes~\citep{atamturk2018strong,han2022polynomial,liu2025polyhedral}, rank-one~\citep{shafiee2024constrained, han2025compact}, admits a sparse factorization $\mQ = \mQ_0^\top \mQ_0$~\citep{del2020subset}, or exhibits other special structures~\citep{lee2024convexification,das2008algorithms}.
Closely related to this line of work are studies that exploit banded or tree-structured sparsity patterns in $\mQ$. In the special case where $\mQ$ is tridiagonal, \citet{liu2023graph} proposed a DP algorithm based on a shortest-path formulation that recovers the exact solution in $\mathcal{O}(n^2)$ time and memory. Building on this idea, \citet{gomez2024real} extended the approach to general banded matrices and developed a \textit{fully polynomial-time approximation scheme} (FPTAS) for Problem~\eqref{eq: MIQP}.

As an extension of the DP approach introduced by \citet{liu2023graph}, \citet{bhathena2025parametric} established that Problem~\eqref{eq: MIQP} can be solved exactly in $\mathcal{O}(n^2)$ time when the sparsity graph of $\mQ$ is a path or a tree. This result already motivates the use of tree decompositions, which provide a principled framework for extending tractability beyond tree-structured graphs. However, the approach by \citet{bhathena2025parametric} relies on a delicate characterization of the conjugates of univariate convex quadratic functions, a property that does not generalize easily beyond tree graphs.\medskip

\paragraph{Methods based on small treewidth.} Graphs with small {treewidth} were first introduced by~\citet{halin1976s} and subsequently popularized by~\citet{robertson1986graph} \citep[see][for a historical account]{diestel2025graph}. The notion of bounded treewidth is fundamental in algorithmic graph theory because it characterizes classes of graphs on which DP can be applied efficiently~\citep{bodlaender1992tourist}. Since the late 1980s, it has been well established that several classically intractable combinatorial problems, such as \textit{Independent Set}, \textit{Graph Coloring}, and \textit{Hamiltonian Cycle}, admit polynomial-time solutions via DP when restricted to graphs with small treewidth~\citep{bern1987linear, bodlaender1988dynamic, arnborg1989linear}.

In contrast, Problem~\eqref{eq: MIQP} exhibits a mixed-integer structure, involving both binary and continuous decision variables. In this broader setting, tree decompositions have been investigated extensively in the context of polynomial and mixed-integer optimization~\citep{wainwright2004treewidth, bienstock2018lp, madani2017finding, kojima2005sparsity, waki2006sums}, encompassing subclasses of problems closely related to ours. In particular,~\citet{bienstock2024solving} consider quadratic optimization problems in which sparsity across disjoint variable blocks is governed by binary indicator variables subject to coupling constraints. Their framework is applicable to Problem~\eqref{eq: MIQP} when the Hessian matrix $\mQ$ has bounded treewidth. Nonetheless, their approach yields only an FPTAS algorithm, and, to the best of our knowledge, no practical implementation of the proposed algorithm has been reported in the literature.

\subsection{Outline and overview}
The remainder of the paper is organized as follows. Section~\ref{sec: preliminaries} introduces the necessary preliminaries and background. Section~\ref{sec: motivating application} presents an application to time-series forecasting based on exponential smoothing with outlier detection. Section~\ref{sec: dp formulation} presents an (inefficient) DP approach for solving Problem~\eqref{eq: MIQP} by sequentially characterizing its parametric cost in exponential time. Although this DP formulation is not practical, it serves as the foundation for our efficient pruning strategy, developed in Section~\ref{sec: Parametric algorithm}. Section~\ref{sec::theoretical-analysis} provides theoretical guarantees on the correctness and runtime of the proposed algorithm. Section~\ref{sec: experiments synthetic} reports numerical results on synthetic instances, as well as real-world time-series data from the NAB dataset. Finally, Section~\ref{sec: conclusion} concludes with a summary of the main contributions.

\section{Preliminaries and background}\label{sec: preliminaries}
Matrices and vectors are denoted by bold uppercase and lowercase symbols (e.g., $\mQ$ and $\vc$), respectively, while scalars are denoted by unbolded symbols (e.g., $n$).
Given a matrix $\mQ \in \mathbb{R}^{n \times n}$ and index sets $\mathcal{I}, \mathcal{J} \subseteq \{1, \ldots, n\}$, we denote by $\mQ_{\mathcal{I},\mathcal{J}}$ the submatrix of $\mQ$ consisting of rows indexed by $\mathcal{I}$ and columns indexed by $\mathcal{J}$. Similarly, for a vector $\vc \in \mathbb{R}^n$, we write $\vc_{\mathcal{J}}$ for the subvector of $\vc$ restricted to the indices in $\mathcal{J}$.  
We use $\bbbone(x)$ to denote the indicator function on $\mathbb{R}$, which equals $0$ if $x = 0$ and $1$ otherwise. 
For a symmetric matrix $\mQ \in \mathbb{R}^{n \times n}$, let $\mu_{\min}(\mQ)$ and $\mu_{\max}(\mQ)$ denote its smallest and largest eigenvalues, respectively. The spectral condition number of $\mQ$ is defined as $\kappa_2(\mQ) := \mu_{\max}(\mQ) / \mu_{\min}(\mQ)$. Similarly, the condition number of $\mQ$ in the induced $\infty$-norm is defined as $\kappa_\infty(\mQ) = \|\mQ^{-1}\|_\infty\|\mQ\|_\infty$. When the argument is omitted, i.e., when we write $\mu_{\min}$, $\mu_{\max}$, $\kappa_2$, or $\kappa_\infty$, these quantities refer to the eigenvalues and condition numbers of the Hessian matrix $\mQ$.  
The entrywise $\ell_{1,1}$-norm of $\bm{S}$ is defined as $\norm{\bm{S}}_{1,1} := \sum_{i=1}^n \sum_{j=1}^n |\bm{S}_{ij}|$. We say that $\bm{S}$ is a \textit{banded matrix} with bandwidth $\bw \in \mathbb{Z}_{+}$ if $\bm{S}_{ij} = 0$ for all $|i-j| > \bw$.
We denote by $f^{\star}$ the optimal objective value of Problem~\eqref{eq: MIQP}, and by $(\vx^{\star}, \vz^{\star})$ its corresponding optimal solution.  

Given a symmetric matrix $\mQ\in R^{n\times n}$, its {\it support graph}, denoted by $\supp(\mQ)$, is a simple unweighted graph $\textsf{G}=(\setV_{\textsf{G}},\setE_{\textsf{G}})$ with vertex set $\setV_{\textsf{G}}=\{1,2,\ldots,n\}$, where an edge $(i,j)\in \setE_{\textsf{G}}$ exists if and only if $\mQ_{ij}\ne 0$ for $i\ne j$. For any two nodes $u,v\in\setV_{\textsf{G}}$, let $\dist(u,v)$ denote the length of the shortest path between $u$ and $v$ in $\supp(\mQ)$. More generally, for subset of nodes $\I,\J\subseteq \setV_{\textsf{G}}$, we define $\dist(\I,\J)=\min\limits_{i\in \I, j\in \J} \left\{\dist(i,j)\right\}$. 

\subsection{Tree decomposition and treewidth}\label{sec: tree-decomposition}

\begin{definition}[Tree decomposition]\label{def: tree decomposition}
	The tree decomposition of a graph $\textsf{G}=(\setV_{\textsf{G}},\setE_{\textsf{G}})$ is a pair $(\textsf{T},\B)$, where $\textsf{T}=(\setV_\textsf{T},\setE_\textsf{T})$ is a tree and $\B=\{\B_u : u\in \setV_\textsf{T}\}$ is a family of subsets (bags) of $\setV_{\textsf{G}}$, satisfying the following conditions:
	\begin{enumerate}
		\item $\setV_\textsf{G}=\bigcup_{u\in \setV_\textsf{T}}\B_u$. That is, every node of $\textsf{G}$ appears in at least one bag.
		\item For every edge $(i,j)\in \setE_\textsf{G}$, there exists a bag $\B_u$ containing both $i$ and $j$. 
		\item For every node $i \in \setV_\textsf{G}$, the set $\{u \in \setV_\textsf{T} : i \in \B_u\}$ induces a connected subtree of $\textsf{T}$. That is, bags containing $i$ form a connected subtree of $\textsf{T}$.
	\end{enumerate}
\end{definition}
The width of a tree decomposition is given by $\max\limits_{u\in \setV_\textsf{T}}\{|\B_u|-1\}$. A graph may admit many different tree decompositions. A trivial decomposition places all vertices of $\setV_\textsf{G}$ into a single bag, 
yielding width $|\setV_\textsf{G}|-1$. A more informative structural measure is the \textit{treewidth}, denoted by~$\omega$, defined as the minimum width among all tree decompositions of~$\textsf{G}$. Intuitively, treewidth measures how close a graph is to being a tree. For example, trees have treewidth 1. Series-parallel graphs have treewidth 2. On the other hand, fully dense graphs---being far from tree-like---have a maximum treewidth $|\setV_\textsf{G}|-1$. 

Next, we introduce the notion of a \textit{balanced tree decomposition}.

\begin{definition}[balanced tree decomposition]\label{def: balanced tree decomposition}
    Let $(\textsf{T},\B)$ be a tree decomposition with width $\omega$. We call $(\textsf{T},\B)$ a {balanced tree decomposition} if:
	\begin{enumerate}
		\item Each bag of $\textsf{T}$ contains exactly $\omega+1$ nodes.
		\item For every edge $(u,v) \in \setE_\textsf{T}$, the corresponding bags satisfy $|\B_u \cap \B_v| = \omega$.
	\end{enumerate}
\end{definition}

Finding an optimal tree decomposition with the smallest possible width—i.e., one that matches the treewidth $\omega$—is NP-hard~\citep{arnborg1987complexity}. Nevertheless, for graphs with bounded treewidth, both recognition and construction of optimal tree decompositions can be performed in linear time~\citep{bodlaender1996linear}. Beyond these theoretical results, practical heuristics such as \textit{fill-reducing} and \textit{nested dissection} algorithms often yield decompositions with near-optimal widths in practice~\citep{heggernes2006minimal}. Moreover, any tree decomposition with width~$\omega$ can be transformed into a balanced one of the same width by adding nodes to existing bags and, if necessary, inserting or removing bags. An illustrative example is shown in Figure~\ref{fig: balanced tree decomposition}. Balanced tree decompositions always exist, and, given a tree decomposition of width $\omega$, can be obtained in $\mathcal{O}(n)$ via the so-called ``nice'' tree decompositions; see~\citep[Section 2]{bodlaender1996efficient} for definitions and the construction.
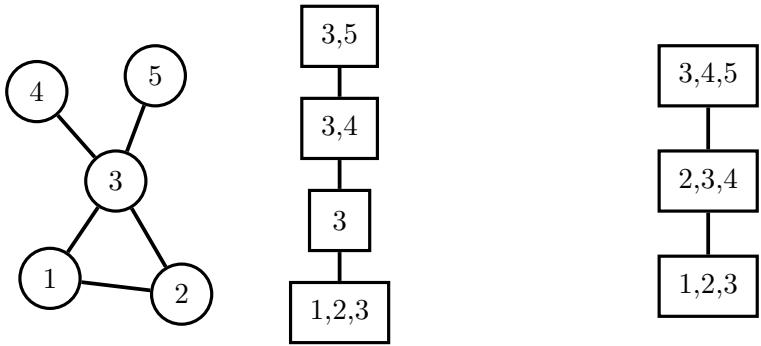
\begin{figure}[htbp]
	\centering
	\begin{tikzpicture}[scale=.7,auto=center,every node/.style={circle,draw=black!100, very thick, minimum size=8mm}, squarenode/.style={rectangle,draw=black!100, very thick, minimum size=8mm},roundnode/.style={rectangle, draw=red!60, fill=red!5, very thick, minimum size=7mm},yellownode/.style={circle, draw=yellow!100, fill=yellow!15, very thick, minimum size=7mm},greennode/.style={circle, draw=green!60, fill=green!15, very thick, minimum size=7mm},bluenode/.style={circle, draw=blue!60, fill=blue!20, very thick, minimum size=7mm}] 
		\node (a1) at (-7.5,0.15) {1};  
		\node (a2) at (-5,-0.15)  {2};
		\node (a3) at (-6.25,2)  {3};
		\node (a4) at (-7.75,3.7)  {4}; 
		\node (a5) at (-5.5,4)  {5};
		
		\draw[-, line width=0.5mm] (a1) -- (a2); %
		\draw[-, line width=0.5mm] (a2) -- (a3);
		\draw[-, line width=0.5mm] (a1) -- (a3);
		\draw[-, line width=0.5mm] (a4) -- (a3);
		\draw[-, line width=0.5mm] (a3) -- (a5);
		
		\node (b1) at (-2,-0.5) [rectangle,draw] {\text{ 1,2,3 }  };
		\node (b2) at (-2,1.25) [rectangle,draw] {\text{  3 }  };
		\node (b3) at (-2,3) [rectangle,draw] {\text{  3,4 }  };
		\node (b4) at (-2,4.75) [rectangle,draw] {\text{ 3,5 }  };
		
		\draw[-, line width=0.5mm] (b1) -- (b2);
		\draw[-, line width=0.5mm] (b3) -- (b2);
		\draw[-, line width=0.5mm] (b3) -- (b4);
		
		\node (c1) at (5,0) [rectangle,draw] {\text{ 1,2,3 }  };
		\node (c2) at (5,2) [rectangle,draw] {\text{ 2,3,4 }  };
		\node (c3) at (5,4) [rectangle,draw] {\text{ 3,4,5 }  };
		
		\draw[-, line width=0.5mm] (c1) -- (c2);
		\draw[-, line width=0.5mm] (c3) -- (c2);
	\end{tikzpicture} 
	\caption{A graph with a tree decomposition that violates Definition~\ref{def: balanced tree decomposition} (left), and its corresponding balanced tree decomposition (right).}
	\label{fig: balanced tree decomposition}
\end{figure}

Throughout this work, we assume that a balanced tree decomposition of $\supp(\mQ)$ with width $\omega$ is available. For example, for our proposed ESOC formulation, $\supp(\mQ)$ has treewidth 2, and the corresponding balanced tree decomposition can be readily obtained (see Figure~\ref{fig: ESO-support-graph}). 
We note that our proposed method does not require a tree decomposition of minimum width; any decomposition with reasonably small width suffices.

We next describe a procedure for labeling the nodes of $\supp(\mQ)$ induced by its balanced tree decomposition~$\textsf{T}$. As will become evident later, this labeling plays a crucial role in establishing the efficiency of the proposed algorithm. We begin by labeling the bags in~$\textsf{T}$.  
By the second property of the balanced tree decomposition, $(\textsf{T}, \mathcal{B})$ contains $n - \omega$ bags. For simplicity, we assume that the edges in~$\textsf{T}$ are oriented naturally toward a designated root. Under this convention, each bag may have multiple parents but at most one child. We denote by $\child_{\textsf{T}}(u)$ the child of bag~$u$ in~$\textsf{T}$, and by $\parent_{\textsf{T}}(u)$ the set of its parents.  
Labels are assigned to the bags according to a topological ordering: for every bag~$u$, we require $u < \child_{\textsf{T}}(u)$. Since $\textsf{T}$ contains $n - \omega$ bags, the root bag receives label $n - \omega$. As $\textsf{T}$ is acyclic, such a topological labeling always exists and can be computed in $\mathcal{O}(n)$ time and memory~\cite[Algorithm~3.8]{ahuja1988network}.  

We now proceed to label the nodes of $\supp(\mQ)$ based on its labeled balanced tree decomposition. For any non-root bag~$u$, by the definition of a balanced tree decomposition, the set $\mathcal{B}_u \setminus \mathcal{B}_{\child_{\textsf{T}}(u)}$ contains a unique node of $\supp(\mQ)$, which we label by~$u$. The nodes in the root bag of~$\textsf{T}$ are labeled arbitrarily with the remaining labels $\{n - \omega + 1, \ldots, n\}$. An illustration of this labeling scheme is provided in Figure~\ref{fig: Label for trees}. Since both $\mathcal{B}_u$ and $\mathcal{B}_{\child_{\textsf{T}}(u)}$ have size $\omega+1$, computing each set difference requires $\mathcal{O}(\omega^2)$ time. Repeating this operation for all bags yields a total labeling cost of $\mathcal{O}(n\omega^2)$ time. The memory required is $\mathcal{O}(n\omega)$, which corresponds to storing the labels for each bag.

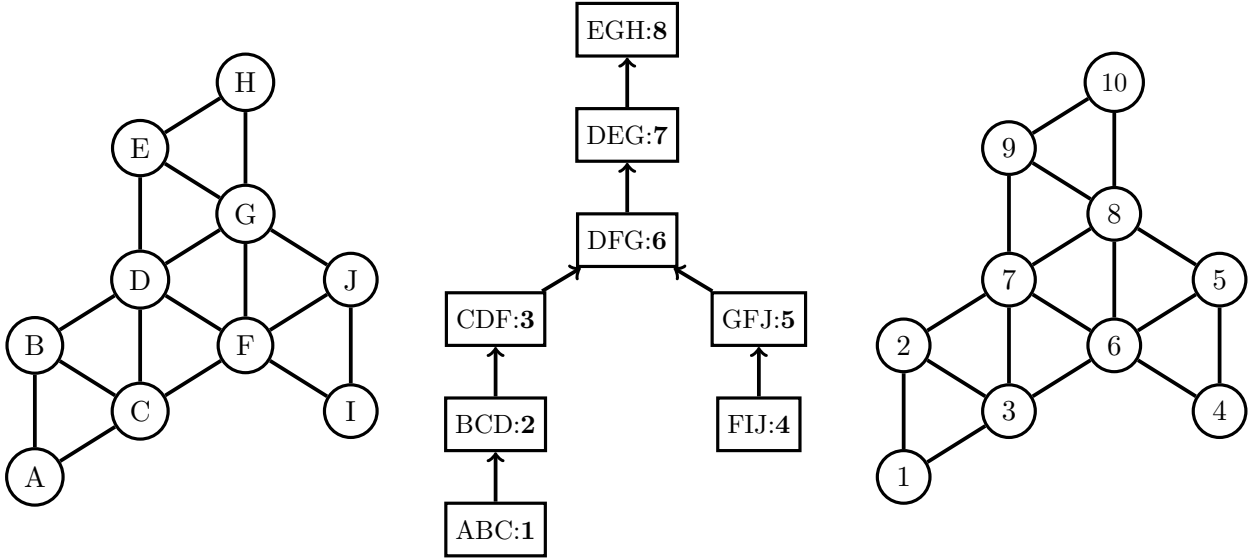
\begin{figure}[htbp]
	\centering
	\begin{tikzpicture}  
		[scale=.7,auto=center,every node/.style={circle,draw=black!100, very thick, minimum size=7mm}, snode/.style={circle,draw=white!0, thin, minimum size=8mm},roundnode/.style={rectangle, draw=red!60, fill=red!5, very thick, minimum size=7mm},yellownode/.style={circle, draw=yellow!100, fill=yellow!15, very thick, minimum size=7mm},greennode/.style={circle, draw=green!60, fill=green!15, very thick, minimum size=7mm},bluenode/.style={circle, draw=blue!60, fill=blue!20, very thick, minimum size=7mm}] 
		
		\node (b1) at (1.25,-1) [rectangle,draw] {\small ABC:\textbf{1}};
		\node (b2) at (1.25,1) [rectangle,draw] {\small BCD:\textbf{2} };
		\node (b3) at (1.25,3) [rectangle,draw] {\small CDF:\textbf{3}  };
		
		\node (b4) at (6.25,1) [rectangle,draw] {\small FIJ:\textbf{4}  };
		\node (b5) at (6.25,3) [rectangle,draw] {\small GFJ:\textbf{5}  };
		
		\node (b6) at (3.75,4.5) [rectangle,draw] {\small DFG:\textbf{6} };
		\node (b7) at (3.75,6.5) [rectangle,draw] {\small DEG:\textbf{7} };
		\node (b8) at (3.75,8.5) [rectangle,draw] {\small EGH:\textbf{8} };
		
		\node (a1) at (-7.5,0) {A};  
		\node (a2) at (-7.5,2.5)  {B}; 
		\node (a3) at (-5.5,1.25)  {C}; 
		\node (a4) at (-5.5,3.75)  {D}; 
		\node (a5) at (-5.5,6.25)  {E}; 
		
		\node (a6) at (-3.5,2.5)  {F};
		\node (a7) at (-3.5,5)  {G};
		\node (a8) at (-3.5,7.5)  {H};
		\node (a9) at (-1.5,1.25)  {I};
		\node (a10) at (-1.5,3.75)  {J};
		
		\draw[-, line width=0.5mm] (a1) -- (a2); 
		\draw[-, line width=0.5mm] (a2) -- (a3);
		\draw[-, line width=0.5mm] (a1) -- (a3);
		\draw[-, line width=0.5mm] (a4) -- (a3);
		\draw[-, line width=0.5mm] (a2) -- (a4);
		\draw[-, line width=0.5mm] (a3) -- (a6);
		\draw[-, line width=0.5mm] (a6) -- (a4);
		\draw[-, line width=0.5mm] (a7) -- (a6);
		\draw[-, line width=0.5mm] (a7) -- (a4);
		\draw[-, line width=0.5mm] (a5) -- (a4);
		\draw[-, line width=0.5mm] (a7) -- (a5);
		\draw[-, line width=0.5mm] (a7) -- (a8);
		\draw[-, line width=0.5mm] (a5) -- (a8);
		\draw[-, line width=0.5mm] (a7) -- (a10);
		\draw[-, line width=0.5mm] (a6) -- (a10);
		\draw[-, line width=0.5mm] (a6) -- (a9);
		\draw[-, line width=0.5mm] (a9) -- (a10);
		
		\node (c1) at (9,0) {1};  
		\node (c2) at (9,2.5)  {2}; 
		\node (c3) at (11,1.25)  {3}; 
		\node (c4) at (11,3.75)  {7}; 
		\node (c5) at (11,6.25)  {9}; 
		
		\node (c6) at (13,2.5)  {6};
		\node (c7) at (13,5)  {8};
		\node (c8) at (13,7.5)  {\small 10};
		\node (c9) at (15,1.25)  {4};
		\node (c10) at (15,3.75)  {5 };
		
		\draw[-, line width=0.5mm] (c1) -- (c2); 
		\draw[-, line width=0.5mm] (c2) -- (c3);
		\draw[-, line width=0.5mm] (c1) -- (c3);
		\draw[-, line width=0.5mm] (c4) -- (c3);
		\draw[-, line width=0.5mm] (c2) -- (c4);
		\draw[-, line width=0.5mm] (c3) -- (c6);
		\draw[-, line width=0.5mm] (c6) -- (c4);
		\draw[-, line width=0.5mm] (c7) -- (c6);
		\draw[-, line width=0.5mm] (c7) -- (c4);
		\draw[-, line width=0.5mm] (c5) -- (c4);
		\draw[-, line width=0.5mm] (c7) -- (c5);
		\draw[-, line width=0.5mm] (c7) -- (c8);
		\draw[-, line width=0.5mm] (c5) -- (c8);
		\draw[-, line width=0.5mm] (c7) -- (c10);
		\draw[-, line width=0.5mm] (c6) -- (c10);
		\draw[-, line width=0.5mm] (c6) -- (c9);
		\draw[-, line width=0.5mm] (c9) -- (c10);

		\draw[->, line width=0.5mm] (b1) -- (b2); 
		\draw[<-, line width=0.5mm] (b3) -- (b2); 
		\draw[<-, line width=0.5mm] (b5) -- (b4); 
		\draw[<-, line width=0.5mm] (b6) -- (b3);
		\draw[<-, line width=0.5mm] (b6) -- (b5); 
		\draw[<-, line width=0.5mm] (b7) -- (b6);
		\draw[<-, line width=0.5mm] (b8) -- (b7);

	\end{tikzpicture}
	\caption{The figure on the left illustrates a graph~$\textsf{G}$ with nodes labeled arbitrarily. The center panel shows a balanced tree decomposition of~$\textsf{G}$, where the bags are labeled according to the topological ordering (labels shown in bold). The panel on the right shows the resulting labeling of the nodes in~$\textsf{G}$ obtained by the described procedure.}
	\label{fig: Label for trees}
	\end{figure}
	
	For any node $u \in \supp(\mQ)$ with $u \le n - \omega$, we define $\supp_u(\mQ)$ as the subgraph of $\supp(\mQ)$ induced by the nodes contained in the largest subtree of $\textsf{T}$ comprising the bag $\mathcal{B}_u$ and all of its ancestors. For $u > n - \omega$, we set $\supp_u(\mQ) := \supp(\mQ)$. For any $u$, the treewidth of $\supp_u(\mQ)$ does not exceed~$\omega$~\citep[Lemma~12.4.1]{diestel2025graph}.  
As an example, for the graph shown in Figure~\ref{fig: Label for trees}, $\supp_5(\mQ)$ is the subgraph induced by the nodes in $\mathcal{B}_5 \cup \mathcal{B}_4 = \{5,6,8\} \cup \{4,5,6\}$, while $\supp_6(\mQ)$ is induced by the nodes in $\bigcup_{i=1}^6 \mathcal{B}_i$.  
For a node $u \in \supp(\mQ)$, we denote by $\mQ_{[u]}$ the principal submatrix of $\mQ$ indexed by the nodes of $\supp_u(\mQ)$; clearly, $\supp(\mQ_{[u]}) = \supp_u(\mQ)$. Similarly, $\vc_{[u]}$ and $\vlambda_{[u]}$ denote the subvectors of $\vc$ and $\vlambda$ restricted to these indices. We use $\mathcal{J}_u$ to denote the set of nodes in $\supp_u(\mQ)$ excluding those in $\mathcal{B}_u$, and let $n_u := |\mathcal{J}_u|$.  
Recall that the tree decomposition $\textsf{T}$ contains $n - \omega$ bags, with $\mathcal{B}_{n - \omega}$ designated as the root. With a slight abuse of notation, we define auxiliary sets $\mathcal{B}_u := \{u, u + 1, \ldots, n\}$ for $u \in \{n - \omega + 1, \ldots, n\}$. Although these sets are not part of the original tree decomposition, we refer to them as \textit{bags} for convenience. For $u > n - \omega$, we also define $\parent_{\textsf{T}}(u) := \{u - 1\}$.  
Finally, for any $u$, we define $\tau_u := \min\{\omega, n - u\}$, ensuring that $|\mathcal{B}_u| = \tau_u + 1$. When $u$ is clear from context, we simply write $\tau$.

    \subsection{The local parametric cost}

    Recall that $\mQ_{[u]}\in \mathbb{R}^{n_u+\tau+1\times n_u+\tau+1}$ is a principal submatrix of $\mQ$ indexed by $\J_u\cup \B_u$. Let $\pi_u: \J_u\cup \B_u\to \{1,\dots n_u+\tau+1\}$ be the canonical indexing map that assigns to each $i \in \J_u \cup \B_u$ its corresponding row/column position within the submatrix $\mQ_{[u]}$.
	For any $u\in \{1,\ldots,n\}$ and $\tau=\min\{\omega,n-u\}$, the local parametric cost, $f_u:\R^{\tau+1}\to \R$, is defined as 
	\begin{subequations}\label{eq: f_u tree}
		\begin{align}
			f_u(\valpha_{\B_u}):=\min_{\vx\in\R^{n_u+\tau+1},\vz\in\{0,1\}^{ n_u}}\quad& \dfrac{1}{2}\vx^\top \mQ_{[u]}\vx+{\vc^\top_{[u]} \vx}+{\vlambda^\top_{\J_u }} \vz\label{eq: f_u objective tree}\\ 
			\text{s.t.}\quad &\vx_i(1-\vz_i)=0\qquad i\in \pi_u(\J_u)\label{const: sparsity tree}\\
			& \vx_{\pi_u(\B_u)}=\valpha_{\B_u}. \label{const: x=alpha tree}
		\end{align}
	\end{subequations}
    We note that the subscript in $\valpha_{\B_u}$ is not mathematically necessary; however, as will be seen later, it helps streamline and clarify the subsequent arguments.
Intuitively, $f_u$ denotes the optimal value of the subproblem defined over $\supp_u(\mQ)$ after fixing the local continuous variables associated with $\mathcal{B}_u$ to $\valpha_{\B_u} \in \mathbb{R}^{\tau + 1}$. As will be explained, this local parametric cost establishes a structured connection between subproblems and enables the DP algorithm to efficiently propagate information through the tree decomposition.

Our next lemma establishes that the local parametric cost can be expressed as the pointwise minimum of at most $2^{n_u}$ strongly convex quadratic functions. The proof of this lemma is provided in Appendix~\ref{app:lemma1}.

\begin{lemma}\label{lemma: f_u as p_s}
	Fix any $u\in\{1,2,\ldots,n\}$. The local parametric cost  $f_u:\R^{\tau+1}\to\R$ can be written as 
	\begin{align}\label{eq: f_u in terms of p_u}
		f_u(\valpha_{\B_u})=\min_{\bs\in\{0,1\}^{n_u}} \left\{p_{u,\bs}(\valpha_{\B_u})\right\}
	\end{align}
	where, for every $\bs\in\{0,1\}^{n_u}$, $p_{u,\bs}(\valpha)$ is a strongly convex quadratic function.
	In particular, let $\J_{u,\bs} = \{i\in \J_u\mid \bs_i=1\}$. Then $p_{u,\bs}(\valpha_{\B_u})$ is given by

	\begin{align*}
			&p_{u,\bs}(\valpha_{\scriptscriptstyle \B_u})= \frac{1}{2}\valpha_{\scriptscriptstyle \B_u}\mA_{u,\bs}\valpha_{\scriptscriptstyle \B_u}+\vb_{u,\bs}^\top\valpha_{\scriptscriptstyle \B_u}+d_{u,\bs},\\
            &\text{where}\ \ \begin{cases}
			    \mA_{u,\bs}&=\mQ_{\scriptscriptstyle \B_u, \B_u}-\mQ_{\scriptscriptstyle \B_u,\J_{u,\bs}} \left(\mQ_{\scriptscriptstyle \J_{u,\bs},\J_{u,\bs}}\right)^{-1}\mQ_{\scriptscriptstyle \B_u,\J_{u,\bs}}^\top\\
		\vb_{u,\bs}&=\vc_{\scriptscriptstyle \B_u}-\vc_{\scriptscriptstyle \J_{u,\bs}}^T \left(\mQ_{\scriptscriptstyle \J_{u,\bs},\J_{u,\bs}}\right)^{-1}\mQ_{\scriptscriptstyle \B_u,\J_{u,\bs}}^{\top}\\
		d_{u,\bs}&= -\frac{1}{2}\vc_{\J_{u,\bs}}^\top (\mQ_{\scriptscriptstyle \J_{u,\bs},\J_{u,\bs}})^{-1}\vc_{\scriptscriptstyle \J_{u,\bs}}+\sum_{i\in \J_{u,\bs}} \vlambda_{i}.
			\end{cases}
	\end{align*}
	\end{lemma}
	
    Since $f_u$ can be expressed as the minimum of a collection of quadratic functions, it can, in principle, be stored in memory by saving the coefficients $(\mA_{u,\bs}, \vb_{u,\bs}, d_{u,\bs})$ corresponding to each quadratic piece $p_{u,\bs}$. Storing the coefficients of each $p_{u,\bs}$ requires $\mathcal{O}(\omega^2)$ memory. However, because the number of sparsity patterns~$\bs$ grows exponentially with the number of nodes~$n_u$ in $\supp_u(\mQ)$, the total memory required to store~$f_u$ also scales exponentially. Evidently, such direct storage is impractical.  

When $\supp(\mQ)$ is a tree (i.e., $\omega = 1$), \citet{bhathena2025parametric} demonstrated that the parametric cost~$f_u$ can be represented as the minimum of only $\mathcal{O}(n_u)$ quadratic pieces, resulting in a linear memory requirement. For the general case~$\omega > 1$, however, the existence of a compact representation of~$f_u$ remains an open question. In this work, we address this challenge by leveraging additional structural properties of the problem, such as its volume growth, which we describe next.

    \subsection{Volume growth of a graph}
    For any node $u$ in $\supp(\mQ)$, define 
$\mathcal{V}_{u,m} := \{\, i \in \mathcal{J}_u \mid \dist(i, \mathcal{B}_u) \le m \,\}$
as the set of nodes in $\mathcal{J}_u$ lying within distance~$m$ of bag $\mathcal{B}_u$. 
This set is referred to as the \textit{$m$-neighborhood} of $\mathcal{B}_u$ within the induced subgraph $\supp_u(\mQ)$. Let 
$\Delta_m := \max_{u \in \{1, \ldots, n\}} |\mathcal{V}_{u,m}|$.
The \textit{volume growth function} of $\supp(\mQ)$ characterizes how rapidly $\Delta_m$ increases with~$m$.

	\begin{assumption}[Polynomial volume growth]\label{assumption: poly growth of m-deg}
		There exist constants $\gamma ,\delta > 0$ such that $$\Delta_m \le \delta\, m^\gamma \qquad \text{ for all } m \ge 1.$$
	\end{assumption}
	
	Assumption~\ref{assumption: poly growth of m-deg} is closely related to the notion of uniform polynomial volume growth~\citep{ebrahimnejad2021planar,kontogeorgiou2022random}, although the precise definitions may vary slightly. 
    This assumption holds for many well-studied graph families. Examples include trees with $\mathcal{O}(1)$ leaves, cycles, grid graphs, and support graphs of banded matrices. In contrast, trees with $\mathcal{O}(n)$ leaves may not satisfy this assumption. Notable examples are complete binary trees, for which $\Delta_m = \mathcal{O}(2^m)$, and star graphs, where $\Delta_1 = n - 2$.

\section{Application: Exponential smoothing with outlier correction}\label{sec: motivating application}

To motivate our methodological framework for convex quadratic optimization over structured graphs, we now introduce a practical application in time-series analysis. Specifically, 
 we adopt a unified perspective that integrates forecasting and outlier detection. Building on classical exponential smoothing, a simple and flexible method for time-series forecasting, we extend it with an MIQP-based formulation to handle outliers effectively. The next subsection reviews exponential smoothing and discusses its limitations, motivating our proposed extension.

\subsection{Exponential smoothing}\label{sec: exponential smoothing}
Exponential smoothing encompasses a family of methods widely used in operations research for time series with varying characteristics such as level, trend, and seasonality.  
By assigning exponentially decreasing weights to past observations, these methods smooth short-term fluctuations while preserving long-term structure, making them valuable for forecasting and decision-making in dynamic environments. Applications span retail and inventory management \citep{ratnasari2019demand,billah2006exponential,gardner2006exponential}, market analysis and policy design~\citep{kahraman2023comparison,tran2020comprehensive}, and quality control \citep{taghezouit2021simple}; see \citep{hyndman2008forecasting} for additional examples.

Among the various exponential smoothing methods, \textit{Simple Exponential Smoothing} (SES)—also known as single exponential smoothing—provides a foundational formulation. Formally, let $\{\vy_t\}_{t=1}^T$ denote a univariate time series signal. The smoothed sequence $\{\vx_t\}_{t=1}^T$ is defined recursively as: 
\begin{align}\label{eq: SES}\tag{SES}
	\vx_t = \beta \vy_t + (1 - \beta) \vx_{t-1} \quad \text{for } t = 2, \dots, T;
\end{align} 
with initialization $\vx_1 = \vy_1$. The smoothing parameter $\beta \in (0, 1) $ controls the weight assigned to the most recent observation. Larger values of $\beta$ make the method more responsive to changes but more sensitive to noise, while smaller values produce smoother outputs at the cost of responsiveness.

To illustrate the limitations of SES, we consider a time series of network traffic volume to a cloud server, measured in bytes at five-minute intervals. This dataset was collected by Amazon CloudWatch and is publicly available via the Numenta Anomaly Benchmark (NAB) \citep{ahmad2015nab}.
Figure~\ref{fig: SES failure revisited} (top row) illustrates SES applied to this time series with smoothing factors $\beta=0.5,0.2,0.05$. The circled points indicate outliers labeled by NAB. With larger values of $\beta$, the smoothed series more closely follows the raw data, but fails to exclude the outliers in the process. Conversely, small values of $\beta$ dampen the influence of the outliers while introducing lag relative to the underlying series. This behavior highlights a fundamental limitation of SES: it cannot simultaneously suppress outliers and adapt quickly to changes in the underlying signal. Motivated by this inherent trade-off in SES, we formulate a robust extension, which we call \textit{exponential smoothing with outlier correction }(ESOC). We focus primarily on SES, but note that our formulation extends naturally to other variants, including double- and triple-exponential smoothing.
\begin{figure}[ht]
		\centering
        \includegraphics[width=0.33\linewidth]{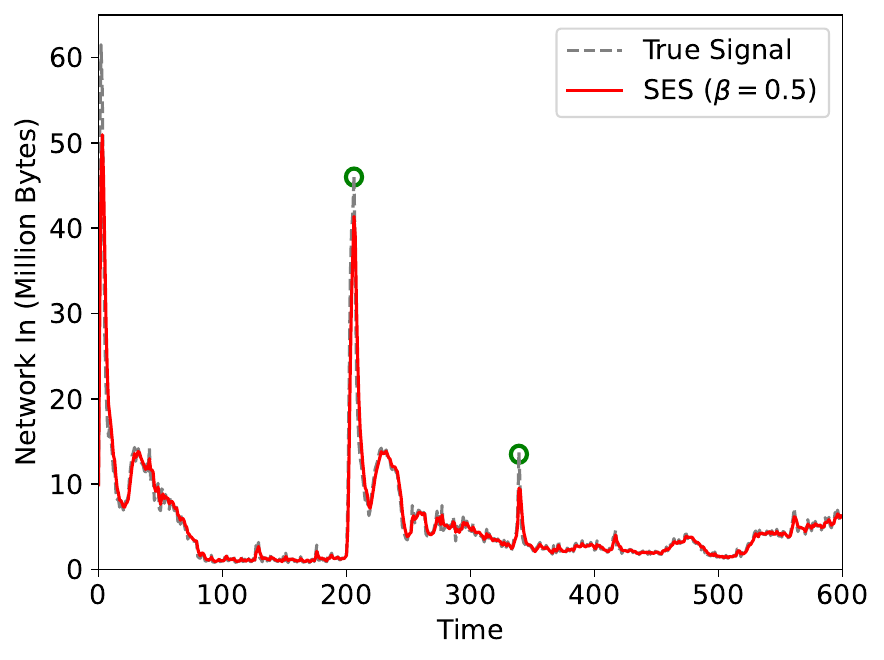}
		\includegraphics[width=0.33\linewidth]{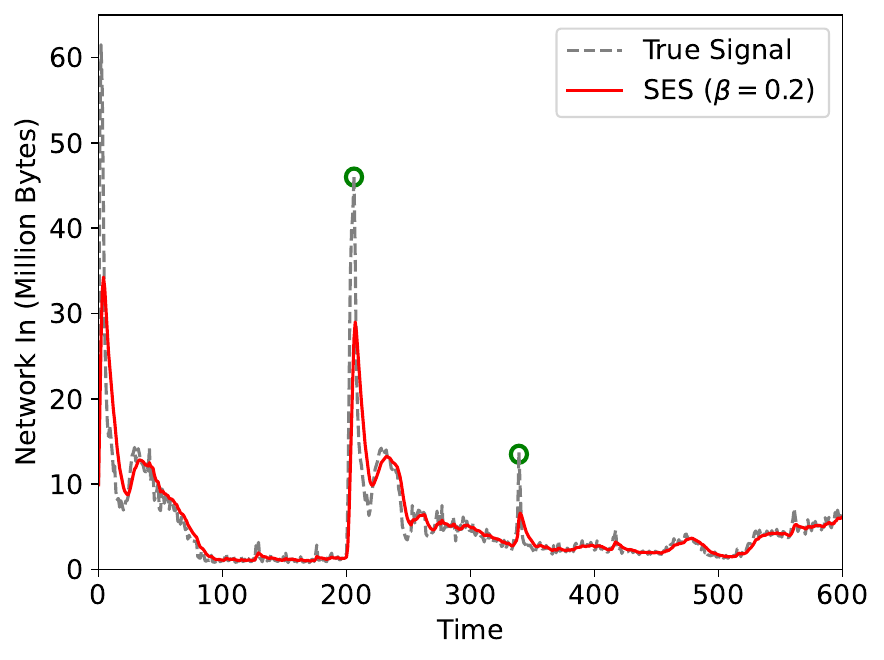}
        \includegraphics[width=0.33\linewidth]{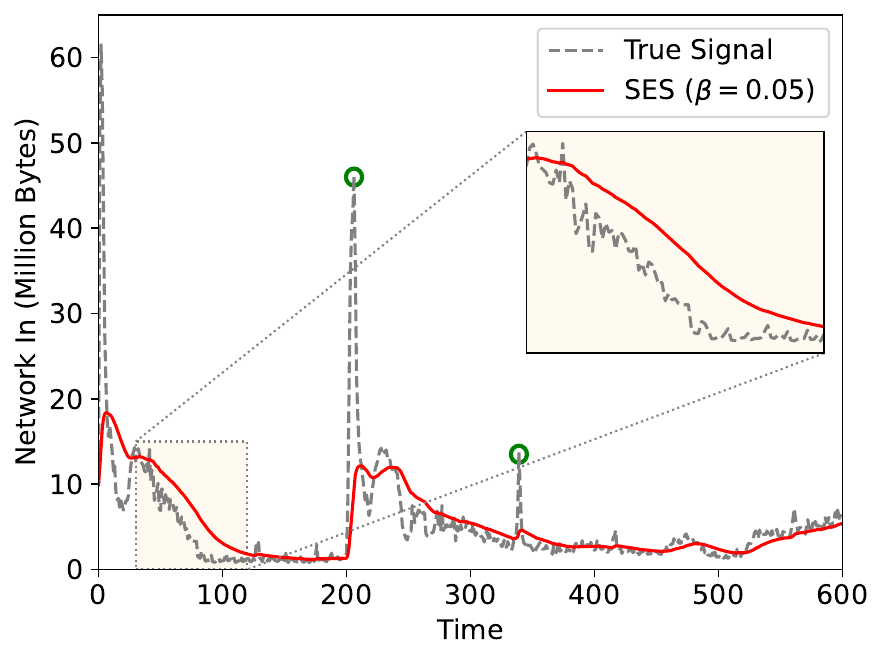}
		\includegraphics[width=0.33\linewidth]{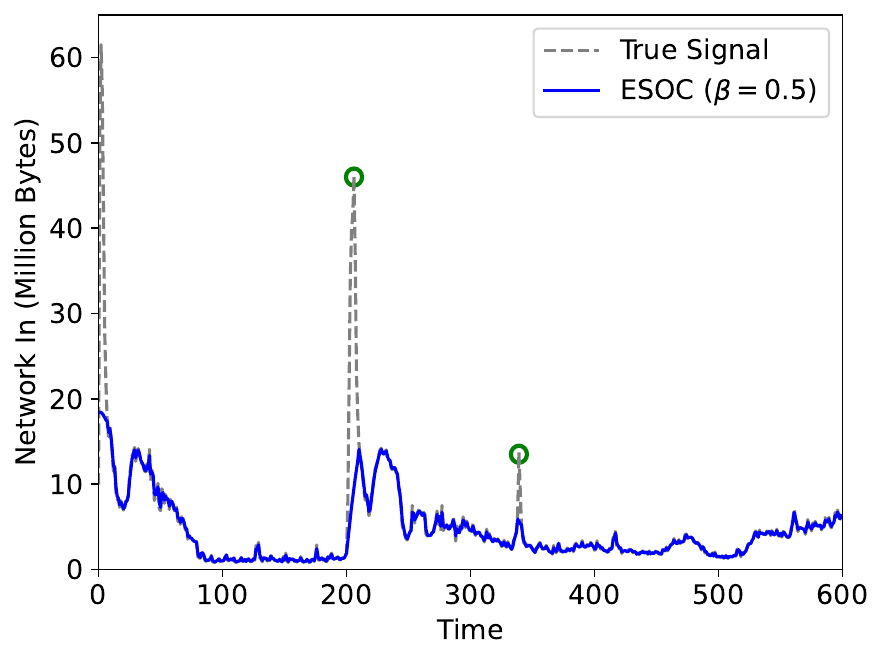}
		\includegraphics[width=0.33\linewidth]{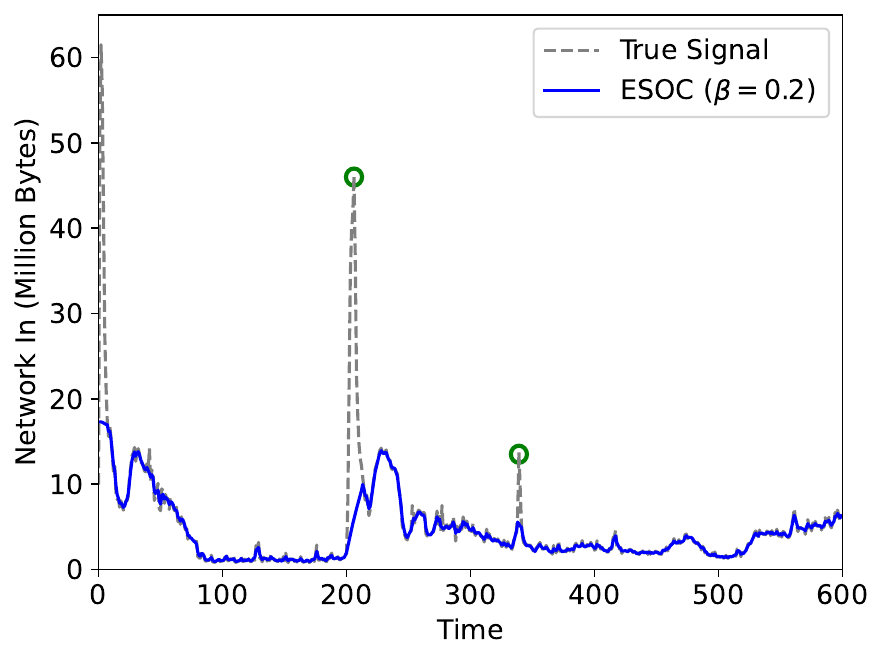}
		\includegraphics[width=0.33\linewidth]{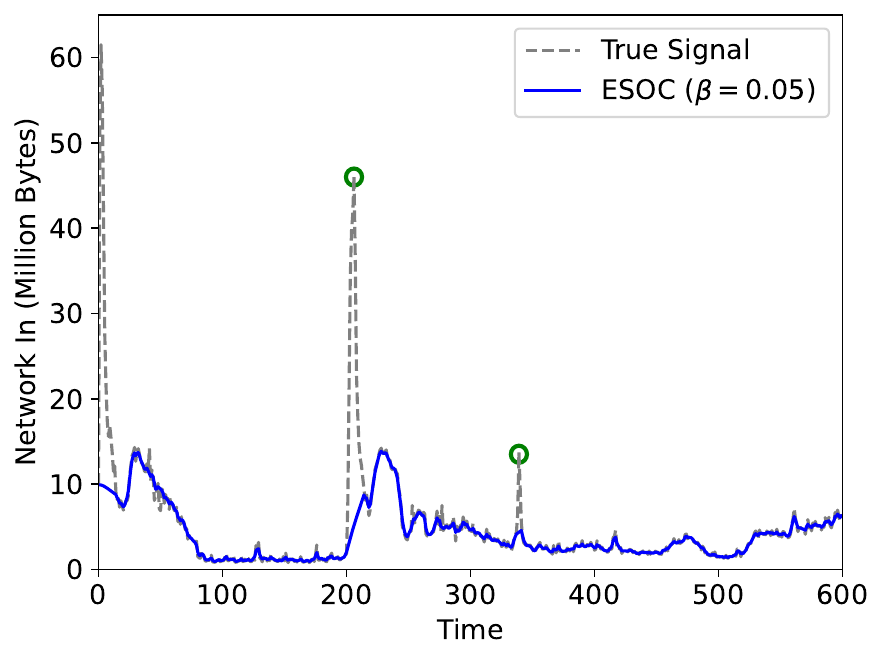}

		\caption{ \textbf{Top row:} SES applied to a signal containing anomalies from the Numenta Anomaly Benchmark (NAB), with smoothing parameters $\beta = 0.5$, $0.2$, and $0.05$ (from left to right). Outliers identified by the NAB ground-truth labels are highlighted with circles. Lower values of $\beta$ reduce sensitivity to outliers but induce a noticeable lag relative to the underlying signal. 
\textbf{Bottom row:} ESOC effectively detects and removes outliers across all values of $\beta$, while preserving the alignment with the true signal and avoiding lag.}
		\label{fig: SES failure revisited}
	\end{figure}

\subsection{Exponential smoothing with outlier correction}\label{sec: ESO}

We now introduce an MIQP extension of SES that explicitly accounts for outliers. In this framework, the vector $\vx \in \mathbb{R}^T$ denotes the smoothed signal, and $\vo \in \mathbb{R}^T$ is a sparse vector capturing outliers. The associated optimization problem is 
\begin{align*}
	\min_{\vx,\vo\in\R^T, \vz\in \{0,1\}^T}\ &\sum_{t=1}^T\left(\vy_t-\vx_t-\vo_t \right)^2+\sum_{t=1}^T\vlambda_t \vz_t\\
	\text{s.t.}\quad &\vx_{t}=\beta(\vy_t-\vo_t)+(1-\beta)\vx_{t-1}&\text{for } t=2,\ldots,T\\
    &\vo_t(1-\vz_t)=0&\text{for } t=1,\ldots,T,
\end{align*}
where $\vz \in \{0,1\}^{T}$ captures the sparsity pattern of $\vo \in \mathbb{R}^T$, and is controlled by a nonnegative regularization vector $\vlambda \in \mathbb{R}^T$. When $\vz_t = 0$, the sample $\vy_t$ is treated as noise-free, yielding the standard SES recursion. When $\vz_t = 1$, the variable $\vo_t$ is allowed to take nonzero values, thereby allowing for correction of the noisy observation $\vy_t$.

Rather than enforcing the exponential smoothing dynamic as a hard constraint, we adopt a relaxed formulation that penalizes deviations from this constraint in the objective function, as follows:
\begin{align}\label{eq: ESO}\tag{ESOC}
		\min_{\vx,\vo\in\R^T}&\ \sum_{t=1}^T\left(\vy_t-\vx_t-\vo_t \right)^2+\sum_{t=1}^T\vlambda_t \vz_t+ \mu_1 \sum_{t=2}^T\left(\beta(\vy_t-\vo_t)+(1-\beta) \vx_{t-1} - \vx_{t}\right)^2+\mu_2 \|\vo\|^2_2\\
    \text{s.t.}\quad &\vo_t(1-\vz_t)=0\qquad \qquad\text{for } t=1,\ldots,T,\nonumber
\end{align}
where $\mu_1 \ge 0$ penalizes deviations from the exponential smoothing dynamics. The additional term $\mu_2 \|\vo\|_2^2$, with $\mu_2 \ge 0$, ensures that the Hessian of the objective is positive definite. We note that although this regularization may induce slight shrinkage in $\vo$, $\mu_2$ is set to a very small value in practice, rendering the effect negligible while providing substantial numerical stability benefits. The resulting relaxed formulation conforms to the structure of Problem~\eqref{eq: MIQP}, with a Hessian matrix whose sparsity pattern has treewidth equal to~2 and exhibits linear volume growth with $\Delta_m\leq 3m$ (see Figure~\ref{fig: ESO-support-graph}). Consequently, \eqref{eq: ESO} falls within the class of problems that can be solved efficiently using the parametric algorithm developed in this paper. 

\begin{figure}[htbp]
\centering
\includegraphics[width=0.95\linewidth]{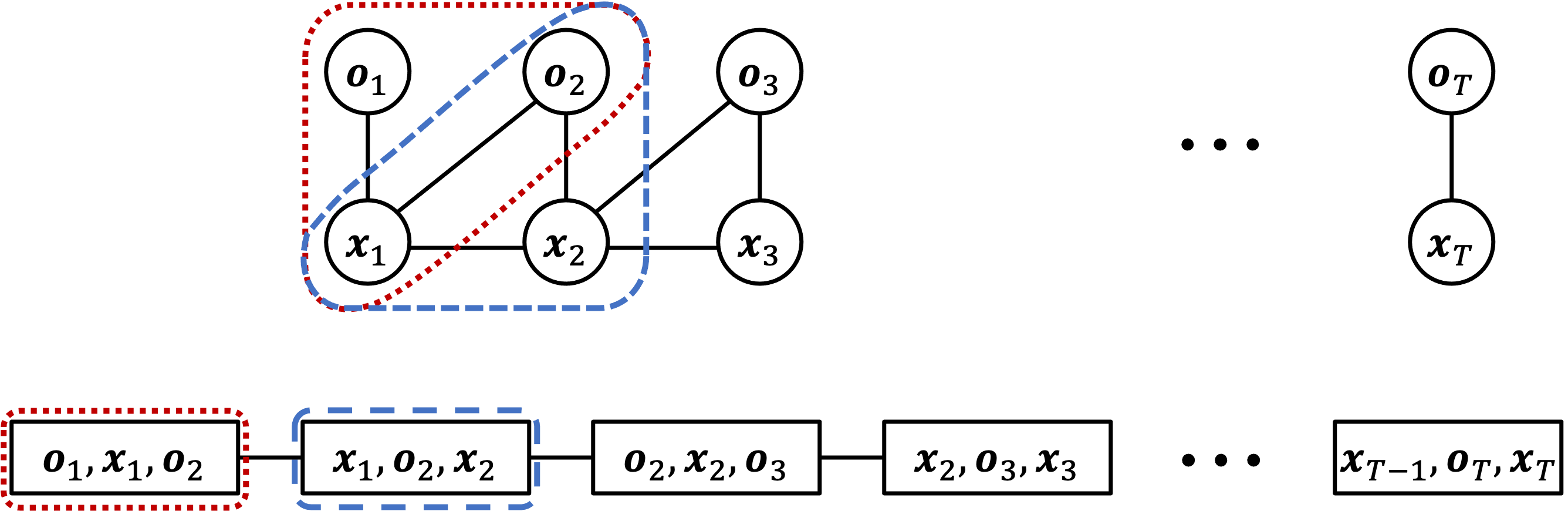}
\caption{Hessian support graph (top) and corresponding tree decomposition (bottom) for ESOC model.}
\label{fig: ESO-support-graph}
\end{figure}

    We conclude this section with preliminary empirical results; a detailed experiment study is provided in Section~\ref{sec: experiments real world}.
    Figure~\ref{fig: SES failure revisited} (second row) illustrates the solution of~\ref{eq: ESO} obtained via our parametric algorithm across different smoothing parameters $\beta$. As discussed earlier,~\ref{eq: SES} exhibits a trade-off between outlier sensitivity and responsiveness. In contrast,~\ref{eq: ESO} effectively suppresses outliers across all $\beta$ values without sacrificing responsiveness.  To quantify this comparison, we report the corresponding forecast mean squared error (MSE) values for both methods. For $\beta \in \{0.5, 0.2, 0.05\}$,~\ref{eq: SES} yields MSE values of $16.89, 9.87$, and $ 25.70$; the corresponding values under~\ref{eq: ESO} are $0.32, 0.46,$ and $0.22$. In each case,~\ref{eq: ESO} attains lower MSE, corresponding to improved predictive accuracy.

    We next evaluate the computational scalability of~\ref{eq: ESO}. Specifically, we apply our model to five real-world time series from the NAB dataset, including AWS CloudWatch metrics and traffic speed data, using $\beta \in \{0.05, 0.2, 0.5\}$. Each resulting optimization problem is solved using both \textsc{Gurobi}, a state-of-the-art commercial MIQP solver, and our proposed algorithm. To assess scalability, we vary the problem size $T$ by truncating each time series and record the corresponding solution times. For each $T$, we solve 15 instances (five signals and three $\beta$ values each). Figure~\ref{fig: intro runtime} reports these runtime comparisons: our algorithm solves all instances to optimality and achieves an average runtime of 54 seconds at $T = 1000$, whereas \textsc{Gurobi}’s runtime increases drastically with $T$ and exceeds the one-hour limit beyond $T = 500$. 

        \begin{figure}[htbp]
    	\centering
    	\includegraphics[width=0.7\linewidth]{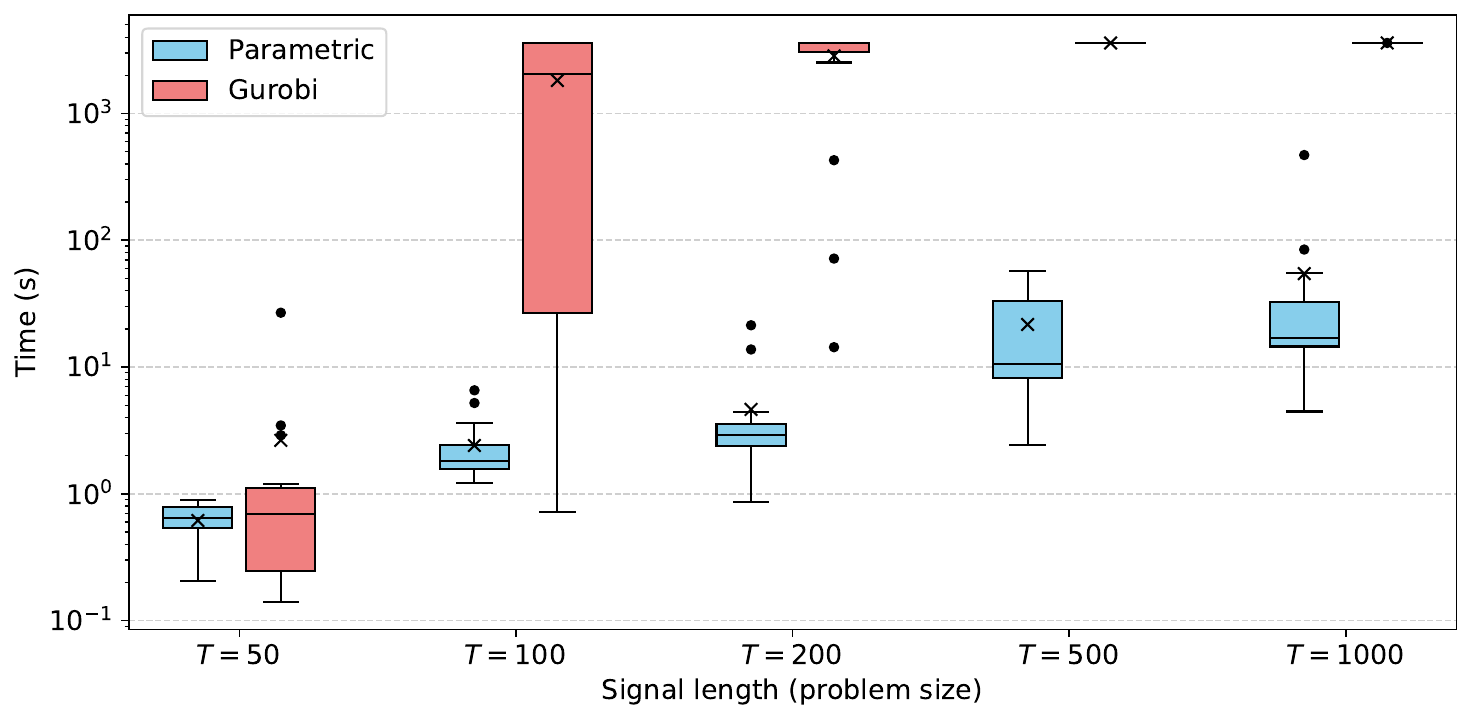}
    	\caption{Runtime comparison between the proposed algorithm and \textsc{Gurobi} on five real-world time series from the NAB dataset, with smoothing parameters. Signal lengths are varied by truncation. For each signal, smoothing parameters are set to $\beta\in\{0.05,0.2,0.5\}$. “×” denotes the mean runtime.  \textsc{Gurobi} runs are terminated after one hour, resulting in flat values at larger signal lengths.}
        \label{fig: intro runtime}
    \end{figure}
    
\section{Dynamic programming via local parametric costs}\label{sec: dp formulation}

Recall the definition of the local parametric cost $f_u$ in~\eqref{eq: f_u tree}. 
A key observation is that once the local variables $\vx_{\pi_u(\B_u)}$ are fixed to $\valpha_{\B_u}$, the subproblem defined over $\supp_u(\mQ)$ decomposes into independent components, each associated with the variables in the subtree of $\textsf{T}$ rooted at a parent of $u$. This decomposition follows directly from the \textit{running intersection property} of the tree decomposition, which states that if a node $v$ appears in two bags $\B_i$ and $\B_j$, then it must also appear in every bag along the path between $\B_i$ and $\B_j$ in $\textsf{T}$. As a result, the nodes in the subgraph $\supp_u(\mQ)$ become connected to the rest of the graph only through the nodes in the bag $\B_u$, and fixing $\vx_{\pi_u(\B_u)}$ removes all remaining couplings. 
The following lemma formalizes this observation and provides an explicit recursive representation of the local parametric cost $f_u$ in terms of the corresponding quantities associated with the parent nodes of $u$ in the tree decomposition. The proof of this lemma is provided in Appendix~\ref{app: fu recursion}.

\begin{lemma}\label{lemma: fu recursion}
    For any node $u$, the local parametric cost $f_u:\R^{\tau+1}\to\R$ satisfies
    \begin{align}\label{eq: f from g}
	f_u(\valpha_{\B_u})&=h_u(\valpha_{\B_u})+\sum_{v \in \parent_{\textsf{T}}(u)} \left(g_{v}(\valpha_{\B_v\backslash v})-\phi_{v}(\valpha_{\B_v\backslash v})\right),
\end{align}
where the functions $h_u:\R^{\tau+1}\to\R$ and $g_v,\phi_v:\R^{\tau}\to\R$ are defined as follows
\begin{subequations}
    \begin{align}
        h_u(\valpha_{\B_u})&:=\frac{1}{2}\valpha_{\B_u}^\top \mQ_{\B_u,\B_u}\valpha_{\B_u}+\vc_{\B_u}^\top \valpha_{\B_u},\label{eq::h}\\[0.4em]
        g_{v}(\valpha_{\B_v\backslash v})&:=\min\limits_{\vx_v\in\R} \left\{f_{v}(\vx_v,\valpha_{\B_v\backslash v})+\vlambda_v\bbbone(\vx_v)\right\},\label{eq: g_u}\\[0.1em]
        \phi_{v}(\valpha_{\B_v\backslash v})&:= \frac{1}{2}\valpha_{\B_v\backslash v}^\top \mQ_{\B_v\backslash v,\B_v\backslash v}\valpha_{\B_v\backslash v}+\vc^\top_{\valpha_{\B_v\backslash v}}\valpha_{\valpha_{\B_v\backslash v}}.\label{eq: phi_u}
    \end{align}
\end{subequations}
\end{lemma}

The above equation can be interpreted as the \textit{$u$-stage DP update}~\cite[Chapter~2]{bertsekas2012dynamic}, which expresses the local parametric cost (or the so-called \textit{cost-to-go}) of bag $\B_u$ in terms of the local parametric costs of its parent bags $\{f_v: v\in \parent_{\textsf{T}}(u)\}$. The intuition behind~\eqref{eq: f from g} is natural: to characterize the local parametric cost~$f_u$, three adjustments are required based on the local costs at its parent bags. First, the variables in~$\mathcal{B}_v$ that do not appear in~$\mathcal{B}_u$ must be eliminated, which is achieved through the minimization over~$\valpha_v$ in the definition of the function~$g_v$. Second, the costs associated with variables in~$\B_u$ must be added; this contribution is captured by the function~$h_u$. Finally, the cost associated with the remaining variables in~$\B_v\backslash v$ is removed through~$\phi_v$ to prevent double counting across parent bags.

This naturally leads to a DP algorithm that traverses the tree decomposition from leaves to the root and, at each bag $\B_u$, recursively computes the parametric cost $f_u$ and the function $g_u$. Since $f_u$ is defined as the minimum of convex quadratic functions, it follows that $g_u$ is also representable as a minimum of quadratic functions, though not necessarily convex. Moreover, if $f_u$ consists of $N$ quadratic pieces, then $g_u$ contains at most $2N$ quadratic pieces. This doubling arises from the indicator function $\bbbone(\vx_v)$ in the definition of $g_u$, which splits each quadratic piece of $f_u$ into two distinct quadratic functions depending on the value of $\bbbone(\vx_v)$. Moreover, the labeling scheme guarantees that each $g_u$ is computed before any parametric cost that depends on it. 
Upon computing $f_n$, the optimal cost can be obtained as 
$$f^\star=\min_{\valpha_n\in\R}\{f_n(\valpha_n)+\vlambda_n\bbbone(\valpha_n)\}.$$
With the optimal cost at hand, the remaining local parameters can be recovered by traversing from the root bag back to the leaf bags. This DP algorithm is stated in Algorithm~\ref{alg: direct algo}. The following proposition establishes the correctness and computational complexity of this DP algorithm. Note that the algorithm is exponential in the number of binary decision variables $n$, thus may be impractical as stated.

\begin{algorithm}[ht]
    \caption{DP for Problem~\eqref{eq: MIQP}}\label{alg: direct algo}
    \textbf{Input:} Problem~\eqref{eq: MIQP} and a balanced tree decomposition~$\textsf{T}$ of~$\supp(\mQ)$ with width~$\omega$;\\
\textbf{Output:} The optimal cost~$f^\star$ and an optimal solution~$(\vx^\star, \vz^\star)$ of Problem~\eqref{eq: MIQP};
    \begin{algorithmic}[1]
        \State Label the nodes of~$\supp(\mQ)$ according to the scheme described in Section~\ref{sec: tree-decomposition};\label{line::labeling_step:oracle}
        \State Set $f_1(\valpha_{\B_1})=\frac{1}{2}\valpha_{\B_1}^\top \mQ_{\B_1,\B_1}\valpha_{\B_1}+\valpha_{\B_1}^\top \vc_{{\B_1}}.$\label{line::setup-f1:oracle}
        \For{$u=1,\dots, n-1$}\label{line:for-begin:oracle}
        \State Calculate $g_{u}$ from $f_{u}$ via  Equation~\eqref{eq: g_u};\label{line:g-u:oracle}
        \State Calculate $ f_{u+1}$ via Equation~\eqref{eq: f from g};\label{line:f-u+1:oracle}
        \EndFor\label{line:for-end:oracle}
        \State Obtain $f^\star=\min\limits_{\valpha_n\in\R} \left\{f_n(\valpha_n)+\vlambda_n\bbbone(\valpha_n)\right\}$, $\vx^\star_{n}=\argmin\limits_{\valpha_n\in \R}  \left\{f_n(\valpha_n)+\vlambda_n\bbbone(\valpha_n)\right\}$, and $\vz^\star_n=\bbbone(\vx^\star_n)$; \label{line:f-star:oracle}
        \For{$u=n-1,\dots, 1$} \label{line:backward-pass-for-loop:oracle}
        \State Set $\vx^\star_{u}=\argmin\limits_{\valpha_u\in\R}\left\{f_u(\valpha_u,\vx^\star_{\B_u\backslash u})+\vlambda_u\bbbone(\valpha_u)\right\}$ and $\vz^\star_u=\bbbone(\vx^\star_u)$;
        \EndFor\label{line:backward-pass-for-loop-end:oracle}
        \State\Return $f^\star$ and $(\vx^\star, \vz^\star)$;
    \end{algorithmic}
\end{algorithm}
	
	\begin{proposition}\label{prop::dp}
		Algorithm~\ref{alg: direct algo} recovers the optimal solution to Problem~\eqref{eq: MIQP} in 
	$O(\delta\ \omega^2\  n\  2^n)$ time and $O(\omega^2\ n\ 2^n)$ memory. 
    \end{proposition}
	\begin{proof}
    We start with the correctness proof.
    \paragraph{Correctness proof.}
    The algorithm computes the local parametric cost $f_u$ for each node $u\in\{1,\ldots,n\}$. For the base case $u=1$, there are no preceding subproblems; consequently, the algorithm computes $f_1$ according to Line~\ref{line::setup-f1:oracle}, which coincides with the local parametric cost defined in Equation~\eqref{eq: f_u tree}. For $u>1$, the algorithm inductively applies Lemma~\ref{lemma: fu recursion} to correctly construct the local parametric cost function $f_u$. After constructing $f_n$, the algorithm evaluates the optimal cost $f^\star$, and an optimal solution $(\vx^\star,\vz^\star)$ is obtained by backtracking through the local parametric costs.  
        
    \paragraph{Complexity proof.}
        We analyze the runtime by examining each step of the algorithm. The labeling step in Line~\ref{line::labeling_step:oracle} requires $\mathcal{O}(n\omega^2)$ time and $\mathcal{O}(n\omega)$ memory, as shown in Section~\ref{sec: tree-decomposition}. Initializing the first local parametric cost $f_1$ requires $\mathcal{O}(\omega^2)$ time and memory (Line~\ref{line::setup-f1:oracle}).

        Next, we analyze the complexity of the first \texttt{for} loop (Lines~\ref{line:for-begin:oracle}--\ref{line:for-end:oracle}). For each $u=1,\dots,n-1$, the function $g_u$ can be obtained by separately minimizing each piece of $f_u$ with and without the indicator variable. Since $f_u$ contains at most $2^{n_u}$ pieces (Lemma~\ref{lemma: f_u as p_s}), computing $g_u$ requires $\mathcal{O}((2\omega^2)2^{n_u})=\mathcal{O}(\omega^2\  2^{n_u})$ time and memory. In Line~\ref{line:f-u+1:oracle}, the function $f_{u+1}$ is computed according to~\eqref{eq: f from g}, which we rewrite here for convenience:
        \begin{align*}
            f_{u+1}(\valpha_{\B_{u+1}})=h_{u+1}(\valpha_{\B_{u+1}})+\sum_{v\in\parent_\textsf{T}(u+1)}\Bigl(g_v(\valpha_{\B_v\backslash v})-\phi_v(\valpha_{\B_v\backslash v})\Bigr).
        \end{align*}
        The function $h_{u+1}$ and the collection $\{\phi_v:v\in\parent_{\textsf{T}}(u+1)\}$ are single-piece quadratic functions. They are computed using Equations~\eqref{eq::h} and \eqref{eq: phi_u}, respectively, and each requires $\mathcal{O}(\omega^2)$ time and memory. Moreover, the functions~$\{g_v:v\in\parent_{\textsf{T}}(u+1)\}$ have already been computed in Line~\ref{line:g-u:oracle}. Constructing a single piece of $f_{u+1}$ proceeds by selecting one piece from each parent term $g_v-\phi_v$ for all $v\in\parent_{\textsf{T}}(u+1)$. Since $|\parent_{\textsf{T}}(u+1)|\le\Delta_1\le \delta$, computing each piece incurs $\mathcal{O}(\delta\ \omega^2)$ time and $\mathcal{O}(\omega^2)$ memory. By Lemma~\ref{lemma: f_u as p_s}, $f_{u+1}$ has at most $2^{n_{u+1}}$ pieces; hence, the total cost of computing $f_{u+1}$ is $\mathcal{O}(\delta\ \omega^2 \  2^{n_{u+1}})$ time and $\mathcal{O}(\omega^2 \  2^{n_{u+1}})$ memory. 
        Since the first \texttt{for} loop runs for $n-1$ iterations, it incurs a cost of $\mathcal{O}\bigl(\sum_{u=1}^{n-1}\delta\ \omega^2\ 2^{n_{u+1}}\bigr)=\mathcal{O}(n\ \delta\ \omega^2 \  2^{n})$ time and $\mathcal{O}(n\ \omega^2 \  2^{n})$ memory. 

        Obtaining $f^\star$, $\vx^\star_n$, and $\vz^\star_n$ in Line~\ref{line:f-star:oracle} requires $\mathcal{O}(2^{n})$ time, since $f_n$ consists of at most $2^n$  pieces. Each iteration of the second \texttt{for} loop (Lines~\ref{line:backward-pass-for-loop:oracle}--\ref{line:backward-pass-for-loop-end:oracle}), requires $\mathcal{O}(\omega^2\  2^{n_u})$ time; we omit the details for brevity. Combining all these steps, we conclude that the algorithm runs in $\mathcal{O}(\delta\ \omega^2\  n\  2^n)$ time and $\mathcal{O}(\omega^2\  n\  2^n)$ memory.

	\end{proof}

\section{Pruning}\label{sec: Parametric algorithm}
The proposed DP approach for solving Problem~\eqref{eq: MIQP} can quickly become intractable, as the number of quadratic functions required to characterize the local parametric costs may grow exponentially. In this section, we aim to address this challenge. We begin by presenting the following lemma, establishing that the optimal solution to Problem~\eqref{eq: MIQP} lies within a bounded region. The proof relies on a result introduced later (Lemma~\ref{thm: decaying inv}) and is deferred to Appendix~\ref{app:proof-U}.

\begin{lemma}\label{lemma: U}
	Let $(\vx^{\star},\vz^{\star})$ be an optimal solution to Problem~\eqref{eq: MIQP}. Define $C_1:=\frac{1}{\mu_{\min}}\ \max\left\{1,\frac{(1+\sqrt{\kappa_2})^2}{2\kappa_2}\right\}$ and $\rho:= \frac{\sqrt{\kappa_2}-1}{\sqrt{\kappa_2}+1}$. We have $\|\vx^{\star}\|_{\infty}\le U$, where:
    \begin{itemize}
        \item $U=\frac{2\bw C_1}{1-\rho}\ \norm{\vc}_\infty$ if $\mQ$ is banded with bandwidth $\bw$;
        \item $U= \frac{\delta \gamma! \, C_1}{(1-\rho)^{\gamma+1}}\  \|\vc\|_\infty$ if the polynomial volume growth (Assumption~\ref{assumption: poly growth of m-deg}) is satisfied.
    \end{itemize}
\end{lemma}
An important feature of the above lemma is that it provides an element-wise $\ell_\infty$-norm bound on the optimal solution, rather than a more conventional $\ell_2$-norm bound. As will be explained, this finer, coordinate-wise control is crucial for our algorithmic development. While \citet[Theorem 2.1]{bertsimas2016best} also derive an $\ell_\infty$-norm bound, their result does not exploit the sparsity structure of $\mQ$, and consequently scales with the problem dimension. In contrast, Lemma~\ref{lemma: U} leverages this structure and avoids explicit dimensional dependence under the stated assumptions.

That said, even the bound in Lemma~\ref{lemma: U} can be conservative in practice. In many applications, sharper instance-specific bounds are readily available. For example, in the exponential smoothing model with outlier correction, one may safely set $U = \max_t |\vy_t|$, the maximum absolute magnitude of the observed signal. Such bounds are often substantially tighter. Our computational results indicate that even the general bound remains practically effective, but incorporating tighter bounds can yield substantial computational gains.

While characterizing the local parametric cost $f_u$ over the entire $\R^{\tau+1}$ may require exponentially many quadratic functions, the above lemma implies that it suffices to characterize this function only within the bounded region $\mathcal{D}=\{\vx: \norm{\vx}_\infty\leq U\}$. Recalling \eqref{eq: f_u in terms of p_u}, this implies that any quadratic function $p_{u,\bs}$ that satisfies $p_{u,\bs}(\valpha_{\B_u})>f_u(\valpha_{\B_u})$ within the region $\mathcal{D}_u=\{\valpha_{\B_u}: \norm{\valpha_{\B_u}}_\infty\leq U\}$ can be safely discarded. 
 
\begin{definition}[Relevant and irrelevant functions]
Let $f_u(\valpha_{\B_u})=\min_{s\in\{0,1\}^{n_u}} \left\{p_{u,\bs}(\valpha_{\B_u})\right\}$. A function $ p_{u,\bs}$ is called irrelevant if $p_{u,\bs}(\valpha_{\B_u})>f_u(\valpha_{\B_u})$ within the region $\mathcal{D}_u=\{\valpha_{\B_u}: \norm{\valpha_{\B_u}}_\infty\leq U\}$. Conversely, $ p_{u,\bs}$ is called relevant if $p_{u,\bs}(\valpha_{\B_u})=f_u(\valpha_{\B_u})$ for some $\valpha_{\B_u}\in \mathcal{D}_u$.
\end{definition}

\begin{figure}[htbp]
	\centering
	\includegraphics[width=0.5\linewidth]{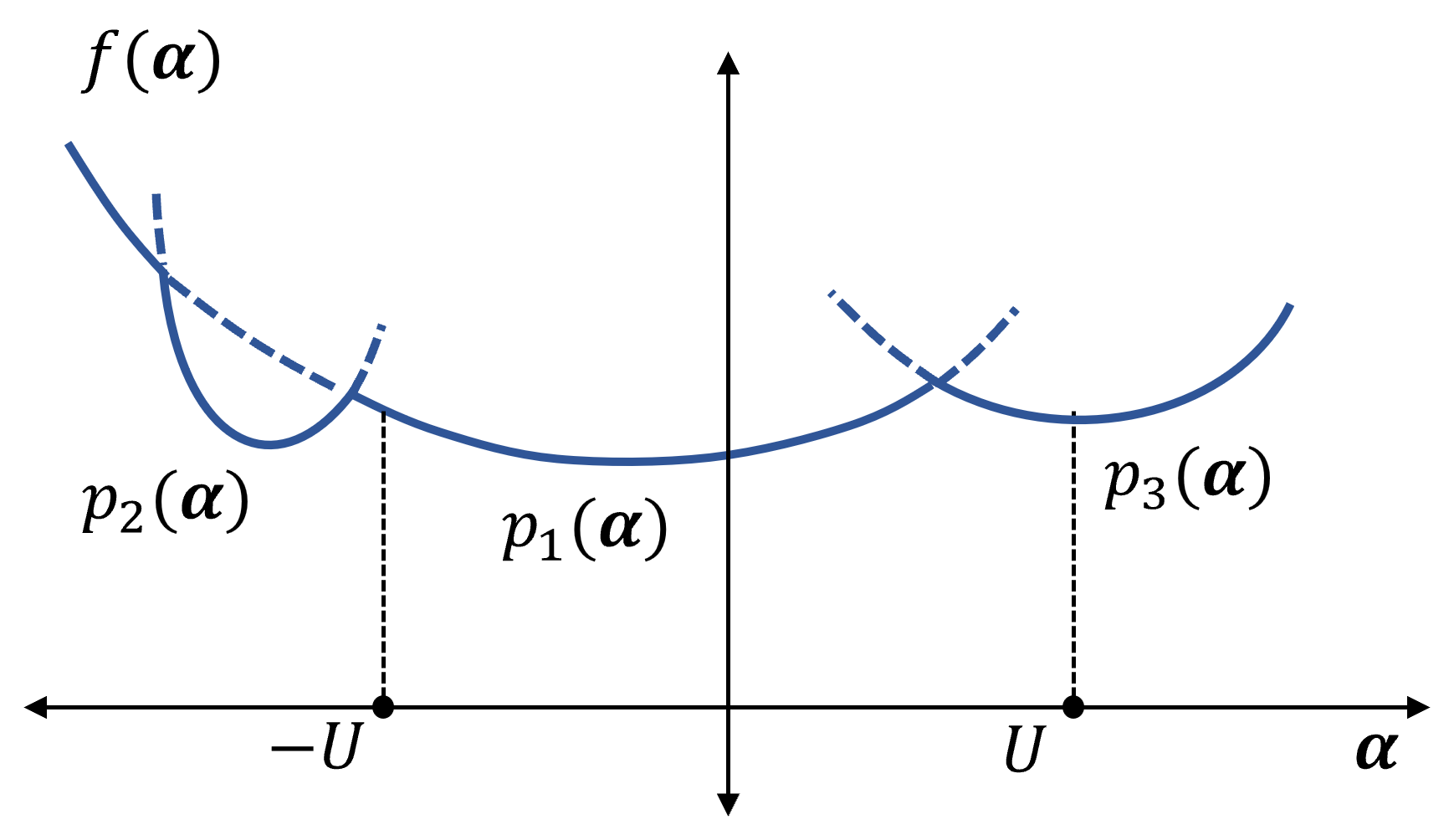}
	\caption{The piecewise quadratic function $f(\valpha)=\min\left\{ p_{1}(\valpha), p_{2}(\valpha), p_{3}(\valpha)\right\}$ is shown by the solid blue line. Here, $p_2$ is irrelevant since $ p_2(\valpha)>f(\valpha)$ for every $-U\leq \valpha\leq U$.}
	\label{fig: irrelevant1}
\end{figure}

Figure~\ref{fig: irrelevant1} illustrates an example distinguishing relevant and irrelevant functions. To separate the two, we present the following lemma, which provides a lower bound on the intersection points of two quadratic functions. The proof of this lemma is presented in Appendix~\ref{subsec:roots-proof}.

\begin{lemma}\label{lemma: roots bound}
Let $p_1, p_2: \mathbb{R}^{\tau+1}\to \mathbb{R}$ be two quadratic functions of the form $p_1(\valpha) := \tfrac{1}{2} \valpha^\top \mA_1 \valpha + \vb_1^\top \valpha + d_1$ and $p_2(\valpha) := \tfrac{1}{2} \valpha^\top \mA_2 \valpha + \vb_2^\top \valpha + d_2$. Let $\bar a\geq \frac{1}{2}\norm{\mA_1-\mA_2}_{1,1},\bar b\geq \norm{\vb_1-\vb_2}_{1}$ and $|d_1-d_2|\geq \bar d$.
	Let $\hat \valpha$ be a real root of the equation $ p_{1}(\valpha)- p_{2}(\valpha)=0$. If no real root exists, we set $\hat{\valpha} = +\infty$. Then, we have
	\begin{align}\label{eq: LB}
		\|\hat \valpha\|_\infty\ge L(p_1,p_2),\qquad\text{where\ } L(p_1,p_2):= \begin{cases}
			\frac{-\bar b+\sqrt{\bar b^2+4\bar a\bar d}}{2\bar a}&\text{if } \bar a\ne 0,\\
			\frac{\bar d}{\bar b }&\text{otherwise. }\\
		\end{cases}
	\end{align}
\end{lemma}

The above lemma provides a criterion for identifying irrelevant functions: given a pair of sparsity patterns $\bs^{(1)}, \bs^{(2)} \in \{0,1\}^{n_u}$ and their corresponding quadratic functions $p_{u,\bs^{(1)}}$ and $p_{u,\bs^{(2)}}$, one can compute $L(p_{u,\bs^{(1)}}, p_{u,\bs^{(2)}})$ using Lemma~\ref{lemma: roots bound}. 
If $L(p_{u,\bs^{(1)}}, p_{u,\bs^{(2)}}) > U$ and $p_{u,\bs^{(1)}}(\bm{0}) > p_{u,\bs^{(2)}}(\bm{0})$ (or $p_{u,\bs^{(1)}}(\bm{0} < p_{u,\bs^{(2)}}(\bm{0})$, respectively), 
then $p_{u,\bs^{(1)}}$ (or $p_{u,\bs^{(2)}}$, respectively) can be declared irrelevant.

Having established the criterion for identifying irrelevant functions, we now describe how this criterion can be incorporated into Algorithm~\ref{alg: direct algo}. For some $1 \leq u \leq n-1$, suppose that the local parametric function $f_{u+1}$ has already been computed in Line~\ref{line:f-u+1:oracle} and consists of $N$ quadratic pieces; that is, $f_{u+1}(\valpha_{\B_{u+1}}) = \min_{\bs \in \mathcal{P}_{u+1}} \{ p_{u+1,\bs}(\valpha_{\B_{u+1}}) \}$ with $|\mathcal{P}_{u+1}| = N$. To identify and discard irrelevant pieces within this set, a pruning subroutine can be inserted immediately after Line~5. This subroutine computes $L(p_{u,\bs^{(1)}}, p_{u,\bs^{(2)}})$ for every pair $\bs^{(1)}, \bs^{(2)} \in \mathcal{P}_{u+1}$ and removes any index $\bs^{(1)} \in \mathcal{P}_{u+1}$ whose corresponding function $p_{u,\bs^{(1)}}$ satisfies $L(p_{u,\bs^{(1)}}, p_{u,\bs^{(2)}}) > U$ and $p_{u,\bs^{(1)}}(\bm{0}) > p_{u,\bs^{(2)}}(\bm{0})$. The resulting refined algorithm is presented in Algorithm~\ref{alg: para algo}. The key distinction between this version and Algorithm~\ref{alg: direct algo} lies in the inclusion of the \texttt{PRUNE} subroutine (Algorithm~\ref{alg: prune}), which performs the pruning operation described above. For a function with $N$ quadratic pieces, the \texttt{PRUNE} subroutine requires $\mathcal{O}(N^2)$ pairwise comparisons between the quadratic pieces. Each comparison involves computing the corresponding $L$, which can be done in $\mathcal{O}(\omega^2)$ time and memory. Hence, the overall time and memory complexities of the \texttt{PRUNE} subroutine are $\mathcal{O}(\omega^2 N^2)$ and $\mathcal{O}(\omega^2 N)$, respectively. We note in passing that the time complexity of this subroutine can be improved to $\mathcal{O}(\omega^2 N)$ when the tree decomposition of $\mQ$ is a path. Further discussion on this special case is deferred to Section~\ref{sec: experiments synthetic} and Appendix~\ref{app: trim heuristic}.

\begin{algorithm}[htbp]
	\caption{Parametric algorithm for graphs with bounded treewidth matrices}\label{alg: para algo}
\textbf{Input:} Problem~\eqref{eq: MIQP} and a tree decomposition~$\textsf{T}$ of~$\supp(\mQ)$ with width~$\omega$ and upper bound $U$;\\
\textbf{Output:} The optimal cost~$f^\star$ and an optimal solution~$(\vx^\star, \vz^\star)$ of Problem~\eqref{eq: MIQP};
    \begin{algorithmic}[1]
        \State Label the nodes of~$\supp(\mQ)$ according to the scheme described in Section~\ref{sec: tree-decomposition};\label{line::ordering}
        \State Set $f_1(\valpha_{\B_1})=\frac{1}{2}\valpha_{\B_1}^\top \mQ_{\B_1,\B_1}\valpha_{\B_1}+\valpha_{\B_1}^\top \vc_{{\B_1}}.$\label{line::setup-f1}
        \For{$u=1,\dots, n-1$}\label{line:for-begin}
        \State Calculate $g_{u}$ from $f_{u}$ via  Equation~\eqref{eq: g_u};\label{line::compute-g_u}
        \State Calculate $ f_{u+1}$ via Equation~\eqref{eq: f from g};\label{line::compute-f-u+1}
        \State Set $f_{u+1}=\texttt{PRUNE}(f_{u+1}, U)$;\label{algstep::prune}
        \EndFor\label{line:for-end}
        \State Obtain $f^\star=\min\limits_{\valpha_n\in\R} \left\{f_n(\valpha_n)+\vlambda_n\bbbone(\valpha_n)\right\}$, $\vx^\star_{n}=\argmin\limits_{\valpha_n\in \R}  \left\{f_n(\valpha_n)+\vlambda_n\bbbone(\valpha_n)\right\}$, and $\vz^\star_n=\bbbone(\vx^\star_n)$; \label{line::x-n}
        \For{$u=n-1,\dots, 1$} \label{line::for2-begin}
        \State Set $\vx^\star_{u}=\argmin\limits_{\valpha_u\in\R}\left\{f_u(\valpha_u,\vx^\star_{\B_u\backslash u})+\vlambda_u\bbbone(\valpha_u)\right\}$ and $\vz^\star_u=\bbbone(\vx^\star_u)$;
        \EndFor\label{line::for2-end}
        \State\Return $f^\star$ and $(\vx^\star, \vz^\star)$;
	\end{algorithmic}
\end{algorithm}

\begin{algorithm}[htbp]
	\caption{$\texttt{PRUNE}(f,U)$}\label{alg: prune}
	\textbf{Input:} function $f$ characterized by its quadratic pieces $\{p_s\}_{s\in \I}$ and the constant $U$. \\
	\textbf{Output: pruned function $ f^{\trim}$} 
	\begin{algorithmic}[1]
		\State Initialize $\mathcal{S}^{\trim}$ as an empty list; 
        \State Convert $\I$ into a list and store it in $\mathcal{S}$;
        \While{$\mathcal{S}\ne \emptyset$}
        \State Set $i$ as the first element of $\mathcal{S}$ and delete $i$ from $\mathcal{S}$; 
        \State Add $i$ to $\mathcal{S}^{\trim}$; \Comment{Assume $p_i$ is relevant}
        \For{$j\in \mathcal{S}$} 
        \State Calculate $L(p_i,p_j)$;
        \If{$L > U$} \Comment{Either $p_i$ or $p_j$ is irrelevant}
            \If{$p_i(0)<p_j(0)$} \Comment{$ p_{j}$ is irrelevant}
    			\State Delete $j$ from $\mathcal{S}$;
    		\Else \Comment{$p_i$ is irrelevant}
    			\State Delete $i$ from $\mathcal{S}^\trim$;
                \State \texttt{break}; \Comment{Exit for-loop}
    		\EndIf
            \EndIf
        \EndFor
        \EndWhile
		\State\Return $f^{\trim}$ characterized by the quadratic pieces $\{p_i\}_{i\in \mathcal{S}^\trim}$; 
		\end{algorithmic}
	\end{algorithm}	

\begin{figure}[htbp]
	\centering
	\includegraphics[width=0.6\linewidth]{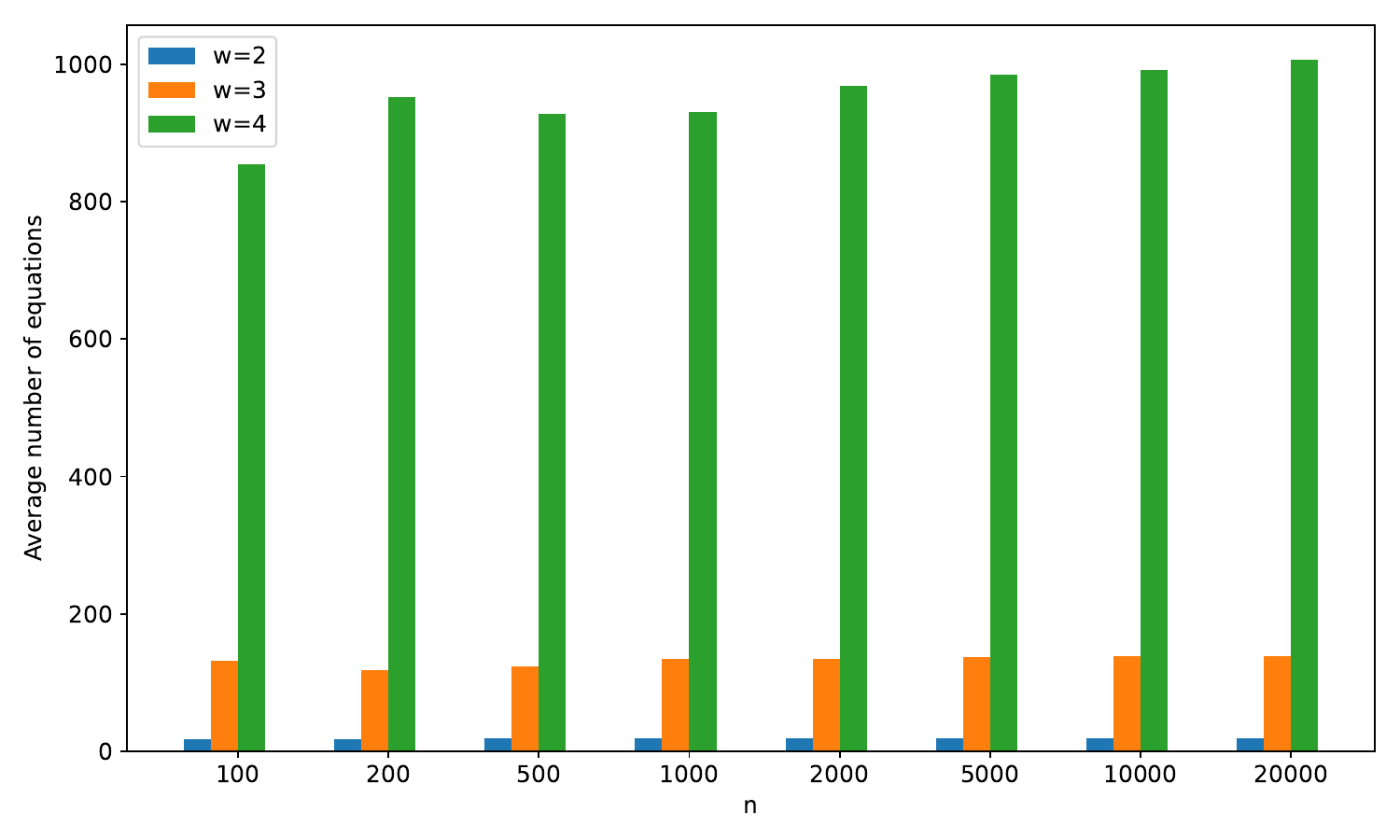}
	\caption{Number of equations computed for the pruned parametric algorithm for different values of $n$ and $\bw$. The reported results are averaged over 5 trials.}
	\label{fig: vary_n_eqs}
\end{figure} 

While the theoretical guarantees of this pruning strategy will be discussed in detail in the next section, here we briefly highlight the empirical effectiveness of this approach. Figure~\ref{fig: vary_n_eqs} illustrates the performance of the pruning subroutine on synthetically generated banded matrices of varying sizes $n$ between $100$ and 20,000, and varying bandwidths $\bw$ between $2$ and $4$ (corresponding to treewidths between $2$ and $4$). The figure shows the number of quadratic equations retained after pruning, averaged over local parametric costs and across five independent trials. Notably, even for the largest instance with $n =$ 20,000, the average number of quadratic equations after pruning does not exceed 1,100, whereas the number of quadratic equations without pruning scales as $2^{20{,}000}$. This shows the effectiveness of the pruning subroutine in eliminating irrelevant equations.

\section{Theoretical analysis}\label{sec::theoretical-analysis}
The empirical performance of the proposed parametric algorithm with pruning suggests that, among exponentially many quadratic equations, only a small subset is relevant for characterizing the local parametric cost $f_u$. This observation implies that, for most pairs of quadratic functions $(p_{u,\bs^{(1)}}, p_{u,\bs^{(2)}})$, their intersection occurs outside the relevant region $\mathcal{D}_u=\{\valpha_{\B_u}: \norm{\valpha_{\B_u}}_\infty\leq U\}$. In this section, we provide a theoretical justification for this key observation. In particular, we show that the greater the similarity between the sparsity patterns $\bs^{(1)}, \bs^{(2)} \in \{0,1\}^{n_u}$, the larger the roots of the difference $p_{u,\bs^{(1)}} - p_{u,\bs^{(2)}}$ become. We quantify the similarity of two sparsity patterns using the notion of \textit{$m$-similarity}. For any bag $\B_u$ and positive integer $m$, recall that $\mathcal{V}_{u,m} = \{ i \in \J_u \mid \dist(i,\B_u) \le m \}$ is the $m$-neighborhood of $\B_u$ within the induced subgraph $\supp_u(\mQ)$,
where $\J_u$ is the set of nodes in $\supp_u(\mQ)$ excluding those in $\B_u$. Recall that $\pi: \I\to \{1,\dots, |\I|\}$ is the canonical indexing map that assigns to each $i \in \I$ its corresponding row/column position within the submatrix $\mQ_{\I,\I}$. 

\begin{definition}[$m$-similarity]\label{def: msimilar}
The two sparsity patterns $\bs^{(1)},\bs^{(2)} \in \{0,1\}^{n_u}$ are called $m$-similar with respect to bag $\B_u$ if $\bs^{(1)}_{\pi_u(i)} = \bs^{(2)}_{\pi_u(i)}$ for all $i \in \mathcal{V}_{u,m}$. When the reference bag $\B_u$ is clear from the context, we simply say that $\bs^{(1)}$ and $\bs^{(2)}$ are {$m$-similar}.
\end{definition}

We will show that, for two $m$-similar sparsity patterns $\bs^{(1)}, \bs^{(2)} \in \{0,1\}^{n_u}$, the norm of the root(s) of the difference $p_{u,\bs^{(1)}} - p_{u,\bs^{(2)}}$ grows exponentially with $m$. To this end, we first establish a key decay property of $\mQ^{-1}$. Specifically, we show that the entries of the inverse of any principal submatrix of $\mQ$ decay exponentially with the distance between the corresponding nodes in its support graph.
\begin{lemma}\label{thm: decaying inv}
	Let $\mQ\in \mathbb{R}^{n\times n}$ be a symmetric positive definite matrix, and $\I\subseteq \{1,\dots, n\}$ any subset of its rows/columns. For any $i,j\in \I$, we have
	\begin{align}
		\left|\left[\mQ^{-1}_{\I,\I}\right]_{\pi(i),\pi(j)}\right|\le C_1\rho^{\dist(i,j)},\quad \text{where}\quad C_1:=\frac{1}{\mu_{\min}}\ \max\left\{1,\frac{(1+\sqrt{\kappa_2})^2}{2\kappa_2}\right\},\ \text{and } \rho:= \frac{\sqrt{\kappa_2}-1}{\sqrt{\kappa_2}+1}.
	\end{align}
\end{lemma}
The proof is presented in Appendix~\ref{app:exp-decay-proof}. The exponential off-diagonal decay of inverses of banded matrices has been extensively studied~\citep{demko1984decay}. The above lemma extends this result to a more general setting, in which the decay structure of the inverse is determined by the sparsity pattern of the matrix.

    Recall from Lemma~\ref{lemma: f_u as p_s} that the local parametric cost $f_u$ can be expressed as $\min_{\bs \in \{0,1\}^{n_u}} \{\, p_{u,\bs}(\valpha_{\B_u}) \,\}$, where each $p_{u,\bs}$ is given as:
    	\begin{equation}\label{eq::p_us2}
        \begin{aligned}
			&p_{u,\bs}(\valpha_{\scriptscriptstyle \B_u})= \frac{1}{2}\valpha_{\scriptscriptstyle \B_u}\mA_{u,\bs}\valpha_{\scriptscriptstyle \B_u}+\vb_{u,\bs}^\top\valpha_{\scriptscriptstyle \B_u}+d_{u,\bs},\\ &\text{where} \begin{cases}
			    \mA_{u,s}&=\mQ_{\scriptscriptstyle \B_u, \B_u}-\mQ_{\scriptscriptstyle \B_u,\J_{u,\bs}} \left(\mQ_{\scriptscriptstyle \J_{u,\bs},\J_{u,\bs}}\right)^{-1}\mQ_{\scriptscriptstyle \B_u,\J_{u,\bs}}^\top\\
		\vb_{u,s}&=\vc_{\scriptscriptstyle \B_u}-\vc_{\scriptscriptstyle \J_{u,\bs}}^T \left(\mQ_{\scriptscriptstyle \J_{u,\bs},\J_{u,\bs}}\right)^{-1}\mQ_{\scriptscriptstyle \B_u,\J_{u,\bs}}^{\top}\\
		d_{u,s}&= -\frac{1}{2}\vc_{\J_{u,\bs}}^\top (\mQ_{\scriptscriptstyle \J_{u,\bs},\J_{u,\bs}})^{-1}\vc_{\scriptscriptstyle \J_{u,\bs}}+\sum_{i\in \J_{u,\bs}} \lambda_{i}.
			\end{cases}
        \end{aligned}
	\end{equation}
    Next, we partition the elements of $\J_{u,\bs} = \{i\in \J_u\mid s_i=1\}$ according to their graph distance from the bag $\B_u$.  
    For an integer $m>0$, let  
    \begin{align}\label{eq: node decomposition}
        \setV_{u,s,m} =  \J_{u,\bs}\cap \setV_{u,m}, \qquad \setW_{u,s,m} =  \J_{u,\bs} \setminus \setV_{u,m}.
    \end{align}
   Since our subsequence analysis holds for any choice of $u\in \{1,\dots, n\}$, $\bs\in \{0,1\}^{n_u}$, and $m\in \{1,\dots, n_u\}$, for simplicity, we drop the subscripts and write $\setV = \setV_{u,\bs,m}$ and $\setW = \setW_{u,\bs,m}$. Intuitively, the set $\setV$ contains all nodes in $\J_{u,\bs}$ that are within distance $m$ from the nodes in the bag $\B_u$, while $\setW$ collects all the nodes in $\J_{u,\bs}$ whose distance from $\B_u$ exceeds $m$. As an illustrative example, consider the graph in Figure~\ref{fig: Label for trees}. When $u = 8$, $m = 2$, and $\bs = [1,1,\dots,1]^\top$, we obtain $\setV = \{2,3,\dots,7\}$ and $\setW = \{1\}$.

    Given $\setV$ and $\setW$, the matrix $\mQ_{\J_{u,\bs},\J_{u,\bs}}$ and its inverse admit the following block structures
    
	\begin{align}\label{eq: Q_s}
		&\mQ_{\J_{u,\bs},\J_{u,\bs}}=\begin{bmatrix}
			\mQ_{\setW,\setW}&\mQ_{\setW,\setV}\\
			\mQ_{\setV,\setW}&\mQ_{\setV,\setV}
		\end{bmatrix},\\
        &\mQ^{-1}_{\J_{u,\bs},\J_{u,\bs}}=\begin{bmatrix}
			(\mQ_{\J_{u,\bs},\J_{u,\bs}}/\mQ_{\setV,\setV})^{-1} & -(\mQ_{\J_{u,\bs},\J_{u,\bs}}/\mQ_{\setV,\setV})^{-1}\mQ_{\setW,\setV}\mQ_{\setV,\setV}^{-1}\\
			-\mQ_{\setV,\setV}^{-1}\mQ_{\setV,\setW}(\mQ_{\J_{u,\bs},\J_{u,\bs}})^{-1} & \mQ^{-1}_{\setV, \setV}+\mQ^{-1}_{\setV, \setV}\mQ_{\setV,\setW}(\mQ_{\J_{u,\bs},\J_{u,\bs}}/\mQ_{\setV,\setV})^{-1}\mQ_{\setW,\setV}\mQ^{-1}_{\setV,\setV}
		\end{bmatrix},\label{eq:Q_inverse}
	\end{align}
    where $\mQ_{\J_{u,\bs},\J_{u,\bs}}/\mQ_{\setV,\setV}:=\mQ_{\setW,\setW}-\mQ_{\setW,\setV}\mQ_{\setV,\setV}^{-1}\mQ_{\setV,\setW}$ is the Schur complement of block $\mQ_{\setW,\setW}$ in $\mQ_{\J_{u,\bs},\J_{u,\bs}}$.
The main motivation for analyzing the block structure of $\mQ^{-1}_{\J_{u,\bs},\,\J_{u,\bs}}$ is that, according to~\eqref{eq::p_us2}, the coefficients $\mA_{u,\bs}$ and $\vb_{u,\bs}$ contain terms involving the block product $\mQ_{\B_u,\setV} \mQ_{\setV,\setV}^{-1} \mQ_{\setV,\setW}$, whose norm, as formally stated in the next lemma, decays exponentially with $m$.

	\begin{lemma}\label{lemma: matrix norm bound}
        We have
        \begin{align}
            \left\|{\mQ_{\B_u,\setV}\mQ^{-1}_{\setV,\setV}\mQ_{\setV,\setW}}\right\|_{2}\le \frac{2\mu_{\max}^2}{\mu_{\min}}\sqrt{\Delta_1\Delta_{m}} \rho^{m-1},   
        \end{align}
        where $0<\rho<1$ is the constant from Lemma~\ref{thm: decaying inv}.
	\end{lemma}
	\begin{proof}
    Let $\pi: \setV\to \{1,\dots, |\setV|\}$ be the indexing map that assigns to each $i \in \setV$ its corresponding row/column position within the submatrix $\mQ_{\setV,\setV}$.
	Let $\I=\{i\in \setV\mid \dist(i,\B_u)=1\}$ and $\J=\{j\in \setV\mid\dist(j,\B_u)= m\}$. Here, $\I$ is the set of nodes in $\setV$ adjacent to $\B_u$, and $\J$ is the set of nodes at a distance $m$ from $\B_u$. One can write
    $$
        \mQ_{\B_u,\setV}\mQ^{-1}_{\setV,\setV}\mQ_{\setV,\setW}=\mQ_{\B_u,\pi(\I)}\left[\mQ^{-1}_{\setV,\setV}\right]_{\pi(\I),\pi(\J)}\mQ_{\pi(\J),\setW}.
    $$
    Indeed, if $\I$ or $\J$ is empty, then $\mQ_{\B_u,\setV}\mQ^{-1}_{\setV,\setV}\mQ_{\setV,\setW}=0$. Therefore, without loss of generality, we assume that neither set is empty.
    By Lemma~\ref{thm: decaying inv}, for any $i\in \I$ and $j\in \J$, we have
	\begin{align*}
		\left|\left[\mQ^{-1}_{\setV,\setV}\right]_{\pi(i),\pi(j)}\right|\le & C_1\rho^{\dist(i,j)}
        = C_1\rho^{m-1}.
	\end{align*}
	This implies that
    \begin{align*}
	\left\|\left[\mQ^{-1}_{\setV,\setV}\right]_{\pi(\mathcal{I}),\pi(\mathcal{J})} \right\|_{2} &\le \sqrt{|\mathcal{I}||\mathcal{J}|}C_1\rho^{m-1}\leq \sqrt{\Delta_1\Delta_{m}}C_1\rho^{m-1}.
	\end{align*}
The last inequality follows from the fact that $\mathcal{J} \subseteq \setV\subseteq\setV_{u,m}$, which implies $|\mathcal{J}| \le |\setV_{u,m}| \le \Delta_m$. By the same reasoning, we have $|\I|\leq \Delta_1$. 
        Therefore,
		\begin{align*}
			\left\|\mQ_{\B_u,\setV}\mQ^{-1}_{\setV,\setV}\mQ_{\setV,\setW}\right\|_{2}&= \left\|\mQ_{\B_u,\pi(\mathcal{I}) }\left[\mQ^{-1}_{\setV,\setV}\right]_{\pi(\mathcal{I}),\pi(\mathcal{J})}\mQ_{\pi(\mathcal{J}),\setW}\right\|_{2}\\
			&\le \|\mQ_{\B_u,\pi(\mathcal{I}) }\|_2\ \left\|\left[\mQ^{-1}_{\setV,\setV}\right]_{\pi(\mathcal{I}),\pi(\mathcal{J})}\right\|_2\ \|\mQ_{\pi(\mathcal{J}),\setW}\|_{2}\\
			&\le \mu_{\max}^2\ \left\|\left[\mQ^{-1}_{\setV,\setV}\right]_{\pi(\mathcal{I}),\pi(\mathcal{J})}\right\|_2\\
			&\le \mu_{\max}^2 \sqrt{\Delta_1\Delta_{m}} C_1\rho^{m-1}\\
            &= \frac{2\mu_{\max}^2}{\mu_{\min}}\sqrt{\Delta_1\Delta_{m}} \rho^{m-1}.
		\end{align*}
		The second inequality uses the property that $\mu_{\max}$ bounds the spectral norm of every submatrix of $\mQ$. The last inequality follows from the fact that $\frac{(1+\sqrt{\kappa_2})^2}{2\kappa_2}\leq 2$, which implies that $C_1=\frac{1}{\mu_{\min}}\ \max\left\{1,\frac{(1+\sqrt{\kappa_2})^2}{2\kappa_2}\right\}\leq \frac{2}{\mu_{\min}}$.
	\end{proof}\medskip
    The above lemma implies that $\left\|{\mQ_{\B_u,\setV}\mQ^{-1}_{\setV,\setV}\mQ_{\setV,\setW}}\right\|_{2}$ decays nearly exponentially with $m$, provided that $\Delta_1$ and $\Delta_m$ do not grow exponentially. As it turns out, this condition is ensured under the polynomial volume growth (Assumption~\ref{assumption: poly growth of m-deg}). This property will play an important role in establishing our final guarantees.

Armed with the above result, we can now establish that, for any two $m$-similar sparsity patterns $\bs^{(1)}, \bs^{(2)} \in \{0,1\}^{n_u}$, the quadratic and linear coefficients of their corresponding quadratic pieces $p_{u,\bs^{(1)}}$ and $p_{u,\bs^{(2)}}$ are exponentially close to each other.    
	\begin{lemma}\label{lemma: p_s norm bounds}
	Let $\bs^{(1)}, \bs^{(2)} \in \{0,1\}^{n_u}$ be two $m$-similar sparsity patterns, and let $p_{u,\bs^{(1)}}$ and $p_{u,\bs^{(2)}}$ be their corresponding quadratic pieces defined in~\eqref{eq::p_us2}. Then, the following bounds hold:
    \begin{align*}
        \|\mA_{u,\bs^{(1)}}-\mA_{u,\bs^{(2)}}\|_{1,1} &\le 4\kappa_2^4\ (\omega+1)^{3/2} \ {\Delta_1\Delta_m}\ \rho^{2m-2},\\
        \|\vb_{u,\bs^1}-\vb_{u,\bs^2}\|_1 &\le 4\kappa_2^2(1+\kappa_\infty)\ U\ (\omega+1)^{3/2}\  \sqrt{\Delta_1\Delta_m}\  \rho^{m-1}.
    \end{align*}
	\end{lemma}    

	\begin{proof}
		For each $\bs^{(i)}$ with $i\in\{1,2\}$, consider the sets $\setV_{u,\bs^{(i)},m} := \J_{u,\bs^{(i)}} \cap \setV_{u,m}$ and $\setW_{u,\bs^{(i)},m} :=  \J_{u,\bs^{(i)}}\setminus \setV_{u,m}.$
        Since $\bs^{(1)}$ and $\bs^{(2)}$ are $m$-similar, it follows that $\setV_{u,\bs^{(1)},m}=\setV_{u,\bs^{(2)},m}$. For notational convenience, we drop the subscripts and write $\setV := \setV_{u,\bs^{(1)},m}=\setV_{u,\bs^{(2)},m}$, $\setW^{(1)} := \setW_{u,\bs^{(1)},m}$, and $\setW^{(2)} := \setW_{u,\bs^{(2)},m}$. Combining the definition of $p_{u,\bs^{(1)}}$ and $p_{u,\bs^{(2)}}$ from \eqref{eq::p_us2} with the block structure of $\mQ^{-1}_{\J_{u,\bs^{(i)}},\J_{u,\bs^{(i)}}}$ in \eqref{eq:Q_inverse}, it follows that
		\begin{align*}
			\mA_{u,\bs^{(1)}}-\mA_{u,\bs^{(2)}}=&-\mQ_{\B_u,\setV}\mQ^{-1}_{\setV,\setV} \mQ_{\setV,\setW^{(1)}}(\mQ_{\J_{u,\bs^{(1)}},\J_{u,\bs^{(1)}}}/\mQ_{\setV,\setV})^{-1}\mQ_{\setW^{(1)},\setV}\mQ^{-1}_{\setV,\setV} \mQ_{\setV,\B_u}\\
			&+\mQ_{\B_u,\setV}\mQ^{-1}_{\setV,\setV} \mQ_{\setV,\setW^{(2)}}(\mQ_{\J_{u,\bs^{(2)}},\J_{u,\bs^{(2)}}}/\mQ_{\setV,\setV})^{-1}\mQ_{\setW^{(2)},\setV}\mQ^{-1}_{\setV,\setV} \mQ_{\setV,\B_u},
		\end{align*}
        where we use the fact that $\mQ_{\B_u,\J_{u,\bs^{(i)}}} = \left[\mQ_{\B_u,\setW^{(i)}}\ \ \mQ_{\B_u,\setV}\right]$ and $\mQ_{\B_u,\setW^{(i)}}=0$. The above equality yields
        \begin{align*}
            \|\mA_{u,\bs^{(1)}}-\mA_{u,\bs^{(2)}}\|_2 \le&\left\|\mQ_{\B_u,\setV}\mQ^{-1}_{\setV,\setV} \mQ_{\setV,\setW^{(1)}}(\mQ_{\J_{u,\bs^{(1)}},\J_{u,\bs^{(1)}}}/\mQ_{\setV,\setV})^{-1}\mQ_{\setW^{(1)},\setV}\mQ^{-1}_{\setV,\setV} \mQ_{\setV,\B_u}\right\|_2\\
				&+\left\|\mQ_{\B_u,\setV}\mQ^{-1}_{\setV,\setV} \mQ_{\setV,\setW^{(2)}}(\mQ_{\J_{u,\bs^{(2)}},\J_{u,\bs^{(2)}}}/\mQ_{\setV,\setV})^{-1}\mQ_{\setW^{(2)},\setV}\mQ^{-1}_{\setV,\setV} \mQ_{\setV,\B_u}\right\|_2\\
                \leq & \left\|\mQ_{\B_u,\setV}\mQ^{-1}_{\setV,\setV} \mQ_{\setV,\setW^{(1)}}\right\|^2_2\ \left\|{(\mQ_{\J_{u,\bs^{(1)}},\J_{u,\bs^{(1)}}}/\mQ_{\setV,\setV})^{-1}}\right\|_2\\
                &+\left\|\mQ_{\B_u,\setV}\mQ^{-1}_{\setV,\setV} \mQ_{\setV,\setW^{(2)}}\right\|^2_2\ \left\|{(\mQ_{\J_{u,\bs^{(2)}},\J_{u,\bs^{(2)}}}/\mQ_{\setV,\setV})^{-1}}\right\|_2.
        \end{align*}
		  From Lemma~\ref{lemma: matrix norm bound}, we have  
          \begin{align*}
              \max\left\{\left\|\mQ_{\B_u,\setV}\mQ^{-1}_{\setV,\setV} \mQ_{\setV,\setW^{(1)}}\right\|_2, \left\|\mQ_{\B_u,\setV}\mQ^{-1}_{\setV,\setV} \mQ_{\setV,\setW^{(2)}}\right\|_2\right\}\leq \frac{2\mu_{\max}^2}{\mu_{\min}}\sqrt{\Delta_1\Delta_{m}} \rho^{m-1}.
          \end{align*}
          On the other hand, 
          \begin{equation}\label{eq::bound_Schur}
          \begin{aligned}
              \frac{1}{\mu_{\min}(\mQ)}=\mu_{\max}\left(\mQ^{-1}\right) = \max_{\vzeta\not=\vzero}\left\{\frac{\vzeta^\top \mQ^{-1}\vzeta}{\vzeta^\top \vzeta}\right\}&\geq \max_{\bar \vzeta\not=\vzero}\left\{\frac{\bar \vzeta^\top \left[\mQ^{-1}\right]_{\setW^{(i)}, \setW^{(i)}}\bar \vzeta}{\bar \vzeta^\top \bar \vzeta}\right\}\\&=\left\|{(\mQ_{\J_{u,\bs^{(i)}},\J_{u,\bs^{(i)}}}/\mQ_{\setV,\setV})^{-1}}\right\|_2.
        \end{aligned}
          \end{equation}
        Combining the above inequalities, we obtain
        \begin{align*}    
        \|\mA_{u,\bs^{(1)}}-\mA_{u,\bs^{(2)}}\|_2\leq \frac{4\mu_{\max}^4}{\mu^3_{\min}}\ \Delta_1\Delta_m\ \rho^{2m-2}\leq 4\kappa_2^4\ \Delta_1\Delta_m\ \rho^{2m-2},
        \end{align*}
        where the last inequality follows from the fact that $\mu_{\min}\leq 1$, and hence, $\mu_{\max}\leq \kappa_2$.
		 The proof of the first statement is completed after noting that $\|\mA_{\bs^1}-\mA_{\bs^2}\|_{1,1} \le (\omega+1)^{3/2}\|\mA_{u,s^{(1)}}-\mA_{u,s^{(2)}}\|_2.$
         
Next, we provide the proof of the second statement. Again, from \eqref{eq::p_us2} and the block structure of $\mQ^{-1}_{\J_{u,\bs^{(i)}},\J_{u,\bs^{(i)}}}$ in \eqref{eq:Q_inverse}, we have
        \begin{align*}
        \vb_{u,\bs^{(1)}}-\vb_{u,\bs^{(2)}}
        =& 
           -\big(\vc_{\setW^{(1)}}-\mQ_{\setW^{(1)},\setV}\mQ^{-1}_{\setV,\setV}\vc_\setV\big)^\top \left(\mQ_{\J_{u,\bs^{(1)}},\J_{u,\bs^{(1)}}}/\mQ_{\setV,\setV}\right)^{-1} \mQ_{\setW^{(1)},\setV}\mQ^{-1}_{\setV,\setV} \mQ_{\setV,\B_u}\\
           &+\big(\vc_{\setW^{(2)}}-\mQ_{\setW^{(2)},\setV}\mQ^{-1}_{\setV,\setV}\vc_\setV\big)^\top \left(\mQ_{\J_{u,\bs^{(2)}},\J_{u,\bs^{(2)}}}/\mQ_{\setV,\setV}\right)^{-1} \mQ_{\setW^{(2)},\setV}\mQ^{-1}_{\setV,\setV} \mQ_{\setV,\B_u},
        \end{align*}
        which implies
		\begin{align*}
				&\|\vb_{u,\bs^{(1)}}-\vb_{u,\bs^{(2)}}\|_{\infty}\\			
                &\leq \norm*{\vc_{\setW^{(1)}}-\mQ_{\setW^{(1)},\setV}\mQ^{-1}_{\setV,\setV}\vc_\setV}_\infty\ 
				\norm*{\left(\mQ_{\J_{u,\bs^{(1)}},\J_{u,\bs^{(1)}}}/\mQ_{\setV,\setV}\right)^{-1}}_{\infty}\ 
				 \norm*{  \mQ_{\setW^{(1)},\setV}\mQ^{-1}_{\setV,\setV} \mQ_{\setV,\B_u}}_\infty\\
                 &\ \ +\norm*{\vc_{\setW^{(2)}}-\mQ_{\setW^{(2)},\setV}\mQ^{-1}_{\setV,\setV}\vc_\setV}_\infty\ 
				\norm*{\left(\mQ_{\J_{u,\bs^{(2)}},\J_{u,\bs^{(2)}}}/\mQ_{\setV,\setV}\right)^{-1}}_{\infty}\ 
				 \norm*{  \mQ_{\setW^{(2)},\setV}\mQ^{-1}_{\setV,\setV}\mQ_{\setV,\B_u}}_\infty.
		\end{align*}
        We now control each term on the right-hand side separately.
		For any $i=1,2$, we have
		$$\norm*{\vc_{\setW^{(i)}}-\mQ_{\setW^{(i)},\setV}\mQ^{-1}_{\setV,\setV}\vc_\setV}_\infty\le \|\vc\|_{\infty}+\|\mQ\|_{\infty}\left\|\mQ^{-1}_{\setV,\setV}\right\|_\infty\left\|\vc_\setV\right\|_{\infty}\le\|\vc\|_{\infty}+U\ \|\mQ\|_{\infty}.$$
        The first inequality follows from the sub-multiplicative property of the induced $\infty$-norm and the fact that $\left\|\mQ_{\setW^{(i)},\setV}\right\|_\infty\leq \left\|\mQ\right\|_\infty$. The second inequality follows from $\left\|\mQ^{-1}_{\setV,\setV}\right\|_\infty\left\|\vc_\setV\right\|_{\infty}\leq U$ (see the proof of Lemma~\ref{lemma: U}).
        Similarly, from \eqref{eq:Q_inverse}, we obtain
		  $$\norm*{\left(\mQ_{\J_{u,\bs^{(i)}},\J_{u,\bs^{(i)}}}/\mQ_{\setV,\setV}\right)^{-1}}_{\infty}= \left\|\left[\mQ^{-1}\right]_{\setW^{(i)}, \setW^{(i)}}\right\|_\infty\leq \norm*{\mQ^{-1} }_{\infty}.$$
		Finally, 
		$$ \norm*{ \mQ_{\setW^{(i)},\setV}\mQ^{-1}_{\setV,\setV} \mQ_{\setV,\B_u}}_\infty\le \sqrt{\omega+1}\left\|  \mQ_{\setW^{(i)},\setV}\mQ^{-1}_{\setV,\setV} \mQ_{\setV,\B_u}\right\|_2\le  \sqrt{\omega+1}\  \frac{2\mu_{\max}^2}{\mu_{\min}}\sqrt{\Delta_1\Delta_{m}} \rho^{m-1},$$
		where in the last inequality, we use Lemma~\ref{lemma: matrix norm bound}. Combining these bounds, we obtain
        \begin{align*}
            \|\vb_{u,\bs^{(1)}}-\vb_{u,\bs^{(2)}}\|_{\infty}&\le 2\left(\|\vc\|_{\infty}+U\ \|\mQ\|_{\infty}\right)\  \norm*{\mQ^{-1} }_{\infty}\  \sqrt{\omega+1}\ \frac{2\mu_{\max}^2}{\mu_{\min}}\sqrt{\Delta_1\Delta_{m}} \rho^{m-1}\\
            &\leq 4\kappa_2^2(1+\kappa_\infty)U\  \sqrt{\omega+1}\  \sqrt{\Delta_1\Delta_{m}} \rho^{m-1},
        \end{align*}
        where the second inequality follows from $\norm*{\mQ^{-1} }_{\infty}\norm*{\vc}_{\infty}\leq U$, $\kappa_\infty = \norm*{\mQ}_{\infty}\norm*{\mQ^{-1} }_{\infty}$, and $\frac{\mu_{\max}^2}{\mu_{\min}}\leq \kappa_2^2$.
        The proof is completed after noting that $\|\vb_{u,\bs^{(1)}}-\vb_{u,\bs^{(2)}}\|_{1}\leq (\omega+1)\|\vb_{u,\bs^{(1)}}-\vb_{u,\bs^{(2)}}\|_{\infty}.$
	\end{proof}\medskip
Our next objective is to analyze the term $d_{u,\bs^{(1)}} - d_{u,\bs^{(2)}}$ for a pair of $m$-similar sparsity patterns $\bs^{(1)}$ and $\bs^{(2)}$. From the expression of $d_{u,\bs}$ in~\eqref{eq::p_us2}, one observes that, unlike the quadratic and linear coefficients $\mA_{u,\bs}$ and $\vb_{u,\bs}$, the constant term $d_{u,\bs}$ \textit{does not} depend on any submatrix of $\mQ$ associated with the bag $\B_u$. Consequently, contrary to the quadratic and linear terms, Lemma~\ref{lemma: matrix norm bound} cannot be used to establish a decaying behavior for $|d_{u,\bs^{(1)}} - d_{u,\bs^{(2)}}|$.
This observation is precisely what enables our pruning strategy to be effective: as established in Lemma~\ref{lemma: p_s norm bounds}, both $\|\mA_{u,\bs^{(1)}} - \mA_{u,\bs^{(2)}}\|_{1,1}$ and $\|\vb_{u,\bs^{(1)}} - \vb_{u,\bs^{(2)}}\|_1$ decay exponentially fast in $m$ (under polynomial growth condition in Assumption~\ref{assumption: poly growth of m-deg}), whereas $|d_{u,\bs^{(1)}} - d_{u,\bs^{(2)}}|$ does not exhibit such decay. Therefore, by invoking Lemma~\ref{lemma: roots bound}, we conclude that for any pair of $m$-similar sparsity patterns $\bs^{(1)}$ and $\bs^{(2)}$, the quantity $L(p_{u,\bs^{(1)}}, p_{u,\bs^{(2)}})$ grows exponentially with $m$. As a result, at least one of the functions $p_{u,\bs^{(1)}}$ or $p_{u,\bs^{(2)}}$ becomes irrelevant for sufficiently large $m$. Our next lemma formalizes this intuition.

\begin{lemma}\label{lemma::irrelevant-iden}
Let $\bs^{(1)}, \bs^{(2)} \in \{0,1\}^{n_u}$ be two $m$-similar sparsity patterns, and let $p_{u,\bs^{(1)}}$ and $p_{u,\bs^{(2)}}$ denote their corresponding quadratic pieces defined in~\eqref{eq::p_us2}. Assume that the polynomial volume growth condition (Assumption~\ref{assumption: poly growth of m-deg}) holds, and that $|d_{u,\bs^{(1)}} - d_{u,\bs^{(2)}}| \ge \eta$, for some $0<\eta\leq 1$. Suppose
\begin{align}\label{eq::m-lb}
    m \ge \max\left\{ \frac{2\log\left(8 U^2/\eta\right)+2\log\left(\delta(\omega+1)^{3/2}\right)+2\log\left(\kappa_2^2(1+\kappa_\infty)\right)}{\log(1/\rho)}+2,\, \frac{2\gamma}{\log(1/\rho)}, \left(\frac{\gamma}{\log(1/\rho)}\right)^2 \right\}.
\end{align}
Then, $L(\bs^{(1)}, \bs^{(2)}) \ge U$; in particular, at least one of $\bs^{(1)}$ or $\bs^{(2)}$ is irrelevant.
\end{lemma}
Before proving the above lemma, we first present the following auxiliary claim, which will play an important role in its proof.
\begin{claim}\label{claim}
            Suppose $\Gamma>0$. Then, $\Gamma m\geq \log(m)$ for any $m\geq \max\{1,2/\Gamma,(1/\Gamma)^2\}$.
        \end{claim}
        \begin{proof}
            First, suppose that $\Gamma\geq 1$. Then, it is easy to verify that $\Gamma m\geq m\geq \log(m)$ for $m>0$. Now, consider the setting where $0<\Gamma<1$. Define $m = \xi/\Gamma$, where $\xi\geq \max\{2,1/\Gamma\}$. Then, we have
\begin{align*}
    \Gamma m-\log(m)=\xi-\log(\xi)-\log(1/\Gamma)\geq \log(\xi)-\log(1/\Gamma)\geq 0,
\end{align*}
where the first inequality follows from the fact that $\xi\geq 2\log(\xi)$ for $\xi\geq 2$, and the last inequality follows from the fact that $\xi\geq \max\{2,1/\Gamma\}\geq 1/\Gamma$. 
        \end{proof}\medskip
\begin{proof}{Proof of Lemma~\ref{lemma::irrelevant-iden}.}
Under the polynomial growth condition, we have $\Delta_1\leq \delta$ and $\Delta_m\leq \delta m^\gamma$. Combined with Lemma~\ref{lemma: p_s norm bounds}, this implies that
    \begin{align*}
        \|\mA_{u,\bs^{(1)}}-\mA_{u,\bs^{(2)}}\|_{1,1} &\le 4\kappa_2^4\ (\omega+1)^{3/2} \ \delta^2m^\gamma\ \rho^{2m-2}\\
        \|\vb_{u,\bs^1}-\vb_{u,\bs^2}\|_1 &\le 4\kappa_2^2(1+\kappa_\infty)\ U\ (\omega+1)^{3/2}\  \delta m^{\gamma/2}\  \rho^{m-1}
    \end{align*}
		Invoking Lemma~\ref{lemma: roots bound} with $\bar a = 4\kappa_2^4\ (\omega+1)^{3/2} \ \delta^2m^\gamma\ \rho^{2m-2}>0$, $\bar b = 4\kappa_2^2(1+\kappa_\infty)\ U\ (\omega+1)^{3/2}\  \delta m^{\gamma/2}\  \rho^{m-1}$, and $\bar d=\eta$, we obtain
		\begin{align*}
			L(\bs^{(1)}, \bs^{(2)})&=\frac{-\bar b+\sqrt{\bar{b}^2+4\bar{a}\bar{d}}}{2\bar{a}}=\frac{4\bar{a}\bar{d}}{2\bar{a}\left(\bar b+\sqrt{\bar{b}^2+4\bar{a}\bar{d}}\right)}\geq \frac{\bar{d}}{\sqrt{\bar{b}^2+4\bar{a}\bar{d}}}\geq \frac{\bar{d}}{\bar{b}+2\sqrt{\bar{a}\bar{d}}}.
		\end{align*}
        Substituting the expressions for $\bar a$, $\bar b$, and $\bar d$, we arrive at
        \begin{align*}
            \frac{\bar{d}}{\bar{b}+2\sqrt{\bar{a}\bar{d}}}&\geq\frac{\eta}{4\kappa_2^2(1+\kappa_\infty)\ U\ (\omega+1)^{3/2}\  \delta m^{\gamma/2}\  \rho^{m-1}+4\sqrt{\eta}\ \kappa_2^2\  (\omega+1)^{3/4}\  \delta m^{\gamma/2}\  \rho^{m-1}}\\
            &=\frac{\eta}{4\delta(\omega+1)^{3/2} \left(\kappa_2^2(1+\kappa_\infty)U+\sqrt{\eta}\ \kappa_2^2\  (\omega+1)^{-3/4}\right)}\  m^{-\gamma/2}\  \rho^{-m+1}\\
            &\geq K m^{-\gamma/2}\rho^{-m+1}, \quad \text{where}\quad K:=\frac{\eta}{8U\  \delta(\omega+1)^{3/2}\  \kappa_2^2(1+\kappa_\infty)}.
        \end{align*}
        In the last inequality, we use the facts that $0<\eta\leq 1$ and $(\omega+1)^{-3/4}\leq 1$, which implies that $\sqrt{\eta}(\omega+1)^{-3/4}\leq 1$. Moreover, we use the fact that $\kappa_2^2\leq \kappa_2^2(1+\kappa_\infty)U$.
        To complete the proof, it suffices to establish that $Km^{-\gamma/2}\rho^{- m+1}\geq U$. 
        To this end, we first note that
		\begin{align*}
			Km^{-\gamma/2}\rho^{-m+1}\ge U&\Longleftrightarrow
			\log \left(m^{-\gamma/2}\rho^{-m+1}\right)\ge\log \left(U/K\right)\\
			&\Longleftrightarrow -\log(m)+ (m- 1) \frac{2\log\left(1/\rho\right)}{\gamma}\ge \frac{2\log \left(U/K\right)}{\gamma}\\
			&\Longleftrightarrow -\log(m)+ m\frac{2\log \left(1/\rho\right)}{\gamma}\ge \frac{2\log \left(U/K\right)}{\gamma}+\frac{2\log(1/\rho)}{\gamma}.
		\end{align*}
        Let $\Gamma:=\frac{\log\left( 1/\rho \right)}{\gamma}$. By our assumption, $m\geq \max\{1,2/\Gamma, (1/\Gamma)^2\}$. Therefore, Claim \ref{claim} implies that $\Gamma m\geq \log(m)$, which leads to
		\begin{align*}
			 -\log(m)+ m\frac{2\log \left(1/\rho\right)}{\gamma}\ge 2\frac{\log \left(U/K\right)}{\gamma}+2\frac{\log(1/\rho)}{\gamma}\Longleftarrow&\   m\frac{\log \left(1/\rho\right)}{\gamma}\ge 2\frac{\log \left(U/K\right)}{\gamma}+2\frac{\log(1/\rho)}{\gamma}\\
			\Longleftarrow&\   m\ge \frac{2\log \left(U/K\right)}{\log(1/\rho)}+2.
		\end{align*}
        Given the definition of $K$, one can verify that the last inequality is implied by \eqref{eq::m-lb}.
		
    \end{proof}\medskip
To leverage the above lemma and obtain a global bound on the maximum number of quadratic pieces retained for each local parametric cost after pruning, we introduce the notion of the $(k,\eta)$-margin. Throughout our subsequent arguments, we fix $m$ to be any quantity satisfying~\eqref{eq::m-lb}. For any $u \in \{1,\dots,n\}$ and any $\vxi \in \{0,1\}^{|\setV_{u,m}|}$, define $\C_{u,\vxi} := \{ \bs \in \{0,1\}^{n_u} : \bs_{\setV_{u,m}} = \vxi \}$ as the set of all $m$-similar sparsity patterns whose entries within the $m$-neighborhood of bag $\B_u$ in $\supp_u(\mQ)$ coincide with $\vxi$. Equivalently, $\C_{u,\vxi}$ is the equivalence class induced by the relation $\bs \sim \bs'$ if and only if $\bs_{\setV_{u,m}} = \bs'_{\setV_{u,m}}$. Recall that $|\setV_{u,m}| \leq \Delta_{m}$; hence, there are at most $2^{\Delta_{m}}$ such equivalence classes.
\begin{definition}[$\eta$-optimal sets and $(k,\eta)$-margin]
    Given any $\eta \geq 0$, $u \in \{1,\dots,n\}$, and $\vxi \in \{0,1\}^{|\setV_{u,m}|}$, the $\eta$-optimal set of the equivalence class $\C_{u,\vxi}$ is defined as
    \begin{align*} 
    \mathcal{R}_{u,\vxi,\eta} := \left\{\bs\in \C_{{u,\vxi}}: d_{u,\bs}-\min_{\vz\in \C_{{u,\vxi}}} \{d_{u,\vz}\}\leq \eta\right\}.
    \end{align*}
    We say that the problem has $(k,\eta)$-margin if 
    $\left| \mathcal{R}_{u,\vxi,\eta} \right| \leq k$
    holds for every $u \in \{1,\dots,n\}$ and $\vxi \in \{0,1\}^{|\setV_{u,m-1}|}$.
\end{definition}
To build intuition behind the notions of $\eta$-optimal sets and $(k,\eta)$-margin, one can show that, from its definition \eqref{eq::p_us2}, $d_{u,\bs}$ coincides with the optimal value of the following quadratic program
\begin{align*}
    d_{u,\bs}=\min_{\vx \in \R^{|\J_{u,\bs}|}}
    \left\{
        \tfrac{1}{2} \vx^\top \mQ_{\J_{u,\bs},\J_{u,\bs}} \vx
        + \vc_{\J_{u,\bs}}^\top \vx
        + \sum_{i \in \J_{u,\bs}} \vlambda_i
    \right\}.
\end{align*}
This value is obtained from the subproblem~\eqref{eq: f_u tree} defined on $\supp_u(\mQ)$ after setting $\vx_{\B_u}=\bm{0}$ and fixing the binary vector $\vz \in \{0,1\}^{n_u}$ according to the sparsity pattern $\bs$, namely
$\vz_i = 1$ for $i \in \J_{u,\bs}$ and $\vz_i = 0$ for $i \notin \J_{u,\bs}$.
Consequently, $\min_{\vz \in \C_{u,\vxi}} \{d_{u,\vz}\}$ represents the optimal value among all such restricted subproblems whose sparsity patterns belong to the equivalence class $\C_{u,\vxi}$. The associated $\eta$-optimal set $\mathcal{R}_{u,\vxi,\eta}$ then consists of those sparsity patterns whose corresponding subproblem values lie within $\eta$ of this minimum. In other words, $\mathcal{R}_{u,\vxi,\eta}$ captures all binary assignments in the equivalence class $\C_{u,\vxi}$ that are nearly optimal (with an optimality gap of $\eta$) for the local subproblem at node $u$. For instance, if $\lvert \mathcal{R}_{u,\vxi,\eta} \rvert = 1$, then the corresponding subproblem admits a unique optimal sparsity pattern from $\C_{u,\vxi}$ with an $\eta$-margin. This observation motivates the definition of the $(k,\eta)$-margin: the problem is said to have $(k,\eta)$-margin if, for every bag $\B_u$ and every equivalence class $\C_{u,\vxi}$, the associated $\eta$-optimal set has cardinality at most $k$; that is, there are at most $k$ solutions whose corresponding optimal values lie within an $\eta$-margin of the minimum.

This assumption serves as a structural ``margin'' condition that controls the multiplicity of nearly indistinguishable discrete solutions. Conceptually, this is analogous to margin conditions in statistical learning---such as the classical Tsybakov low-noise assumption~\citep{tsybakov2004optimal} and its refinements~\citep{diakonikolas2024near}---which limit the mass of points lying near the decision boundary. Our framework adapts this idea by imposing a bound on the size of the $\eta$-optimal set. 

Indeed, the effectiveness of the proposed pruning strategy critically depends on the parameters of the $(k,\eta)$\nobreakdash-margin. According to Lemma~\ref{lemma::irrelevant-iden}, the pruning strategy guarantees the identification of irrelevant pieces once they fall outside the $\eta$-margin. Consequently, within each equivalence class $\C_{u,\vxi}$, at most $k$ equations remain relevant after pruning. Since there are at most $2^{|\setV_{u,m}|}$ equivalence classes, the total number of equations retained after the pruning step is at most $k\  2^{|\setV_{u,m}|}$. This is formalized in the following lemma. 
\begin{lemma}\label{lemma::pruned-set}
    Suppose that Problem~\eqref{eq: MIQP} has $(k,\eta)$-margin and that the polynomial growth condition in Assumption~\ref{assumption: poly growth of m-deg} holds. For every $u \in \{1,\dots,n-1\}$, let $\mathcal{P}_{u+1}$ denote the pruned index set obtained after the pruning step (Line~\ref{algstep::prune} of Algorithm~\ref{alg: para algo}). Then, $\mathcal{P}_{u+1}$ satisfies $|\mathcal{P}_{u+1}| \le k\  2^{|\setV_{u+1,m}|}$, where $m$ is the smallest integer satisfying~\eqref{eq::m-lb}.
\end{lemma}
    \begin{proof}
Choose any $\vxi \in \{0,1\}^{|\setV_{u+1,m}|}$ and consider the corresponding equivalence class $\C_{u+1,\vxi}$, whose elements are $m$-similar by definition. 
    By Lemma~\ref{lemma::irrelevant-iden}, for any two sparsity patterns $\bs^{(1)},\bs^{(2)} \in \C_{u+1,\vxi}$ with $|d_{u+1,\bs^{(1)}} - d_{u+1,\bs^{(2)}}| \ge \eta$, at least one of the two patterns is irrelevant and is therefore removed by the pruning step. Consequently, after pruning, the only sparsity patterns in $\C_{u+1,\vxi}$ retained are those whose constant term $d_{u+1,\bs}$ lies within $\eta$ of the minimum over the class, namely those in
    $$
        \setR_{u+1,\vxi,\eta}
        = \Bigl\{ \bs \in \C_{u+1,\vxi} :
        d_{u+1,\bs} - \min_{\vz \in \C_{u+1,\vxi}} d_{u+1,\vz} \le \eta \Bigr\}.
    $$
    By the $(k,\eta)$-margin assumption, we have $|\setR_{u+1,\vxi,\eta}| \le k$ for every $\vxi$, and thus at most $k$ sparsity patterns remain in each equivalence class after pruning.
    Each $\vxi \in \{0,1\}^{|\setV_{u+1,m}|}$ defines a unique equivalence class $\C_{u+1,\vxi}$, and hence there are $2^{|\setV_{u+1,m}|}$ such classes in total. Summing over all classes yields $|\mathcal{P}_{u+1}| \le k \  2^{|\setV_{u+1,m}|},$ which proves the statement.
	\end{proof}\medskip
Equipped with the above lemma, we are now ready to provide our main theorem on the correctness and runtime of Algorithm \ref{alg: para algo}.
\begin{theorem}\label{thm:main}
    Suppose that Problem~\eqref{eq: MIQP} has $(k,\eta)$-margin and that the polynomial growth condition in Assumption~\ref{assumption: poly growth of m-deg} holds. Then, the proposed parametric algorithm (Algorithm \ref{alg: para algo}) solves Problem~\eqref{eq: MIQP} in $\mathcal{O}(n\omega^2\  \delta \  k^{2\delta}\  4^{\Delta_{m+1} })$ time and $\mathcal{O}(n\omega^2\  k^\delta\  2^{\Delta_{m+1}})$ memory, where
    \begin{align*}
    m\! =\!\! \left\lceil \! \max\!\left\{ \frac{2\log\left(8U^2/\eta\right)+2\log\left(\delta(\omega+1)^{3/2}\right)+2\log\left(\kappa_2^2(1+\kappa_\infty)\right)}{\log(1/\rho)}+2,\, \frac{2\gamma}{\log(1/\rho)}, \left(\frac{\gamma}{\log(1/\rho)}\right)^2\! \right\}\!\right\rceil\!.
\end{align*}
\end{theorem}
\begin{proof}
    We first prove the correctness of the algorithm, then analyze its time and memory complexities.

		\paragraph{Correctness proof.}
        Let $(\vx^\star,\vz^\star)$ be an optimal solution of Problem~\eqref{eq: MIQP}, and let $f^\star$ be its objective value. To establish correctness, it suffices to show that $\vz^\star_{\J_n}$ belongs to the pruned set $\mathcal{P}_n$ of the local parametric cost $f_n$. 
        The proof proceeds by contradiction. 
        Suppose that $\vz^\star_{\J_n}\notin\mathcal{P}_{n}$. Then there exists an index $u\le n$ at which the sparsity pattern of the optimal solution restricted to nodes in $\J_u$, namely $\vz^\star_{\J_u}$, is discarded by the pruning step. Let $u$ denote the smallest index for which $\vz^\star_{\J_u}\notin\mathcal{P}_{u}$. Since $\vz^\star_{\J_u}$ is discarded by the pruning step, the corresponding piece $p_{u,\vz^\star_{\J_u}}$ is irrelevant by definition of the pruning procedure. Therefore, there exists $\bar \bs\in\mathcal{P}_u$ such that $p_{u,\vz^\star_{\J_u}}(\valpha)>p_{u,\bar \bs}(\valpha)$ for all $\norm*{\valpha}_{\infty}<U$. 
        
        To arrive at a contradiction, we evaluate the objective of Problem~\eqref{eq: MIQP} at the optimal solution $(\vx^\star,\vz^\star)$. Before doing so, we partition the nodes of $\supp(\mQ)$ into three disjoint sets: $\setA=\{1,\dots,n\}\backslash(\J_u\cup\B_u)$, $\B_u$, and $\J_u$. From the properties of a tree decomposition, there is no edge between nodes in $\setA$ and $\J_u$ in $\supp(\mQ)$; hence $\mQ_{\setA,\J_u}=\mQ_{\J_u,\setA}^\top=\vzero$.
        With respect to this partition, $\mQ$ and $\vc$ take the block form:
    	\begin{align*}
    			\mQ= \begin{bmatrix}
    				\mQ_{\setA,\setA}&\mQ_{\setA,\B_u}& \vzero\\
    				\mQ_{\B_u,\setA}&\mQ_{\B_u,\B_u}& \mQ_{\B_u,\J_u}\\
    				\vzero	&\mQ_{\J_u,\B_u}&\mQ_{\J_u,\J_u}
    				\end{bmatrix},\quad \vc=\begin{bmatrix}
    					\vc_\setA\\
    					\vc_{\B_u}\\
    					\vc_{\J_u}
    		\end{bmatrix}.
    	\end{align*}
        
        Evaluating the objective function of Problem~\eqref{eq: MIQP} at the optimal solution yields: 
        \begin{align*}
		f^{\star}&\!=\frac{1}{2} (\vx^{\star})^\top \mQ \vx^{\star}+\vc^\top \vx^{\star}+\vlambda^\top \vz^\star\\
		&\!=\frac{1}{2} (\vx^{\star}_{\setA})^\top \mQ_{\setA,\setA} \vx^{\star}_{\setA}+\vc_{\setA}^\top \vx^{\star}_{\setA}+\sum_{i\in \setA}\lambda_i\vz^{\star}_i+\vx^{\star}_{\setA}Q_{\setA,\B_u}\vx^{\star}_{\B_u}+\sum_{i\in \B_u}\lambda_i\vz^\star_i+p_{u,\vz^\star_{\J_u}}(\vx^{\star}_{\B_u})\\
        &\!>\frac{1}{2} (\vx^{\star}_{\setA})^\top \mQ_{\setA,\setA} \vx^{\star}_{\setA}+\vc_{\setA}^\top \vx^{\star}_{\setA}+\sum_{i\in \setA}\lambda_i\vz^{\star}_i+\vx^{\star}_{\setA}Q_{\setA,\B_u}\vx^{\star}_{\B_u}+\sum_{i\in \B_u}\lambda_i\vz^\star_i+p_{u,\bar \bs}(\vx^{\star}_{\B_u}).
	\end{align*}
    This shows that the vector $\bar \vz = \left[\vz^\star_{\setA}\ \vz^\star_{\B_u} \ \bar\bs\right]^\top$ is a strictly better choice for the binary variables, thereby contradicting the optimality of $(\vx^\star,\vz^\star)$. 
\paragraph{Complexity proof.}
To analyze the runtime, we examine each step of the algorithm. 
The topological ordering and the subsequent labeling can be carried out in $\mathcal{O}(n)$ and $\mathcal{O}(n\omega^2)$ time, respectively, and in $\mathcal{O}(n\omega)$ memory (see Section~\ref{sec: tree-decomposition}). Initializing $f_1$ requires $\mathcal{O}(\omega^2)$ time and memory (Line~\ref{line::setup-f1}).

Next, we analyze the complexity of the first \texttt{for} loop (Lines~\ref{line:for-begin}--\ref{line:for-end}). For each $u=1,\dots,n-1$, the function $g_u$ can be obtained by separately minimizing each piece of $f_u$ with and without the indicator variable, with a total cost of $\mathcal{O}(|\mathcal{P}_{u}|\omega^2)$ time and memory. On the other hand, by Lemma~\ref{lemma::pruned-set}, we have $|\mathcal{P}_u|\leq k\  2^{|\setV_{u,m}|}$. This implies that the cost of computing $g_u$ can be bounded by $\mathcal{O}(\omega^2\  k\  2^{|\setV_{u,m}|}) = \mathcal{O}(\omega^2\  k\  2^{\Delta_{m}})$.
Moreover, given $\{g_v\}_{v\in \parent_T(u+1)}$, the function $f_{u+1}$ in Line~\ref{line::compute-f-u+1} can be computed according to~\eqref{eq: f from g}, which we rewrite here:
\begin{align*}
    f_{u+1}(\valpha_{\B_{u+1}})
    &= h_{u+1}(\valpha_{\B_{u+1}})
    + \sum_{v\in\parent_T(u+1)} \Bigl(g_v(\valpha_{\B_{u+1}\setminus v})-\phi_v(\valpha_{\B_{u+1}\backslash v})\Bigr).
\end{align*}
As established in the runtime analysis of Algorithm~\ref{alg: direct algo}, $h_{u+1}$ and $\{\phi_v:v\in\parent_{\textsf{T}}(u+1)\}$ are single-piece quadratic functions and can be computed in $\mathcal{O}(\omega^2)$ time and memory.

 Next, we turn to the computation of $g_v$. Since the number of pieces of $g_v$ is at most twice that of the corresponding $f_v$, each $g_v$ contains at most $2k \  2^{|\setV_{v,m}|}$ pieces. Consequently, $f_{u+1}$ can be constructed by considering all combinations of the pieces of $g_v-\phi_v$ for $v \in \parent_T(u+1)$, leading to at most
\begin{align*} 
\prod_{v\in \parent_T({u+1})} 2k\  2^{|\setV_{v,m}|} = k^{|\parent_{T}(u+1)|}\  2^{\sum_{v\in \parent_T(u+1)}(|\setV_{v,m}|+1)} 
\end{align*}
pieces.
Indeed, we have $|\parent_T(u+1)|\le \Delta_1\le \delta$. Moreover, by definition of $\setV_{u+1,m+1}$, we have $\setV_{u+1,m+1} = \bigcup_{v\in\parent_{\textsf{T}}(u+1)}(\setV_{v,m}\cup\{v\})$. Since the sets $\{\setV_{v,m}\cup\{v\}\}_{v\in \parent_{\textsf{T}}(u+1)}$ are disjoint, it follows that $|\setV_{u+1,m+1}| = \sum_{v\in \parent_{\textsf{T}}(u+1)} (|\setV_{v,m}|+1)$.
This implies that $f_{u+1}$, before pruning, can have at most $k^\delta\  2^{|\setV_{u+1,m+1}|}$ pieces, and can be formed in $\mathcal{O}(\delta\ \omega^2\  k^\delta\  2^{|\setV_{u+1,m+1}|})$ time and $\mathcal{O}(\omega^2\  k^\delta\  2^{|\setV_{u+1,m+1}|})$ memory.
Finally, Line~\ref{algstep::prune} invokes the pruning subroutine \texttt{PRUNE}, which runs in $\mathcal{O}\left(\omega^2\  \left(k^\delta\  2^{|\setV_{u+1,m+1}|}\right)^2\right) = \mathcal{O}\left(\omega^2\  k^{2\delta}\  4^{\Delta_{m+1}}\right)$ time and $\mathcal{O}\left(\omega^2\  k^\delta\  2^{|\setV_{u+1,m+1}|}\right)= \mathcal{O}\left(\omega^2\  k^{\delta}\  2^{\Delta_{m+1}}\right)$ memory, as discussed in Section~\ref{sec: Parametric algorithm}. Since the first \texttt{for} loop runs for $n-1$ iterations, it incurs a total cost of $\mathcal{O}(n\omega^2\  \delta \  k^{2\delta}\  4^{\Delta_{m+1} })$ time and $\mathcal{O}(n\omega^2\  k^\delta\  2^{\Delta_{m+1}})$ memory.

Finally, obtaining $f^\star$, $\vx^\star_n$, and $\vz^\star_n$ in Line~\ref{line::x-n}, and each iteration of the second \texttt{for} loop (Lines~\ref{line::for2-begin}--\ref{line::for2-end}), can be carried out in $\mathcal{O}(\omega^2\  2^{\Delta_{m}})$ time and memory, and we omit the details for brevity.
Combining all these steps, we conclude that the algorithm runs in $\mathcal{O}(n\omega^2\ \delta\  k^{2\delta}\  4^{\Delta_{m+1}})$ time and $\mathcal{O}(n\omega^2\  k^\delta\  2^{\Delta_{m+1}})$ memory.
\end{proof}\medskip

To provide further insight into Theorem~\ref{thm:main}, we focus on special classes of problems whose sparsity graph $\supp(\mQ)$ exhibits \textit{linear volume growth}, i.e., Assumption~\ref{assumption: poly growth of m-deg} holds with $\gamma=1$, as our complexity bound takes a more crisp form in this setting. 
\begin{corollary}\label{cor::banded}
    Suppose that Problem~\eqref{eq: MIQP} has $(k,\eta)$-margin and satisfies linear volume growth (Assumption~\ref{assumption: poly growth of m-deg} with $\gamma=1$). Then, the proposed parametric algorithm (Algorithm~\ref{alg: para algo}) solves Problem~\eqref{eq: MIQP} in
    \begin{align*}
        \mathcal{O}\!\left(
            n \omega^2 \ \delta\  k^{2\delta} \  
            \max\!\left\{
                \left(
                    \frac{U^2}{\eta} \  
                    \delta(\omega+1)^{3/2} \  
                    \kappa_2^2(1+\kappa_\infty)
                \right)^{\!\frac{4\delta}{\log(1/\rho)}},
                \;4^{\delta\theta}
            \right\}
        \right),
    \end{align*}
    time, where $\theta := \max\!\left\{\tfrac{2}{\log(1/\rho)},\, \tfrac{1}{\log^2(1/\rho)}\right\}$.
\end{corollary}

\begin{proof}
    The result follows directly from Theorem~\ref{thm:main} after setting $\gamma=1$. 
\end{proof}

The significance of the above corollary lies in the fact that, for fixed treewidth $\omega$, margin parameters $(k,\eta)$, condition numbers $\kappa_2,\kappa_\infty$, and the volume growth parameters $(\delta, \gamma)$, the runtime scales \textit{linearly} with $n$. In the next section, we examine the extent to which this theoretical dependence is reflected in practice. Here, we highlight this result in the context of two important special cases for which existing guarantees are available: trees with a bounded number of leaves and banded matrices with bandwidth~$\bw$. The former has treewidth equal to~$1$, while the latter has treewidth at most~$\bw$.

\begin{itemize}
    \item \textbf{Tree-structured problems.}  
    When $\supp(\mQ)$ is a tree, \cite{bhathena2025parametric} shows that a specialized version of the parametric algorithm solves Problem~\eqref{eq: MIQP} in $\mathcal{O}(n^2)$ time and memory. Although the theoretical bound is quadratic in $n$, the empirical performance reported therein is nearly linear. The above corollary partially explains this phenomenon: for trees with a bounded number of leaves, the theoretical complexity is in fact linear in $n$ under the mild $(k,\eta)$-margin condition. This observation aligns the theoretical guarantees with empirical performance and clarifies why tree structures are particularly favorable for our method.

    \item \textbf{Banded matrices.}  
    Suppose that $\mQ$ is banded with bandwidth $\bw$. For the special case $\bw=2$ (corresponding to the tridiagonal structure), \cite{liu2023graph} proved that Problem~\eqref{eq: MIQP} can be solved to optimality in $\mathcal{O}(n^2)$ time. More recently, \citet{gomez2024real} developed an FPTAS for Problem~\eqref{eq: MIQP} with banded $\mQ$ with bandwidth $\bw\geq 2$, computing $\epsilon$-accurate solutions in time polynomial in $n$ and in $\|\vc\|_\infty/\epsilon$, provided both $\delta$ and $\kappa_2$ are fixed. Under the aforementioned margin assumption, our result yields an exact algorithm and reduces the dependence on $n$ for both problem classes.
\end{itemize}
    
\section{Numerical experiments}\label{sec: experiments synthetic}
Next, we evaluate the performance of our algorithm across a range of synthetic and realistic instances. In Subsection~\ref{subsec::methods}, we describe our experimental setup and implementation details. In Subsection~\ref{subsec::synthetic}, we examine the dependence of the algorithm on various problem parameters using synthetic data. Finally, in Subsection~\ref{sec: experiments real world}, we demonstrate the effectiveness of our approach on the exponential smoothing problem with outlier correction using real-world data.

\subsection{Methods and settings}\label{subsec::methods}
All experiments were conducted on a computer with 8-core 3.0 GHz Xeon Gold 6154 processors and 16 GB of memory.
We compare the performance of the proposed parametric algorithm against \textsc{Gurobi}~v10.0.2. For all \textsc{Gurobi} runs, we impose a time limit of one hour and terminate the solver once the optimality gap falls below $0.01\%$. If \textsc{Gurobi} does not attain an optimality gap of $0.01\%$ or smaller within this limit, we report the best optimality gap achieved by the solver. All results reported in the tables and figures are averaged across five independent trials. 
The Python implementation of our algorithm, along with the code used for the case studies presented in this paper, is available at:
\begin{center}
    \href{https://github.com/aareshfb/Treewidth-Parametric-Algorithm}{https://github.com/aareshfb/Treewidth-Parametric-Algorithm}.
\end{center}

\begin{itemize}
    \item \textbf{\textsc{Gurobi}} We reformulate Problem~\eqref{eq: MIQP} as
    \begin{subequations}\label{eq: MIQP_gurobi}
    \begin{align}
        \min_{\vx \in \R^n,\;\vz \in \{0,1\}^n}\quad & 
            \frac{1}{2}\vx^\top \mQ \vx + \vc^\top \vx + \vlambda^\top \vz,
            \label{eq: MIQP_obj_gurobi}\\[2mm]
        \text{s.t.}\quad 
            & -U \vz_i \le \vx_i \le U \vz_i, 
            \qquad i = 1,2,\ldots,n.
            \label{eq: MIQP_const_gurobi}
    \end{align}
    \end{subequations}
    where $U$ is the same upper bound also used in our parametric algorithm. We note that the problem also admits a perspective reformulation \citep{frangioni2006perspective}. We experimented with explicitly incorporating this reformulation within \textsc{Gurobi}; however, we found that the solver’s default formulation consistently outperformed our manually implemented perspective reformulation. We conjecture that this is because \textsc{Gurobi} internally exploits perspective-based strengthening more effectively. Consequently, all numerical results reported in this paper are based on \textsc{Gurobi}’s built-in configuration.
\item {\bf Parametric method}\quad We compare the performance of \textsc{Gurobi} with that of our proposed parametric algorithm (Algorithm~\ref{alg: para algo}). We note that the pruning subroutine (Algorithm~\ref{alg: prune}) has a runtime of $\mathcal{O}(\omega^2 N^2)$, as it exhaustively compares all pairs of the $N$ quadratic pieces that constitute a local parametric cost. In the special case where the support graph of $\mQ$ admits a tree decomposition that is a path (e.g., banded matrices), we can leverage an enhanced pruning procedure that processes the pieces in a single pass, requiring only $\mathcal{O}(N)$ comparisons. This improvement is possible when the quadratic functions are stored such that every consecutive pair is $m$-similar, which can be ensured at no additional cost when the tree decomposition has a path structure. The details of this implementation are deferred to Appendix~\ref{app: trim heuristic}.
\end{itemize}

\subsection{Results for synthetically generated instances} \label{subsec::synthetic}
Recall that, according to Theorem~\ref{thm:main}, the runtime of the parametric algorithm scales linearly with the problem size, polynomially with the margin parameters $(k,\eta)$, the volume growth parameter~$\gamma$, and the conditioning parameters $\kappa_2$ and~$\kappa_\infty$, and exponentially with the volume growth parameter~$\delta$. In this subsection, we empirically study these dependencies across a wide range of synthetic instances.

The Hessian $\mQ$ is constructed as $\mQ = \mY^\top \mY + \nu I,$
where $\mY$ is an upper triangular matrix with entries sampled uniformly from $[-1,1]$ within the band $\bw$ and zero otherwise (i.e., $\mY_{ij}\sim \mathcal{U}(-1,1)$ for $i \le j \le i+\bw$ and zero otherwise). Simple calculation reveals that $\supp(\mQ)$ is banded with bandwidth $\bw$. The parameter $\nu > 0$ serves to regulate the condition number of $\mQ$. The vector $\vc$ is generated with independent entries uniformly distributed on $[-10,10]$. Unless indicated otherwise, the regularization parameter $\vlambda$ is generated with independent entries uniformly distributed from the interval $(3.5,4.5)$, which typically results in solutions with roughly $50\%$ nonzero entries. In some experiments, we vary $\nu$ (to change the condition number) and $\vlambda$ (to change the sparsity level), while keeping the construction of $\mQ$ and $\vc$ fixed as described above.

\paragraph{\bf Effect of problem size.} In our first experiment, we examine the performance of the parametric algorithm on problems with varying $n$. Table~\ref{tab: vary n} reports the runtime for bandwidths $\bw=2$ and $\bw=4$. For instances exceeding $n=200$, \textsc{Gurobi} is unable to solve the instance to optimality within the time limit of one hour (indicated by TL). In contrast, the parametric algorithm can solve much larger instances within a few seconds. 
To verify correctness, we report the optimal objective values of both methods. It can be seen that the value obtained by the parametric algorithm is at least as good as that obtained by \textsc{Gurobi} (and in some cases strictly better). In addition, for the parametric algorithm, we report the average number of quadratic pieces retained after pruning. It can be seen that this number increases only modestly with $n$, showcasing the effectiveness of the proposed pruning approach in eliminating the irrelevant pieces.

\begin{table}[htbp]
	\begin{center}\small
		\caption{Performance comparison for varying sizes}
        \label{tab: vary n}
		\begin{tabular}{ c|c|c|c|c|c|c|c }

			\hline
			$\pmb{\bw}$&\textbf{Metric} & \textbf{Method} & $\pmb{n=100}$ & $\pmb{n=200}$&  $\pmb{n=500}$& $\pmb{n=1000}$&$\pmb{n=2000}$\\ 
			\hline
            \multirow{8}{*}{2}
& Condition no. $(\kappa_2)$ & --- 
& $\kappa_2 \approx 7.10$ 
& $\kappa_2 \approx 7.43$ 
& $\kappa_2 \approx 7.43$ 
& $\kappa_2 \approx 8.11$ 
& $\kappa_2 \approx 8.02$ \\
\cline{2-8}
& \multirow{2}{*}{Time(s)} 
& Parametric 
& 0.17 & 0.35 & 0.91 & 1.82 & 3.70 \\
&& \textsc{Gurobi} 
&  11.01 & 565.43  & TL & TL & TL \\
\cline{2-8}
& \multirow{2}{*}{Objective value}
& Parametric
& -588.89 & -1143.57 & -3033.48 & -6363.05 & -12367.60 \\
&& \textsc{Gurobi}
& -588.89 & -1143.57 & -3033.28 & -6362.32 & -12365.20 \\
\cline{2-8}
& Avg no. eqs
& Parametric
& 23 & 23 & 25 & 25 & 25 \\
\cline{2-8}
& B\&B nodes
& \multirow{2}{*}{Gurobi}
& 31326 & 2963376 & 4882224 & 2938436 & 1169022 \\
& Opt. gap
&& 0.00\% & 0.00\% & 1.11\% & 1.30\% & 1.50\% \\
\hline
\hline
\multirow{8}{*}{4}&Condition no. $(\kappa_2)$&---& $\kappa_2 \approx 6.18$& $\kappa_2 \approx 6.20$& $\kappa_2 \approx 6.18$& $\kappa_2 \approx 6.43$& $\kappa_2 \approx 6.79$\\
\cline{2-8}
&\multirow{2}{*}{Time(s)} &Parametric &7.82  &17.11	& 42.45 & 83.88 & 176.98\\ 
&& \textsc{Gurobi}&9.20  &1634.54& TL	& TL	& TL\\ 
\cline{2-8}
&\multirow{2}{*}{Objective value}&
  Parametric&-277.17&-547.57&-1383.69&-2886.91&-5837.25\\
  &&\textsc{Gurobi}	&-277.17&-547.57&-1383.44&-2885.99&-5833.97\\
\cline{2-8}
&Avg no. eqs&Parametric&995&1103&1082&1092&1139\\
\cline{2-8}
&	B\&B nodes&\multirow{2}{*}{Gurobi}&13265&3162950&3669416&2583014&1017048\\
&	Opt. gap && 0.00\%&0.41\%& 1.82\%& 2.32\%&2.51\%\\

\hline
		\end{tabular}
	\end{center}
\end{table}

We next evaluate the runtime of the parametric algorithm over a broader range of problem sizes and bandwidths. Figure~\ref{fig: vary_n}~(left) reports performance for $\bw = 2,3,4,5$ with $n$ ranging up to $20{,}000$. In these experiments, we tune $\nu$ such that $\kappa_2 \approx 6.50$. We observe that the parametric algorithm solves the largest instances with $\bw = 4$ in under one hour, whereas for $\bw=5$, it solves instances with $n$ up to $2{,}000$ within one hour. The runtime curves exhibit nearly identical slopes across all values of $\bw$, indicating that the bandwidth influences the overall runtime scale but not its linear dependence on~$n$, which is consistent with Corollary~\ref{cor::banded}.

\paragraph{\bf Effect of the bandwidth.}

We next isolate the effect of the bandwidth $\bw$ on the runtime of the parametric algorithm. Figure~\ref{fig: vary_n}~(right) reports the runtime as a function of $\bw$ for problem sizes $n= 50,200$ and condition number $\kappa_2\approx 4.63$, with $\bw$ ranging from $2$ to $6$. This log-linear plot shows that runtime scales exponentially with $\bw$, consistent with Corollary~\ref{cor::banded}.

\begin{figure}[htbp]
	\centering
	\includegraphics[width=0.5\linewidth]{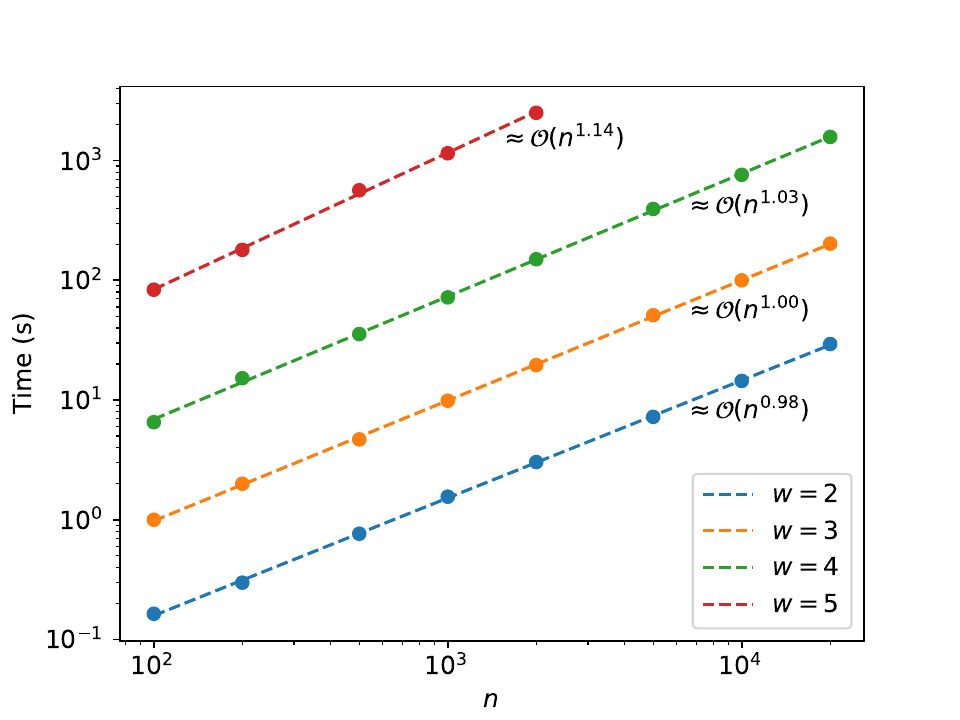}\includegraphics[width=0.5\linewidth]{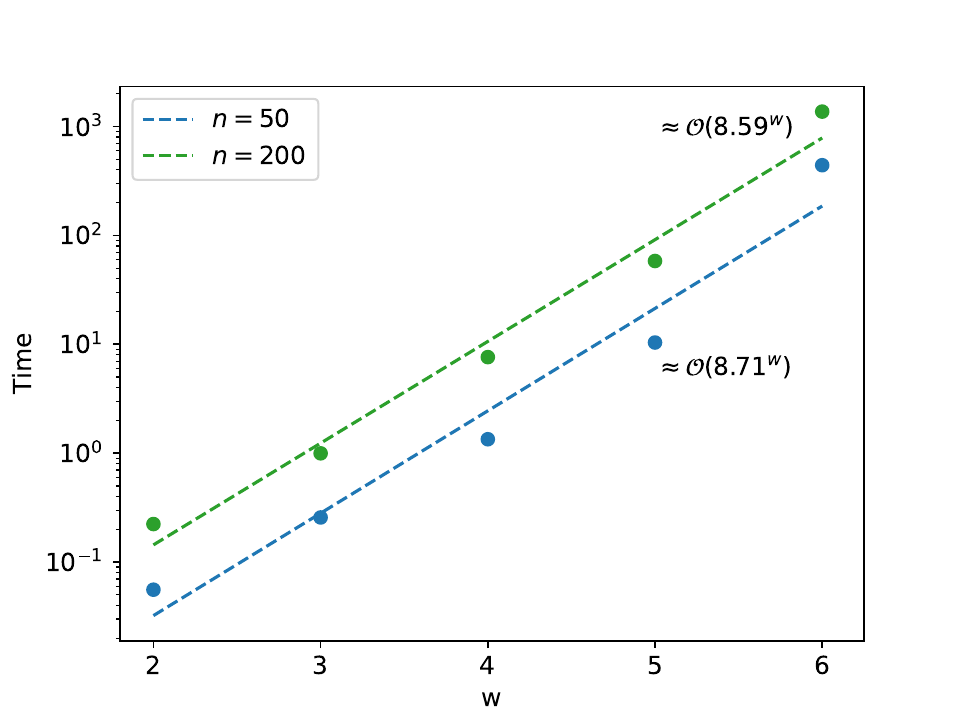}
	\caption{(Left) Runtime of the parametric algorithm as a function of $n$ for fixed $\bw$, shown on a log-log scale. The empirical runtime scales linearly with $n$. (Right) Runtime as a function of $\bw$ for fixed $n$, shown on a log-linear scale. The approximately linear trends indicate a dominant exponential dependence on $\bw$.}
	\label{fig: vary_n}
\end{figure}   	
 
\paragraph{\bf Effect of the condition number.} For $n = 1{,}000$, we vary the condition number $\kappa_2$ while keeping the bandwidth $\bw$ fixed. The regularization parameter $\vlambda$ is generated with independent entries uniformly distributed on $(2,3)$ to maintain comparable sparsity levels in the optimal solution across instances. Figure~\ref{fig: vary_condn}~(left) reports the corresponding runtimes for $\bw = 2,3,4,5$. According to Corollary~\ref{cor::banded}, the runtime of the algorithm grows polynomially with $\kappa_2$, with the exponent of this growth increasing as the bandwidth $\bw$ increases. This empirical observation is fully consistent with our theoretical result. For the small bandwidth of $\bw=2$, the runtime scales as $\mathcal{O}(\kappa_2^{1.31})$, whereas larger bandwidths yield substantially steeper growth; for instance, $\bw = 5$ exhibits a scaling of approximately $\mathcal{O}(\kappa_2^{4.20})$.

\begin{figure}[htbp]
	\centering
	\includegraphics[width=0.5\linewidth]{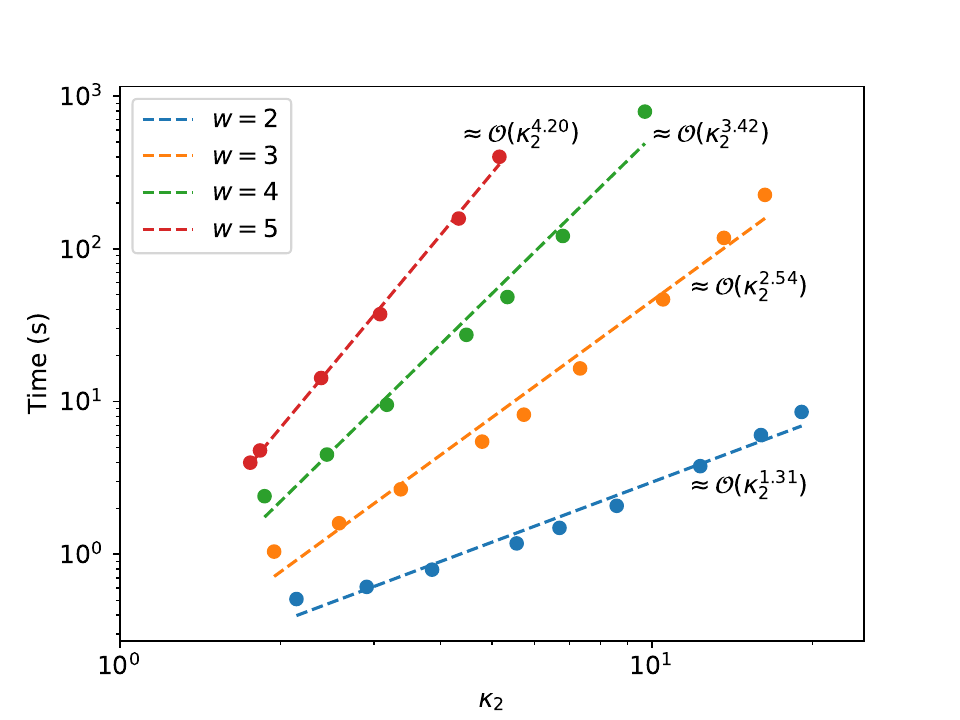}\includegraphics[width=0.5\linewidth]{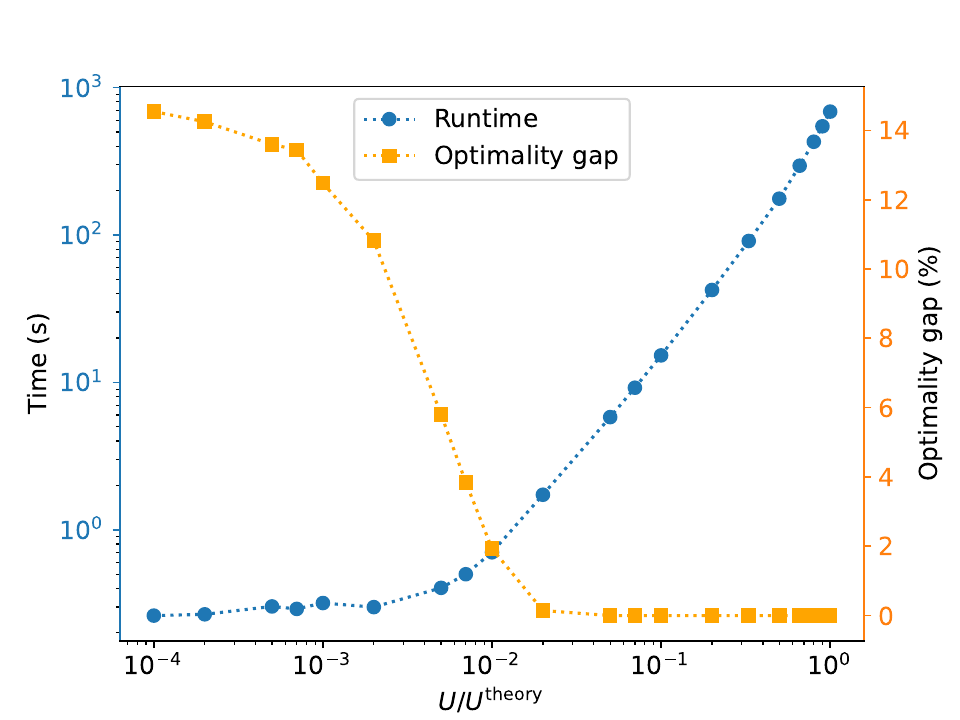}
	\caption{(Left) Runtime of the parametric algorithm as a function of $\kappa_2$ for fixed $w$, shown on a log--log scale. The empirical runtime scales polynomially with $\kappa_2$.
    (Right) Runtime (left axis) and optimality gap (right axis) as functions of the normalized parameter $U/U^{\mathrm{theory}}$, shown on a semi-logarithmic scale. Reducing $U$ reduces runtime at the cost of a larger optimality gap.
    }

	\label{fig: vary_condn}
\end{figure}

\paragraph{\bf Effect of the solution upper bound.} 
In our next experiment, we study how the choice of $U$ affects the performance of the parametric algorithm. Let $U^{\text{theory}}$ denote the value of $U$ prescribed by Lemma~\ref{lemma: U}. For this choice, the optimal solution satisfies $\|\vx^\star\|_\infty \le U^{\text{theory}}$, and therefore the parametric algorithm is guaranteed to recover the exact optimal solution. When $U < U^{\text{theory}}$, the pruning subroutine becomes more aggressive, which reduces the runtime of the algorithm. However, this runtime improvement may come at the cost of eliminating potentially optimal sparsity patterns, thereby leading to a nonzero optimality gap.

 Figure~\ref{fig: vary_condn} (right) reports the runtime (blue circles, left y-axis) and the optimality gap (orange squares, right y-axis) as functions of the normalized parameter $U / U^{\text{theory}}$. For this experiment, we fix $n = 1{,}000$, $\bw = 4$ and $\kappa_2 \approx 9.69$. As expected, when $U / U^{\text{theory}} = 1$, the algorithm recovers the exact optimal solution. As $U / U^{\text{theory}}$ decreases below one, the runtime decreases. Interestingly, for \(U / U^{\text{theory}} \ge 0.05\), the algorithm consistently recovers the optimal solution, while the runtime drops dramatically from \(600\) seconds to just \(6\) seconds. This suggests that the theoretical bound on \(U\) may be conservative in practice.

\paragraph{\bf Effect of the sparsity parameter.} The goal of the next experiment is to evaluate how the choice of the sparsity parameter $\vlambda$ affects the performance of the parametric algorithm. We fix the problem size at $n=1{,}000$, the bandwidth at $\bw=4$, and the condition number at approximately $\kappa_2 \approx 6.56$. For each trial, every coordinate $\lambda_i$ is set to the same constant value, denoted by $\bar{\lambda}$. The results are reported in Table~\ref{table: Vary lambda}. As the optimal solutions become denser, the runtime of the parametric algorithm increases modestly, reaching $135.41$ seconds when the solution has $93.7\%$ nonzero (NZ) entries. By contrast, \textsc{Gurobi} is only able to return an optimal solution in this single dense case. In all other settings, \textsc{Gurobi} times out, and the reported optimality gaps increase sharply with $\bar{\lambda}$. For extremely sparse solutions (i.e., with only $0.2\%$ nonzeros), \textsc{Gurobi} reports an optimality gap of $+\infty$.

\begin{table}[htbp]
	\begin{center}\small
		\caption{Performance comparison for varying regularization} 
		\label{table: Vary lambda}
        \resizebox{\textwidth}{!}{
		\begin{tabular}{ c|c|c|c|c|c|c|c } 
            			\hline
			\rule{0pt}{12pt}\multirow{2}{*}{$\mathbf{Metric}$}
			&\multirow{2}{*}{$\mathbf{Method}$}
			& $\pmb{\bar\lambda=0.07}$
			& $\pmb{\bar\lambda=1}$
			& $\pmb{\bar\lambda=5}$
			& $\pmb{\bar\lambda=7}$
			& $\pmb{\bar\lambda=14}$
			& $\pmb{\bar\lambda=20}$\\
			&&NZ $\approx 93.7\%$ 	
			  &NZ $\approx 75.2\%$		
			  &NZ $\approx 50.8\%$ 
			  &NZ $\approx 35.0\%$		
			  &NZ $\approx 7.38\%$
			  &NZ $\approx 0.2\%$ \\ 
			\hline
			\multirow{2}{*}{Time(s)} 
			& Parametric 	
			&135.41 &119.85 &86.71 &67.44 &36.23 &21.97\\ 
			&Gurobi		
			&10.04  & TL	 & TL	& TL	 & TL	& TL\\ 
			\hline
			B\&B nodes 
			& \multirow{2}{*}{Gurobi}	
			&1		
			&1780599	
			&2549880 
			&2840523 
			&4025755 
			&4416035\\
			Opt. gap 
			& 								
			&0.00\%
			&0.15\%		
			&2.26\%	 
			&9.63\% 
			&210.25\%
			&$\infty$ \\
			\hline
		\end{tabular}
        }
	\end{center}
	\footnotesize NZ refers to the percentage of non-zero elements in the optimal solution $\vx^\star$. 
\end{table}
This behavior highlights a key strength of the parametric algorithm. In applications such as ESOC, typically only a small fraction of observed values are corrupted with outlier noise. In precisely these regimes, where \textsc{Gurobi} returns solutions with large optimality gaps, the parametric algorithm recovers the optimum with much smaller runtime.

\paragraph{\bf Beyond banded structures.}
Next, we show that structural properties implied by small treewidth can be exploited beyond simple banded structures. To this end, we compare the performance of two variants of our algorithm: one that fully exploits the low-treewidth structure of the problem, and another that leverages only the banded structure.
We construct instances for $\mQ$ where the treewidth is fixed at $\omega=2$, while the bandwidth varies over $\bw=\{3,4,5\}$. One such graph is illustrated in Figure~\ref{fig: exp small tw}. To generate a matrix $\mQ$ with the desired properties, we first form a symmetric positive definite matrix with the specified bandwidth, following the procedure outlined in the previous section. We then selectively zero out certain off-diagonal entries to reduce the treewidth to $\omega=2$, while preserving the original banded structure. 

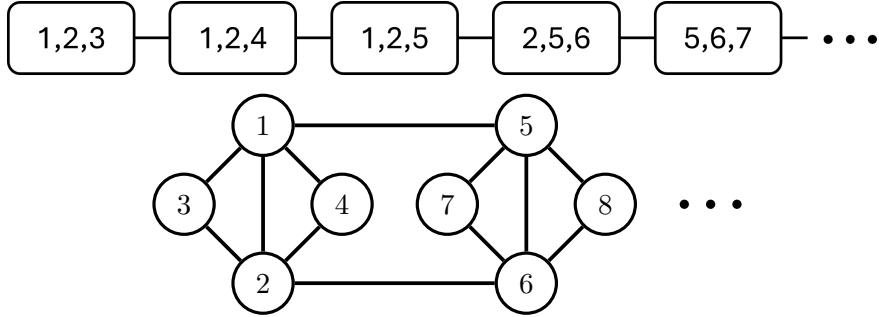
\begin{figure}[ht]
    \centering
    \includegraphics[width=0.7\linewidth]{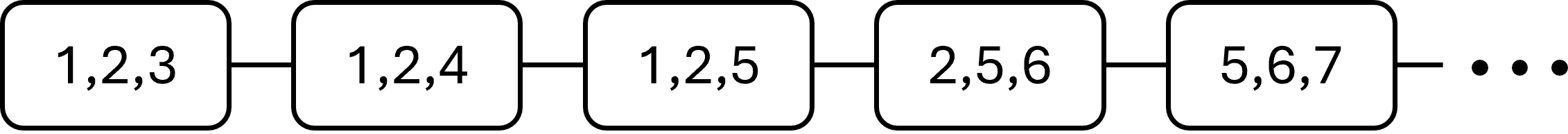}
    \vspace{2mm}
    
    \begin{tikzpicture}[scale=.7,auto=center,every node/.style={circle,draw=black!100, very thick, minimum size=8mm}, squarenode/.style={rectangle,draw=black!100, very thick, minimum size=8mm},roundnode/.style={rectangle, draw=red!60, fill=red!5, very thick, minimum size=7mm},yellownode/.style={circle, draw=yellow!100, fill=yellow!15, very thick, minimum size=7mm},greennode/.style={circle, draw=green!60, fill=green!15, very thick, minimum size=7mm},bluenode/.style={circle, draw=blue!60, fill=blue!20, very thick, minimum size=7mm}] 

    \node (n1) at (0,0) {1};
    \node (n2) at (0,-3) {2};
    \node (n3) at (-1.5,-1.5) {3};
    \node (n4) at (1.5,-1.5) {4};
    
    \node (n5) at (5,0) {5};
    \node (n6) at (5,-3) {6};
    \node (n7) at (3.5,-1.5) {7};
    \node (n8) at (6.5,-1.5) {8};

\draw[-, line width=0.5mm]  (n1) -- (n3);
\draw[-, line width=0.5mm]  (n2) -- (n3);
\draw[-, line width=0.5mm]  (n1) -- (n5);
\draw[-, line width=0.5mm]  (n1) -- (n4);
\draw[-, line width=0.5mm]  (n2) -- (n4);

\draw[-, line width=0.5mm]  (n1) -- (n2);
\draw[-, line width=0.5mm]  (n2) -- (n6);
\draw[-, line width=0.5mm]  (n5) -- (n6);

\draw[-, line width=0.5mm]  (n5) -- (n7);
\draw[-, line width=0.5mm]  (n5) -- (n8);
\draw[-, line width=0.5mm]  (n6) -- (n7);
\draw[-, line width=0.5mm]  (n6) -- (n8);

\foreach \x in {8,8.5,9}
        \fill (\x,-1.5) circle (2.5pt);

	\end{tikzpicture}
    \caption{Tree decomposition of a graph derived from a banded matrix with bandwidth $\bw=4$ and treewidth $\omega=2$. The treewidth is reduced by eliminating edges while preserving the original bandwidth.}
    \label{fig: exp small tw}
\end{figure}

Table~\ref{table: exp small tw} reports the performance of our methods for a fixed treewidth $\omega = 2$ and bandwidths $\bw = 3,4,5$. All experiments are conducted with problem size $n=1{,}000$, and the corresponding condition numbers $\kappa_2$ are reported in the table. The row labeled ``{Parametric}'' corresponds to the parametric algorithm applied to a tree decomposition of width $\omega = 2$. The row labeled ``{Banded}'' corresponds to the same parametric algorithm applied under the assumption of a banded structure only; in this case, the width of the induced tree decomposition equals $\bw>2$. As shown in the table, exploiting the small treewidth of the graph—beyond merely its banded structure—substantially reduces both the runtime and the average number of quadratic pieces. In contrast, \textsc{Gurobi} is unable to solve any of these instances within the one-hour time limit. 

\begin{table}[htbp]
	\begin{center}\small
		\caption{Performance comparison for varying bandwidth}
		\label{table: exp small tw}
		\begin{tabular}{ c|c|c|c|c} 
			\hline
			\textbf{Metric} & \textbf{Method} & $\pmb{\bw=3}$ & $\pmb{\bw=4}$&  $\pmb{\bw=5}$\\ 
			\hline
            Condition no. ($\kappa_2$)&---&8.71&9.25&8.12\\
            \hline
			\multirow{3}{*}{Time(s)} & Parametric & 8.60& 48.58& 167.83\\
            &Banded & 8.52& 64.36& 658.24\\ 
			& \textsc{Gurobi}& TL& TL& TL\\ 
			\hline
            \multirow{2}{*}{Avg no. eqs}& Parametric&81&442&1477\\
            &Banded&123&848&3385\\
            \hline
		B\&B nodes&\multirow{2}{*}{Gurobi}&3193465&3594783&3552258\\
		Opt. gap && 10.02\%&14.02\%&12.67\%\\
        \hline
		\end{tabular}
	\end{center}
\end{table}

\subsection{Results for real-world instances}\label{sec: experiments real world}

Next, we apply the proposed parametric algorithm to solve the problem of exponential smoothing with outlier correction (\ref{eq: ESO}), introduced in Section~\ref{sec: ESO}. In addition to its runtime, we focus on the forecasting accuracy of this model relative to simple exponential smoothing (\ref{eq: SES}).

\paragraph{\bf Datasets}
Our experiments use four real-world time series from the Numenta Anomaly Benchmark (NAB)~\citep{ahmad2015nab}. NAB is a dataset designed to evaluate anomaly detection algorithms on streaming data. It includes real-world time-series data from diverse domains, such as cloud infrastructure metrics and traffic data. For detailed descriptions of the dataset, evaluation protocols, and related methods, see \citep{lavin2015evaluating,ahmad2017unsupervised}. The specific datasets used in our experiments are:
\begin{enumerate}
    \item \texttt{ec2\_cpu\_utilization\_53ea38.csv}; referred to as \texttt{CPU-1},
    \item \texttt{ec2\_cpu\_utilization\_ac20cd.csv}; referred to as \texttt{CPU-2},
    \item \texttt{rds\_cpu\_utilization\_e47b3b.csv}; referred to as \texttt{CPU-3},
    \item \texttt{speed\_7578.csv}; referred to as \texttt{Traffic}.
\end{enumerate}
Below, we briefly explain these datasets.

\paragraph{CPU utilization data}
The first three datasets consist of CPU utilization percentages collected from Amazon Web Services (AWS) servers. In such datasets, anomalies may stem from sudden workload surges or Distributed Denial-of-Service (DDoS) attacks~\citep{el2022flow}. The first two datasets correspond to general-purpose compute instances (Amazon EC2), while the third is obtained from a database server (Amazon RDS). Each dataset contains $2{,}000$ CPU utilization measurements recorded at $5$-minute intervals. As shown in the first three rows of Figure~\ref{fig: SES vs ESO}, these time series exhibit distinct temporal patterns, providing a diverse range of behaviors for evaluating the accuracy of the \ref{eq: ESO} model.

\paragraph{Traffic data}
The final dataset consists of average traffic speed measurements from the Twin Cities Metro area in Minnesota and contains $1{,}127$ observations recorded at irregular time intervals. Accurate short-term traffic speed forecasting is critical for anticipating changing road conditions and enabling effective traffic management~\citep{ouyang2020large}. In this context, anomalies may signal a major crash causing traffic obstruction~\citep{zhao2023unsupervised, van2019real}. This dataset is shown in the last row of Figure~\ref{fig: SES vs ESO} and is included to evaluate the proposed method on data exhibiting sampling irregularity and noise characteristics that differ substantially from those of the CPU utilization signals.

\paragraph{\bf Experimental setup}
Let $\vy \in \mathbb{R}^T$ denote a time series of length $T$, where $\vy_t$ represents the observed value at time $t$. Our goal is to produce one-step-ahead point forecasts at each time $t$. We denote the resulting forecasts by $\hat{\vy}^{\text{SES}}_{t}$ for \ref{eq: SES}, and by $\hat{\vy}^{\text{ESOC}}_{t}$ for \ref{eq: ESO}.
For both models, the one-step-ahead forecast at time $t+1$ is obtained from the smoothed value at the previous time $t$~\citep{hyndman2008forecasting}. Specifically, for~\ref{eq: SES}, the forecast is given by $\hat{\vy}^{\text{SES}}_{t+1} := \vx^{\text{SES}}_t = \beta \vy_t + (1-\beta)\vx^{\text{SES}}_{t-1}$, where $\vx^{\text{SES}}_t$ denotes the exponentially smoothed signal. For~\ref{eq: ESO}, the forecast is defined as $\hat{\vy}^{\text{ESOC}}_{t+1} := \vx^{\text{ESOC}}_t$, where $\vx^{\text{ESOC}} \in \mathbb{R}^T$ is the smoothed signal obtained by solving~\ref{eq: ESO}. To evaluate the forecasting accuracy of~\ref{eq: SES}, we use the mean squared error (MSE):
\begin{align}\label{eq: MSE-SES}
\operatorname{MSE}^{\text{SES}} = \frac{1}{T-1} \sum_{t=2}^{T} \bigl(\hat{\vy}^{\text{SES}}_t - \vy_t\bigr)^2.
\end{align}
Unlike the model in~\eqref{eq: SES}, the model in~\eqref{eq: ESO} is capable of identifying and discounting outliers. Accordingly, we evaluate its forecasting accuracy only over the outlier-free region. Specifically, we define
\begin{align}\label{eq: MSE-ESO}
\operatorname{MSE}^{\text{ESOC}} = \frac{1}{T-1-\sum_{t=2}^{T}\bbbone(\vo_t)} \sum_{t=2}^{T} \bigl(1 - \bbbone(\vo_t)\bigr)\bigl(\hat{\vy}^{\text{ESOC}}_t - \vy_t\bigr)^2,
\end{align}
where $\bbbone(\vo_t)$ is an indicator that equals one if $\vy_t$ is identified as an outlier and zero otherwise.

The parameters of both models are selected via grid search over a predefined set of candidate values. For parameter tuning, we split each time series into a training set consisting of the first half of the observations, $\{\vy_t\}_{t=1}^{\lfloor T/2\rfloor}$, and a test set consisting of the remaining observations, $\{\vy_t\}_{t=\lfloor T/2\rfloor+1}^{T}$. For each method and parameter configuration, we compute the MSE on the training set and select the parameters that minimize this quantity. Using the selected parameters, we run each algorithm on the entire signal and compare the performance by computing MSE on the test set, according to~\eqref{eq: MSE-SES} for~\eqref{eq: SES} and~\eqref{eq: MSE-ESO} for~\eqref{eq: ESO}.

For the model in~\eqref{eq: ESO}, we perform a grid search over $\beta \in \{0.01, 0.1, 0.2, \ldots, 0.9, 0.99\}$ and $\vlambda_t = \lambda, t=1,\dots, T$, where $\lambda \in \{10^{-5}, 5\times10^{-5}, 10^{-4}, 5\times10^{-4}, 10^{-3}, 5\times10^{-3}, 10^{-2}, 5\times10^{-2}\}$, while fixing $\mu_1 = 1.2$ and $\mu_2 = 0.001$. To avoid degenerate solutions, we restrict the grid search to parameter settings that classify fewer than $10\%$ of the observations as anomalies.
Since the model in~\eqref{eq: SES} has a single tunable parameter $\beta$, we select $\beta$ using the same candidate set as above.

\paragraph{\bf Results}
Table~\ref{tab: SES_vs_ESO} reports both the training and test MSE values of the smoothed signals, using the parameters obtained over the training set. For \eqref{eq: ESO}, we also report the percentage of detected outliers. In all cases, \eqref{eq: ESO} achieves lower MSE than \eqref{eq: SES}, demonstrating its ability to suppress anomalies and produce more accurate forecasts. Notably, the largest improvements in test MSE are observed for the \texttt{traffic} and \texttt{CPU-3} signals. We note that the fraction of detected outliers is constrained to be below \(10\%\) only during the training phase. The results in this table are obtained by solving the optimization problem over the entire signal (training and test) using the parameters learned from training. Consequently, for the \texttt{traffic} signal, the observed outlier proportion exceeds this threshold.

\begin{table}[ht]
\caption{Comparison of MSE for SES and ESOC.}
\centering\small
\begin{tabular}{l|c|cc|c}
\hline
\multirow{2}{*}{\textbf{Signal}} & \multirow{2}{*}{\ \textbf{Segment}\ }
& \multicolumn{2}{c|}{\textbf{MSE}}
& \multirow{2}{*}{\textbf{Outliers detected}} \\
\cline{3-4}
& 
& \ \textbf{SES}\ & \ \textbf{ESOC}\ 
&  \\
\hline
\multirow{2}{*}{\texttt{CPU-1}}
 & Train & 0.0101 & 0.0063 & 3.2\% \\
 & Test  & 0.0106 & 0.0068 & 2.8\% \\
\hline
\multirow{2}{*}{\texttt{CPU-2}}
 & Train & 9.1930 & 3.0941 & 2.9\% \\
 & Test  & 5.3983 & 3.1840 & 3.0\% \\
\hline
\multirow{2}{*}{\texttt{CPU-3}}
 & Train & 6.4085 & 0.1404 & 6.0\% \\
 & Test  & 0.8021 & 0.1649 & 16.6\% \\
\hline
\multirow{2}{*}{\texttt{Traffic}\ }
 & Train & 20.2650 & 6.7490 & 11.5\% \\
 & Test  & 65.5931 & 6.0920 & 20.7\% \\
\hline
\end{tabular}
\label{tab: SES_vs_ESO}
\end{table}

Table~\ref{tab: ESO_runtime} compares the runtime of \textsc{Gurobi} and the proposed parametric algorithm for solving~\eqref{eq: ESO}. For all instances, \textsc{Gurobi} reaches the one-hour time limit; we therefore report only the optimality gaps at termination. In contrast, the parametric algorithm returns provably optimal solutions for all instances. Most datasets are solved within approximately $30$ seconds, with one notably larger instance (\texttt{CPU-3}) requiring about $958$ seconds. Overall, these results demonstrate that the parametric algorithm consistently attains optimal solutions and does so far more efficiently than \textsc{Gurobi}. 

\begin{table}[ht]\small
\caption{Performance comparison between the parametric algorithm and \textsc{Gurobi} for solving~\ref{eq: ESO}.}
\centering
\label{tab: realworld runtime}
\begin{tabular}{c|c|c|c|c|c}
\hline
\textbf{Metric} & \textbf{Method} & \textbf{\texttt{CPU-1}} & \textbf{\texttt{CPU-2}} & \textbf{\texttt{CPU-3}} & \textbf{\texttt{Traffic}} \\
\hline
Signal length ($T$)&---&2000&2000&2000&1127\\
\hline
\multirow{2}{*}{Time(s)}& Parametric  & 21.24 & 24.94 & 958.93 & 27.10 \\
&Gurobi & TL & TL & TL & TL \\
\hline
B\&B nodes & \multirow{2}{*}{Gurobi}& 2077648 &3909773
 &1609698 &7360291\\
Opt. gap &  & 78.88\% & 2.04\% & 1.06\% & 1.94\% \\

\hline
\end{tabular}
\label{tab: ESO_runtime}

\end{table}

\begin{figure}[htbp]
    \centering
    \vspace{-5mm}
    \includegraphics[width=0.45\linewidth]{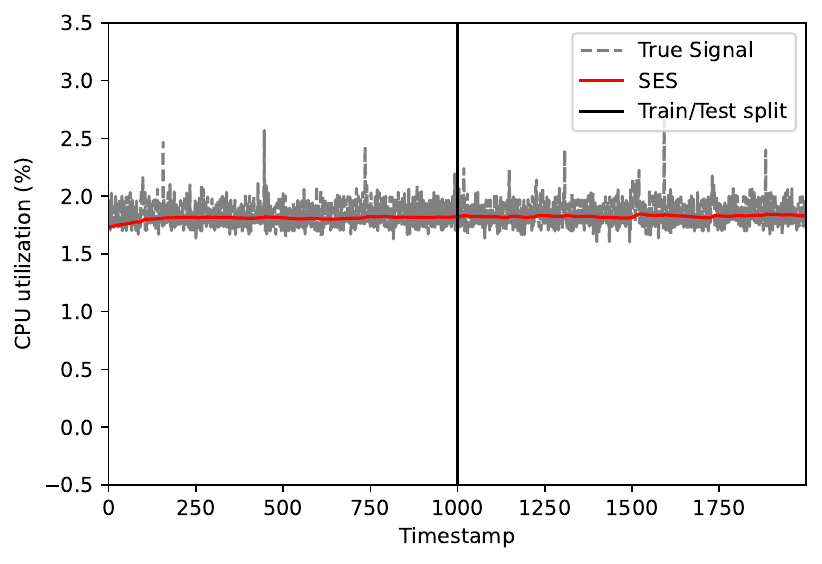}\includegraphics[width=0.45\linewidth]{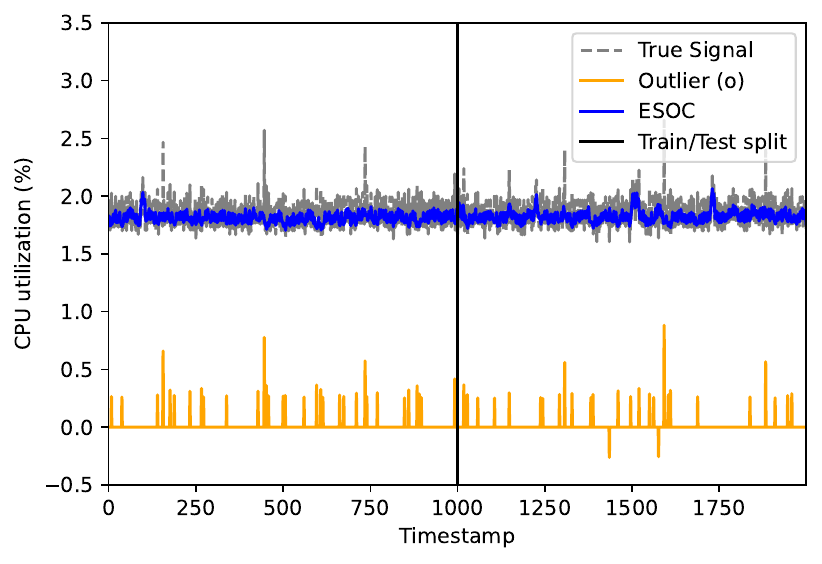}
    \includegraphics[width=0.45\linewidth]{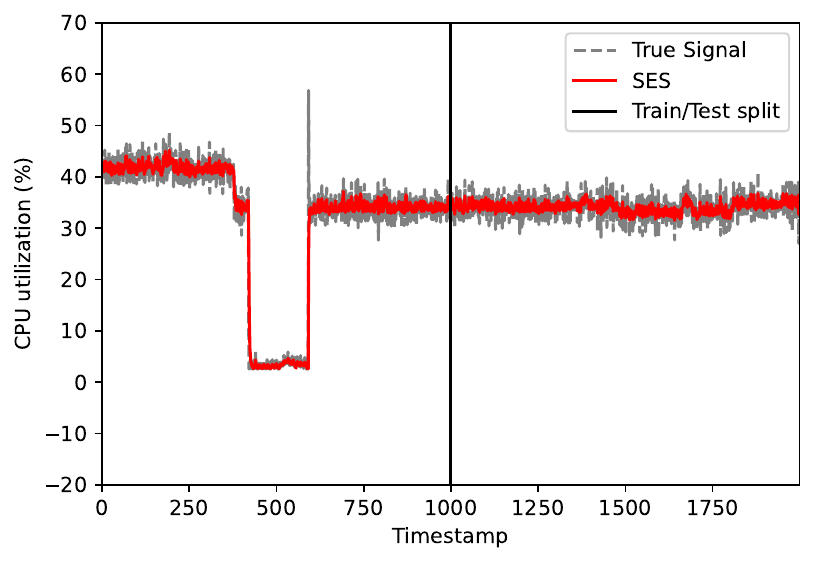}\includegraphics[width=0.45\linewidth]{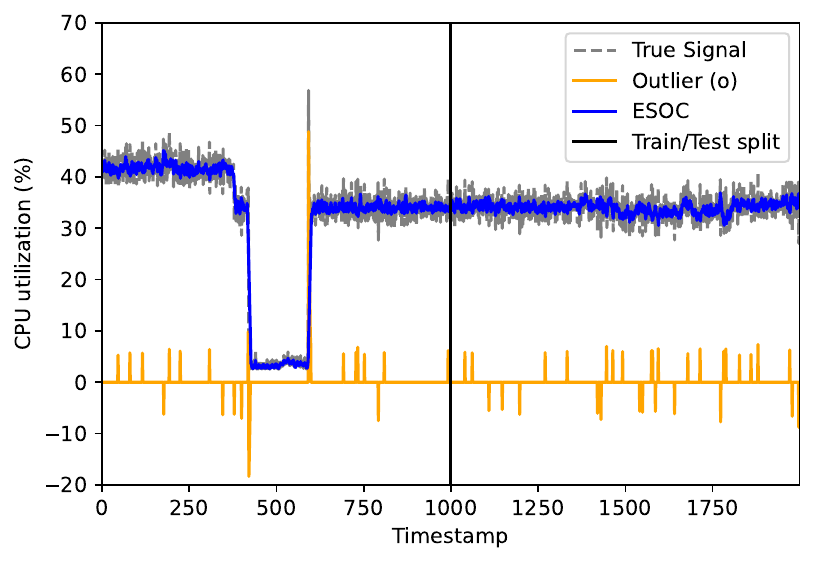}
    \includegraphics[width=0.45\linewidth]{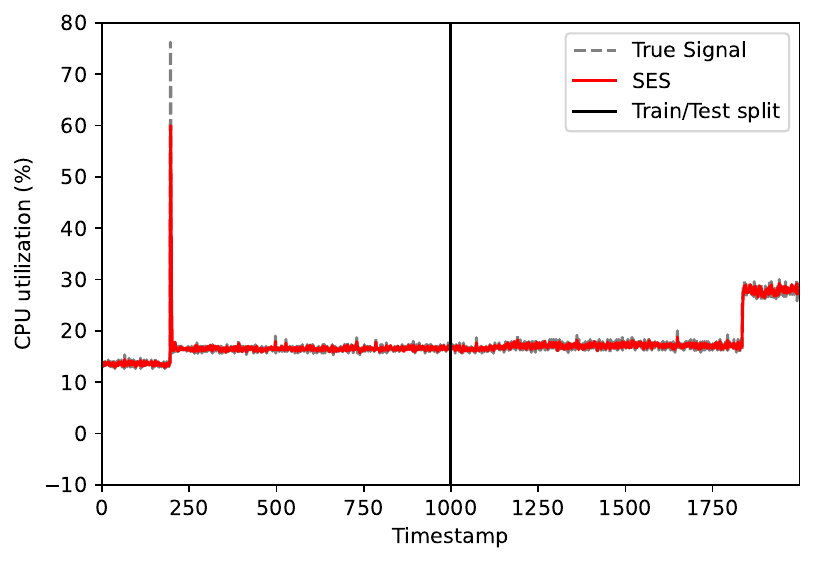}\includegraphics[width=0.45\linewidth]{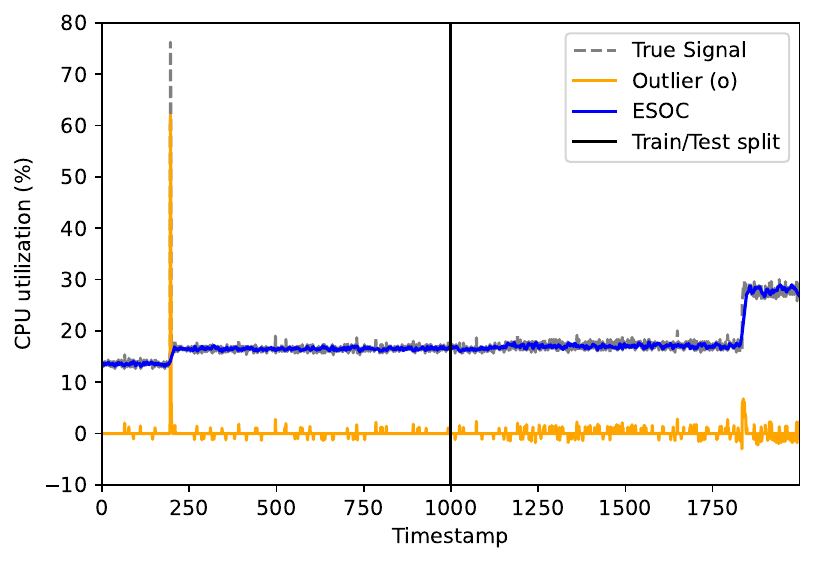}
    \includegraphics[width=0.45\linewidth]{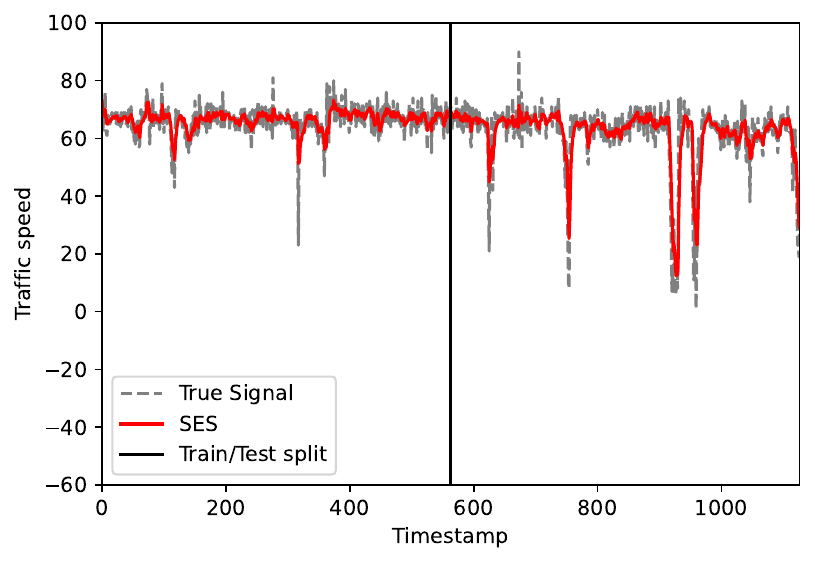}\includegraphics[width=0.45\linewidth]{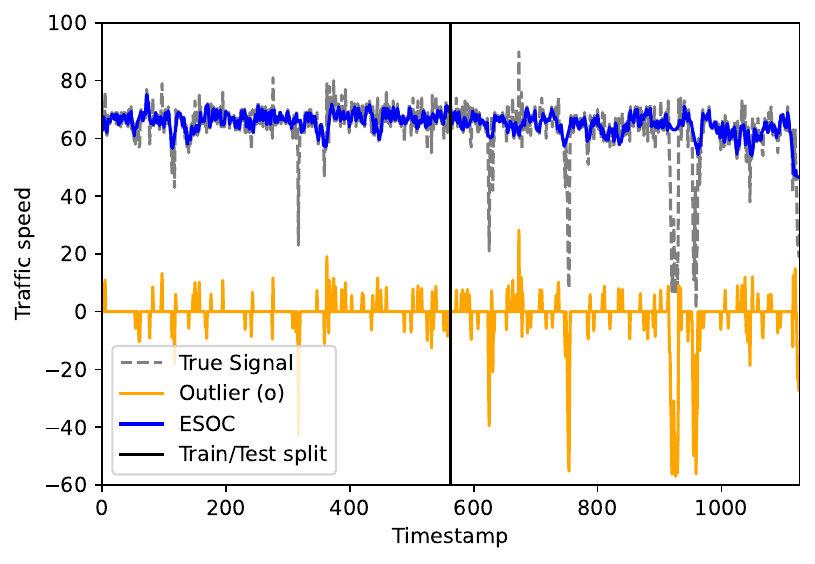}
    \caption{\small Comparison of SES (left column) and ESOC (right column) on four representative real-world signals. Each row corresponds to a different signal instance, with the observed signal shown alongside the smoothed estimates.
}
    \label{fig: SES vs ESO}
\end{figure}

Finally, Figure~\ref{fig: SES vs ESO} visualizes the recovered signals obtained using the models in~\eqref{eq: SES} and~\eqref{eq: ESO} across all four datasets.

Each row corresponds to a single dataset, with the output of~\ref{eq: SES} shown in the left column and that of~\ref{eq: ESO} shown in the right column. In all panels, the observed signal is plotted together with the corresponding smoothed estimate, and for~\ref{eq: ESO}, the estimated outlier vector $\vo$ is highlighted in orange.

For the first dataset (\texttt{CPU-1}), the trend in the training set is representative of the full series. Under~\ref{eq: SES}, the forecast is relatively smooth due to the small value of $\beta$ selected on the training data, but exhibits a noticeable lag in the early portion of the series, where the recursion is dominated by the initial observation $\vy_1$ (see Section~\ref{sec: exponential smoothing}). In contrast, the forecast from~\ref{eq: ESO} adapts more quickly, tracks the underlying signal more closely, and, by explicitly modeling outliers, shows reduced sensitivity to the initial value $\vy_1$.

For the second dataset (\texttt{CPU-2}), the training and test segments exhibit visibly different behaviors. The training portion contains abrupt changes (between $t=420$ and $t=600$) followed by a large outlier, both of which are absent in the test set (the second half of the signal). While both~\ref{eq: SES} and~\ref{eq: ESO} attenuate this outlier,~\ref{eq: ESO} produces a smoother estimate of the underlying trend. As shown in Table~\ref{tab: SES_vs_ESO}, this leads to a substantially lower test MSE for~\ref{eq: ESO}.

For the third dataset (\texttt{CPU-3}), the training set contains a large outlier followed by a single jump (around $t\approx 240$), while the test set exhibits an additional jump. Indeed,~\ref{eq: SES} adapts to the jumps but incorporates the large outlier into its forecast. In contrast,~\ref{eq: ESO} successfully isolates and removes the outlier, accurately tracks the jumps, and yields a smoother estimate of the underlying trend.

Finally, for the fourth dataset (\texttt{Traffic}), the signal is less smooth and exhibits frequent short-term fluctuations across both the training and test sets. Toward the end of the time horizon, two large spikes occur. While~\ref{eq: SES} retains both small and large fluctuations, including the large spikes,~\ref{eq: ESO} captures the underlying trend and smaller variations while effectively discarding the large spikes.

\section{Conclusion}\label{sec: conclusion}

In summary, this paper demonstrates that convex quadratic optimization problems with indicator variables can be solved efficiently by exploiting graph-structured sparsity and a margin condition. By leveraging these properties, the proposed parametric algorithm achieves polynomial time and memory complexity—and, in certain regimes, even linear time and memory complexity—under suitable assumptions, substantially outperforming general-purpose mixed-integer solvers on large-scale instances. Beyond its theoretical guarantees, the framework offers a principled extension of exponential smoothing for joint forecasting and outlier detection, and exhibits strong empirical performance on challenging real-world time-series datasets.

\section{Code and Data Disclosure}\label{sec: Code and Data Disclosure}
The code and data to support the numerical experiments in this paper can be found at 
\begin{center}
    \href{https://github.com/aareshfb/Treewidth-Parametric-Algorithm}{https://github.com/aareshfb/Treewidth-Parametric-Algorithm}.
\end{center}
Note that the implementation provided is restricted to instances where the tree decomposition is a path. 
The Numenta Anomaly Benchmark (NAB)~\citep{ahmad2015nab} used in the experiments containing real-world data can be accessed at 
\begin{center}
    \href{https://github.com/numenta/NAB/tree/master/data}{https://github.com/numenta/NAB/tree/master/data}.
\end{center}

\section*{Acknowledgments}

This research is supported, in part, by NSF grants 2152776, 2337776, ONR
grants N00014-26-1-2074, N00014-22-1-2602, N00014-26-12117, and AFOSR grant FA9550-24-1-0086.

\bibliographystyle{plainnat} 
\bibliography{mybib}

\appendix

\section{Omitted proofs}
\subsection{Proof of Lemma~\ref{lemma: f_u as p_s}} \label{app:lemma1}
Recall that $\J_u$ was defined as the set of nodes in $\supp_u(\mQ)$, excluding the nodes in $\B_u$. Let us define $p_{u,\bs}(\valpha_{\B_u})$ as
\begin{subequations}\nonumber
	\begin{align}
		p_{u,\bs}(\valpha_{\B_u}) = \min_{\vx\in\R^{n_u}}&\ \frac{1}{2}\valpha_{\B_u}^\top \mQ_{\B_u,\B_u}\valpha_{\B_u}+\vc_{\B_u}^\top \valpha_{\B_u}+\left(\frac{1}{2}\vx^\top \mQ_{\J_u,\J_u}\vx+ \valpha_{\B_u}^\top \mQ_{\B_u,\J_u} \vx + \vc_{\J_u}^\top \vx+\vlambda_{\J_u}^\top \bs\right)\\ 
		\text{s.t. }& \ x_i(1-s_i)=0\quad  i\in \J_u.
	\end{align}
\end{subequations}
It is easy to verify that $f_u(\valpha_{\B_u}) = \min_{\bs\in\{0,1\}^{n_u}}\{p_{u,\bs}(\valpha_{\B_u})\}$. Therefore, it remains to characterize the explicit form of $p_{u,\bs}(\valpha_{\B_u})$ for every $\bs\in\{0,1\}^{n_u}$, and show that it is strongly convex and quadratic. 

First, let $\J_{u,s}=\{i\in \J_u \mid  s_i=1\}$. Therefore, we have
\begin{subequations}\nonumber
	\begin{align}
		p_{u,\bs}(\valpha_{\B_u}) =& \frac{1}{2}\valpha_{\B_u}^\top \mQ_{\B_u,\B_u}\valpha_{\B_u}+\vc_{\B_u}^\top\valpha_{\B_u}\\
        &+\min_{\vx\in\R^{|\J_{u,\bs}|}}\left\{\frac{1}{2}\vx^\top \mQ_{\J_{u,\bs},\J_{u,\bs}}\vx+ \valpha_{\B_u}^\top \mQ_{\B_u,\J_{u,\bs}} \vx + \vc_{\J_{u,\bs}}^\top \vx+\sum_{i\in \J_{u,\bs}} \vlambda_{i}\right\}. 
	\end{align}
\end{subequations}
From the Karush-Kuhn-Tucker conditions, it follows that
\begin{subequations}\nonumber
	\begin{align}
		p_{u,\bs}(\valpha_{\B_u}) = &\frac{1}{2}\valpha_{\B_u}^\top\left(\mQ_{\B_u,\B_u}-\mQ_{\B_u,\J_{u,\bs}} \left(\mQ_{\J_{u,\bs},\J_{u,\bs}}\right)^{-1}\mQ_{\B_u,\J_{u,\bs}}^\top\right)\valpha_{\B_u}\\
        &+\left(\vc_{\B_u}-\vc_{\J_{u,\bs}}^T \left(\mQ_{\J_{u,\bs},\J_{u,\bs}}\right)^{-1}\mQ_{\B_u,\J_{u,\bs}}^{\top}\right)^\top\valpha_{\B_u} \\
		&           +\left(-\frac{1}{2}\vc_{\J_{u,\bs}}^\top (\mQ_{\J_{u,\bs},\J_{u,\bs}})^{-1}\vc_{\J_{u,\bs}}+\sum_{i\in \J_{u,\bs}} \lambda_{i}\right).
	\end{align}
\end{subequations} 
Furthermore, note that $\left(\mQ_{\B_u,\B_u}-\mQ_{\B_u,\J_{u,\bs}}\left(\mQ_{\J_{u,\bs},\J_{u,\bs}}\right)^{-1}\mQ_{\B_u,\J_{u,\bs}}^\top\right)$ is the Schur complement of \linebreak $\mQ_{\J_{u,\bs}\cup\{\B_u\}, \J_{u,\bs}\cup\{\B_u\}}$, which, owing to the positive definiteness of $\mQ$, is positive definite. This completes the proof.\qed\medskip

\subsection{Proof of Lemma~\ref{lemma: fu recursion}}\label{app: fu recursion}
    
    First, we derive an explicit expression for $g_v$ for each $v \in \parent_{\textsf T}(u)$. Next, we show that the optimization problem defining $f_u$ can be decomposed into independent subproblems for each $v$. Finally, we substitute the expressions for $g_v$ within each subproblem to obtain the stated equation for $f_u$.

    We start with the first step and derive an explicit expression for $g_v$. Fix $v\in\parent(u)$ and consider the local parametric cost $f_v$. Setting the variables $\vx_{\pi_v(\B_v)}=\valpha_{\B_v}$ according to Constraint~\eqref{const: x=alpha tree} in the definition of $f_v$, we obtain:
        \begin{subequations}\label{eq: f_u with valpha fixed}
            \begin{align}
            f_v(\valpha_{\B_v})=&\frac{1}{2}\valpha_{\B_v}^\top \mQ_{\B_v,\B_v} \valpha_{\B_v}+\vc_{\B_v}^\top \valpha_{\B_v}\nonumber\\
            &+\min_{\vx\in\R^{n_v},\vz\in\{0,1\}^{n_v}} \left\{\frac{1}{2}\vx_{\pi_v(\J_v)}^\top \mQ_{\J_v,\J_v} \vx_{\pi_v(\J_v)}+\vx_{\pi_v(\J_v)}^\top \mQ_{\J_v,\B_v} \valpha_{\B_v}+\vc_{\J_v}^\top \vx_{\pi_v(\J_v)}+\vlambda_{\J_v}^\top \vz_{\pi_v(\J_v)}\right\}\label{obj: f_u with valpha fixed}\\
            &\qquad\qquad \text{s.t.}\qquad \vx_i(1-\vz_i)=0\qquad \forall i\in \pi_v(\J_v).
        \end{align}
        \end{subequations}
        After isolating the contribution of the variable $\valpha_v$, $\frac{1}{2}\valpha_{\B_v}^\top \mQ_{\B_v,\B_v} \valpha_{\B_v}+\vc_{\B_v}^\top \valpha_{\B_v}$ can be rewritten as
        \begin{align}\label{eq: f_v decomposition1}
            \frac{1}{2}\valpha_{\B_v}^\top \mQ_{\B_v,\B_v} \valpha_{\B_v}
            + \vc_{\B_v}^\top \valpha_{\B_v}
            &=
            \frac{1}{2}\valpha_{\B_v\setminus v}^\top
            \mQ_{\B_v\setminus v,\,\B_v\setminus v}
            \valpha_{\B_v\setminus v}
            + \vc_{\B_v\setminus v}^\top \valpha_{\B_v\setminus v}   \nonumber\\
            &\quad
            + \frac{1}{2}\mQ_{v,v}\valpha_v^2
            + \valpha_v \mQ_{v,\B_v\setminus v}\valpha_{\B_v\setminus v}+ \vc_v \valpha_v  \nonumber\\
            &=
            \phi_v(\valpha_{\B_v\setminus {v}})
            + \frac{1}{2}\mQ_{v,v}\valpha_v^2
            + \valpha_v \mQ_{v,\B_v\setminus v}\valpha_{\B_v\setminus v}
            + \vc_v \valpha_v .
        \end{align}
         
        Similarly, the term $\vx_{\pi_v(\J_v)}^\top \mQ_{\J_v,\B_v}\valpha_{\B_v}$ that appears in~\eqref{obj: f_u with valpha fixed} can be decomposed as
        \begin{align}\label{eq: f_v decomposition2}
            \vx_{\pi_v(\J_v)}^\top \mQ_{\J_v,\B_v}\valpha_{\B_v}&=\vx_{\pi_v(\J_v)}^\top \mQ_{\J_v,\B_v\backslash v}\valpha_{\B_v\backslash v}+\vx_{\pi_v(\J_v)}^\top \mQ_{\J_v,v}\valpha_{v}.
        \end{align}
        We now recall the definition of $g_v$:
        $$g_{v}(\valpha_{\B_v\backslash v})=\min\limits_{\vx_v\in\R} \left\{f_{v}(\vx_v,\valpha_{\B_v\backslash v})+\vlambda_v\bbbone(\vx_v)\right\}.$$ 
        Substituting \eqref{eq: f_v decomposition1} and \eqref{eq: f_v decomposition2} together with \eqref{eq: f_u with valpha fixed}, into the expression for \(g_v\) yields
        \begin{subequations}\label{eq: g_u for recursion}
        \begin{align}
            g_v(\valpha_{\B_v\backslash v})&=\phi_v(\valpha_{\B_v\backslash v})+\nonumber\\
            &\min_{\vx\in\R^{n_v+1},\vz\in\{0,1\}^{n_v+1}} \Biggl\{\frac{1}{2}\vx_{\pi_v(\J_v)}^\top \mQ_{\J_v,\J_v} \vx_{\pi_v(\J_v)}+\frac{1}{2}\mQ_{v,v}\vx_{\pi_v(v)}^2+\nonumber\\
            &\hspace{8em} \vx_{\pi_v(\J_v)}^\top \mQ_{\J_v,\B_v\backslash v} \valpha_{\B_v\backslash v}+\vx_{\pi_v(v)}\mQ_{v,\B_v\backslash v}\valpha_{\B_v\backslash v} +\nonumber\\
            &\hspace{8em} \vx_{\pi_v(\J_v)}^\top \mQ_{\J_v,v}\vx_{\pi_v(v)} +\vc_v \vx_{\pi_v(v)} +\vc_{\J_v}^\top \vx_{\pi_v(\J_v)}+\nonumber\\
            &\hspace{8em} \vlambda_v\vz_{\pi_v(v)}+\vlambda_{\J_v}^\top \vz_{\pi_v(\J_v)} \Biggr\}\\
            &\qquad\qquad \text{s.t.}\qquad \vx_i(1-\vz_i)=0\qquad \forall i\in \pi_v(\J_v\cup \{v\}).
        \end{align}
        \end{subequations}
        Having derived the expression for $g_v$, we now establish the equation for $f_u$ stated in the lemma. 
        
        The nodes in the induced subgraph $\supp_u(\mQ)$ can be written as the disjoint union $\B_u \cup \J_u$, and the set $\J_u$ further decomposes as
        \[
        \J_u = \bigcup_{v\in\parent_{\textsf{T}}(u)} (\J_v \cup \{v\}).
        \]
        To formalize the properties of this decomposition, we state the following claim.
        \begin{claim}\label{claim: f_u recursion}
            For any two distinct $v_i,v_j\in\parent_{\textsf{T}}(u)$, the following holds
            \begin{enumerate}
                \item $\bigl(\J_{v_i}\cup \{v_i\}\bigr)\cap\bigl(\J_{v_j}\cup \{v_j\}\bigr)=\emptyset$;
                \item For any $\bar i\in \J_{v_i}\cup \{v_i\}$ and $\bar j\in \J_{v_j}\cup \{v_j\}$, we have $\mQ_{\bar i,\bar j}=0$.
            \end{enumerate}
        \end{claim}
        \begin{proof}[Proof of Claim~\ref{claim: f_u recursion}.]
        Suppose by contradiction there exists $\ell\in \bigl(\J_{v_i}\cup \{v_i\}\bigr)\cap\bigl(\J_{v_j}\cup \{v_j\}\bigr) $. By the running intersection property (the third condition of Definition~\ref{def: tree decomposition}) of a tree decomposition, the collection of bags containing $\ell$ induces a connected subtree of $\textsf{T}$. Since the bag $\B_u$ lies on the unique path connecting the subtrees rooted at $v_i$ and $v_j$, it follows that $\ell\in \B_u$.
        On the other hand, $J_{v_i}\cup\{v_{i}\} \subset \J_u$ and $\J_{v_j}\cup\{v_j\} \subset \J_u$; moreover, by construction, $\J_u \cap \B_u = \emptyset$. Hence $\bigl(\J_{v_i} \cup {v_i}\bigr) \cap \B_u = \emptyset$ and $(\J_{v_j} \cup {v_j}) \cap \B_u = \emptyset$. This leads to a contradiction since $\ell\in \B_u$. Therefore, $\ell\in \bigl(\J_{v_i}\cup \{v_i\}\bigr)\cap\bigl(\J_{v_j}\cup \{v_j\}\bigr)=\emptyset$, thereby completing the proof of the first statement.

         We proceed to prove the second statement. Fix \(\bar i\in \J_{v_i}\cup \{v_i\}\) and \(\bar j\in \J_{v_j}\cup \{v_j\}\). Suppose by contradiction $\mQ_{\bar{i},\bar{j}} \neq 0$. Then, from the second condition of Definition~\ref{def: tree decomposition}, there must exist a bag that contains both $\bar{i}$ and $\bar{j}$. We proceed to show that this bag cannot exist.
        
         By the first statement of the claim, the sets $\{\J_v\cup\{v\}\}_{v\in\parent_{\textsf{T}}(u)}$ are pairwise disjoint. Hence, since \(\bar i\in \J_{v_i}\cup \{v_i\}\), we have 
         $\bar i\notin \J_{v}\cup\{v\}, \forall v\in\parent_{\textsf{T}}(u)\backslash\{v_i\},$ and similarly
         $\bar j\notin \J_{v}\cup\{v\}, \forall v\in\parent_{\textsf{T}}(u)\backslash\{v_j\}.$ 
         Consequently, the only bags in the tree decomposition that can possibly contain both $\bar i$ and $\bar j$ are $\B_u$ and the bags $\{\B_v\}_{v\in\parent_{\textsf{T}}(u)}$. By construction, however, $\B_u$ contains neither $\bar i$ nor $\bar j$. Therefore, there must exist some $v\in\parent_{\textsf{T}}(u)$ such that $\B_v$ contains both $\bar i$ and $\bar j$.
         Now, by the running intersection property of tree decomposition, the collection of bags containing $\bar i$ forms a connected subtree of $\textsf{T}$, and the same holds for $\bar j$. Since $\B_v$ contains both nodes and $u$ is adjacent to $v$, it follows that $\B_u$ must belong to at least one of these two connected subtrees. In particular, this implies that  $\bar i\in \B_u$ or $\bar j\in \B_u$. This contradicts the fact that $\B_u$ contains neither $\bar i$ nor $\bar j$, thereby completing the proof of the second statement.
    \end{proof}\medskip
     
    Consider the local parametric cost $f_u$ at node $u$. Similar to~\eqref{eq: f_u with valpha fixed}, we have
    \begin{align*}
        f_u(\valpha_{\B_u})&=\frac{1}{2}\valpha_{\B_u}^\top \mQ_{\B_u,\B_u} \valpha_{\B_u}+\vc_{\B_u}^\top \valpha_{\B_u}+\\
        &\min_{\vx\in\R^{n_u},\vz\in\{0,1\}^{n_u}}\Biggl\{ \frac{1}{2}\vx_{\pi_u(\J_u)}^\top \mQ_{\J_u,\J_u} \vx_{\pi_u(\J_u)}+\vx_{\pi_u(\J_u)}^\top \mQ_{\J_u,\B_u} \valpha_{\B_u}+\vc_{\J_u}^\top \vx_{\pi_u(\J_u)}+\vlambda_{\J_u}^\top \vz_{\pi_u(\J_u)} \Biggr\}\\
        &\qquad\qquad \text{s.t.}\qquad \vx_i(1-\vz_i)=0\qquad \forall i\in \pi_u(\J_u).
    \end{align*}
     From~\eqref{eq::h}, the first two terms of the objective function coincide with $h_u(\valpha_{\B_u})$.
    We now turn to the remaining terms in the objective function. First, note that
    \begin{align*}
        \frac{1}{2}\vx_{\pi_u(\J_u)}^\top\mQ_{\J_u,\J_u}\vx_{\pi_u(\J_u)}&=\sum_{v\in\parent_\textsf{T}(u)} \Bigl(\frac{1}{2}\vx_{\pi_u(\J_v\cup v)}^\top\mQ_{\J_v\cup v,\J_v\cup v}\vx_{\pi_u(\J_v\cup v)}\Bigr)\\
        &=\sum_{v\in\parent_\textsf{T}(u)} \Bigl(\frac{1}{2}\vx_{\pi_u(\J_v)}^\top\mQ_{\J_v,\J_v}\vx_{\pi_u(\J_v)}+\frac{1}{2}\mQ_{v,v}\vx^2_{\pi_u(v)}+\vx_{\pi_u(v)}\mQ_{v,\J_v}x_{\pi_u(\J_v)}\Bigr).
    \end{align*}
    This decomposition uses Claim~\ref{claim: f_u recursion}, which asserts that the union $\J_u = \bigcup_{v\in\parent_{\textsf{T}}(u)} \bigl(\J_v \cup \{v\}\bigr)$ is disjoint and the cross blocks satisfy
    \(\mQ_{\J_{v_i}\cup\{v_i\},\,\J_{v_j}\cup\{v_j\}}=\mathbf{0}\), for all distinct
    \(v_i, v_j\in \parent_{\textsf{T}}(u)\).
    
    Next, we consider the the bilinear term coupling variables $\vx_{\pi_u(\J_u)}$ and $\valpha_{\B_u}$.
    \begin{align*}
        \vx_{\pi_u(\J_u)}^\top\mQ_{\J_u,\B_u}\valpha_{\B_u}&=\sum_{v\in\parent_\textsf{T}(u)}\Bigl(\vx_{\pi_u(\J_v\cup v)}^\top\mQ_{\J_v\cup v,\B_u}\valpha_{\B_u}\Bigr)\\
        &=\sum_{v\in\parent_\textsf{T}(u)}\Bigl(\vx_{\pi_u(\J_v)}^\top\mQ_{\J_v,\B_u}\valpha_{\B_u}+\vx_{\pi_u(v)}^\top\mQ_{v,\B_u}\valpha_{\B_u}\Bigr)\\
        &=\sum_{v\in\parent_\textsf{T}(u)}\Bigl(\vx_{\pi_u(\J_v)}^\top\mQ_{\J_v,\B_v\backslash v}\valpha_{\B_v\backslash v}+\vx_{\pi_u(v)}^\top\mQ_{v,\B_v\backslash v}\valpha_{\B_v\backslash v}\Bigr).
    \end{align*}
    To see the third equality, first note that the bag \(\B_u\) admits the disjoint decomposition $\B_u = \bigl(\B_v\setminus\{v\}\bigr) \,\cup\, \bigl(\B_u\setminus \B_v\bigr), \text{for any } v\in\parent_{\textsf{T}}(u)$.
    Furthermore, \(\mQ_{\J_v,\B_u\setminus \B_v} = \mathbf{0}\), since there are no edges between nodes in \(\J_v\) and nodes in \(\B_u\setminus \B_v\).

    We now turn to the remaining terms:
    \begin{align*}
        \vc_{\J_u}^\top \vx_{\pi_u(\J_u)}+\vlambda_{\J_u}^\top \vz_{\pi_u(\J_u)}&=\sum_{v\in\parent_{\textsf{T}}(u)}\Bigl(\vc_{\J_v}^\top \vx_{\pi_u(\J_v)}+\vc_v\vx_{\pi_u(v)}+\vlambda_{\J_v}^\top \vz_{\pi_u(\J_v)}+\vlambda_v\vz_{\pi_v(v)}\Bigr).
    \end{align*}\medskip
    Substituting the terms derived above into the expression for $f_u$, we obtain
    \begin{align*}
        f_u(\valpha_{\B_u})&=h_u(\valpha_{\B_u})+\\
        &\min_{\vx\in\R^{n_u},\vz\in\{0,1\}^{n_u}} \sum_{v\in\parent_{\textsf{T}}(u)}\Bigl( \frac{1}{2}\vx_{\pi_u(\J_v)}^\top\mQ_{\J_v,\J_v}\vx_{\pi_u(\J_v)}+
        \frac{1}{2}\mQ_{v,v}\vx^2_{\pi_u(v)}+\vx_{\pi_u(v)}\mQ_{v,\J_v}x_{\pi_u(\J_v)}\\
        &\qquad\qquad+\vx_{\pi_u(\J_v)}^\top\mQ_{\J_v,\B_v\backslash v}\valpha_{\B_v\backslash v}+\vx_{\pi_u(v)}^\top\mQ_{v,\B_v\backslash v}\valpha_{\B_v\backslash v}\\
        &\qquad\qquad+\vc_{\J_v}^\top \vx_{\pi_u(\J_v)}+\vc_v\vx_{\pi_u(v)}+\vlambda_{\J_v}^\top \vz_{\pi_u(\J_v)}+\vlambda_v\vz_{\pi_u(v)}\Bigr)\\
        &\qquad\qquad \text{s.t.}\qquad \vx_i(1-\vz_i)=0\qquad \forall i\in \pi_u(\J_u).
    \end{align*}
    Since each term in the sum depends only on the variables in $\J_{v}\cup\{v\}$, the optimization problem decomposes into independent subproblems for each $v\in\parent_{\textsf{T}}(u)$. Furthermore, each subproblems equals $g_v(\valpha_{\B_v\backslash v})-\phi_v(\valpha_{\B_v\backslash v})$ by~\eqref{eq: g_u for recursion}. Hence, we conclude that,
    $$f_u(\valpha_{\B_u})=h_u(\valpha_{\B_u})+\sum_{v\in\parent_\textsf{T}(u)}g_v(\B_v\backslash v)-\phi_v(\valpha_{\B_v\backslash v}).$$
    \qed
\subsection{Proof of Lemma~\ref{lemma: U}}\label{app:proof-U}
Let $\J = \{i: \vz^\star_i=1\}$. Then, $\vx^{\star}=(\mQ_{\J,\J})^{-1}\vc_\J,$ which implies
	\begin{align*}
		\|\vx^{\star}\|_{\infty}&=\left\|(\mQ_{\J,\J})^{-1}\vc_\J\right\|_{\infty}\le \left\|(\mQ_{\J,\J})^{-1}\right\|_{\infty} \left\|\vc_\J\right\|_{\infty}\le \left\|(\mQ_{\J,\J})^{-1}\right\|_{\infty} \|\vc\|_{\infty}.
	\end{align*}
	We now derive an upper bound on $\|(\mQ_{\J,\J})^{-1}\|_{\infty}$. By Lemma~\ref{thm: decaying inv}, 
\begin{align*}
    \|(\mQ_{\J,\J})^{-1}\|_{\infty}
    \le \max_{i \in \J} \sum_{j \in \J} 
    \bigl| [\mQ^{-1}_{\I,\I}]_{\pi(i),\pi(j)} \bigr|
    \le \max_{i \in \J} \sum_{j \in \J} C_1 \rho^{\dist(i,j)} .
\end{align*}
Let $N_{r,i}$ denote the number of nodes in $\supp(\mQ)$ at distance exactly $r$ from node $i$, and set $N_r := \max_{i} \{N_{r,i}\}$. Then
\begin{align*}
    \max_{i \in \J}\sum_{j \in \J} C_1 \rho^{\dist(i,j)}
    \;\le\;
    C_1 \sum_{r=0}^{n} N_r \rho^{r}.
\end{align*}

We distinguish two cases:

\begin{itemize}

\item If $\mQ$ is banded with bandwidth $\bw$, then $N_r \le 2\bw$ for all $0 \le r \le n$. Hence,
\begin{align*}
    \|(\mQ_{\J,\J})^{-1}\|_{\infty}
    \le C_1 \sum_{r=0}^{n} 2\bw\, \rho^r
    \le 2\bw C_1 \sum_{r=0}^{\infty} \rho^r
    = \frac{2\bw C_1}{1-\rho}.
\end{align*}

\item If $\supp(\mQ)$ satisfies the polynomial volume-growth condition (Assumption~\ref{assumption: poly growth of m-deg}), then $N_r \le \delta r^\gamma$. Therefore,
\begin{align*}
    \|(\mQ_{\J,\J})^{-1}\|_{\infty}
    \le
    C_1 \sum_{r=0}^{n} \delta r^\gamma \rho^r
    \le \delta C_1 \sum_{r=0}^{\infty} r^\gamma \rho^r
    \le \frac{\delta \gamma! \, C_1}{(1-\rho)^{\gamma+1}},
\end{align*}
where the last inequality uses the classical bound $\sum_{r=0}^{\infty} r^\gamma \rho^r \le \frac{\gamma!}{(1-\rho)^{\gamma+1}}$;
see, e.g., \cite[Proposition~1.4.4]{charalambides2018enumerative}.
\end{itemize}
\qed

    \subsection{Proof of Lemma~\ref{lemma: roots bound}}\label{subsec:roots-proof}
Consider the quadratic equation $ p_{1}(\valpha)- p_{2}(\valpha)=0$. Let $\hat\valpha$ denote its solution. Then
	\begin{align*}
		 & p_{1}(\hat \valpha)- p_{2}(\hat \valpha)=0\\
		\implies & \frac{1}{2}\hat\valpha^\top(\mA_{1}-\mA_{2})\hat\valpha+(\vb_{1}-\vb_{2})^\top \hat\valpha+(d_{1}-d_{2})=0\\
		\implies &\frac{1}{2}\norm{\mA_1-\mA_2}_{1,1} \|\hat\valpha\|_{\infty}^2+\|\vb_{1}-\vb_{2}\|_1\|\hat\valpha\|_\infty+(d_{1}-d_{2})\ge 0\\
        \implies & \bar a \|\hat\valpha\|_{\infty}^2+\bar b\|\hat\valpha\|_{\infty}\ge \bar d,
	\end{align*}
	where, without loss of generality, we assume $d_{2}\ge d_{1}$. 
	We consider two cases. Indeed, if $\bar a=0$, then
	\begin{align*}
		\|\valpha\|_{\infty}\ge \frac{\bar d}{\bar b }.
	\end{align*}
On the other hand, if $\bar a>0$, we have
	\begin{align*}   
		&\|\hat \valpha\|_{\infty}^2+\frac{\bar b}{\bar a}\|\hat\valpha\|_{\infty}\ge \frac{\bar d}{\bar a}\\
        \implies &\|\hat \valpha\|_{\infty}^2+\frac{\bar b}{\bar a}\|\hat\valpha\|_{\infty}+\frac{\bar b^2}{4\bar a^2}\ge \frac{\bar d}{\bar a}+\frac{\bar b^2}{4\bar a^2}\\
		\implies &\|\hat \valpha\|_{\infty}\ge-\frac{\bar b}{2\bar a}+ \sqrt{\frac{\bar d}{\bar a}+\frac{\bar b^2}{4\bar a^2}}.
	\end{align*}
	Rearranging the terms, we get 
	\begin{align*}
		\|\hat \valpha\|_\infty&\ge\frac{-\bar b+\sqrt{\bar b^2+4\bar a\bar d}}{2\bar a}.
	\end{align*}
	Combining the two cases, we conclude that
	\begin{align*}
		\|\hat\valpha\|_{\infty}\ge\begin{cases}
			\frac{-\bar b+\sqrt{\bar b^2+4\bar a\bar d}}{2\bar a}&\text{if}\ \bar a\ne 0,\\
			\frac{\bar d}{\bar b }&\text{if}\ \bar a= 0.
		\end{cases}
	\end{align*}
	\qed\medskip

\subsection{Proof of Lemma~\ref{thm: decaying inv}}\label{app:exp-decay-proof}

The proof proceeds via a polynomial approximation of the function $\psi(x)=x^{-1}$ on a
compact interval, for which we use the following result.

\begin{proposition}[Polynomial approximation of $x^{-1}$,\citep{meinardus2012approximation}]\label{prop: chebyshev}

Let $0<a<b$. For $\ell \ge 0$, let $\Pi_\ell$ denote the set of polynomials of degree at most $\ell$, and define
\[
e_\ell([a,b]) := \inf_{\bar p\in\Pi_\ell}\, \max_{x\in[a,b]} \bigl|x^{-1}-\bar p(x)\bigr|.
\]
Set $r=b/a$ and
\[
\rho := \frac{\sqrt{r}-1}{\sqrt{r}+1}.
\]
Then
\[
e_\ell([a,b]) = \frac{(1+\sqrt{r})^2}{2ar}\, \rho^{\,\ell+1}.
\]
\end{proposition}

Let $a=\mu_{\min}(\mQ), b=\mu_{\max}(\mQ)$, $C_0=\frac{(1+\sqrt{\kappa_2})^2}{2\kappa_2\mu_{\min}(\mQ)}$ and $\sigma(\mQ)$ denote the set of eigen values of $\mQ$. Since $\mQ$ is positive definite and invertible, we have $0<a<b$. From the Proposition~\ref{prop: chebyshev} there exist a sequence of polynomials $\bar p_\ell\in \Pi_{\ell}$ satisfying 
$$
\max_{x\in [a,b]} \left\{\left|x^{-1}-\bar p_{\ell}(x)\right|\right\}=C_{0}\rho^{\ell+1}.
$$
From spectral theory \citep{rudin1973functional}
\begin{align*}
    \|\mQ^{-1}-\bar p_{\ell}(\mQ)\|_2&=\max_{x\in \sigma(A)} \left\{\left|x^{-1}-\bar p_{\ell}(x)\right|\right\}\\
    &\le \max_{x\in [a,b]} \left\{\left|x^{-1}-\bar p_{\ell}(x)\right|\right\}\\
    &=C_{0}\rho^{\ell+1}.
\end{align*}

Moreover, for every integer $k\ge 0$ we have
\[
(\mQ^k)_{ij}=0 \qquad \text{whenever } \dist(i,j)>k.
\]
Hence, if $\bar p_k\in \Pi_k$ then $\bar p_k(\mQ)$ satisfies
\[
\bar p_k(\mQ)_{ij}=0 \qquad \text{whenever } \dist(i,j)>k,
\]
since $\bar p_k(\mQ)$ is a linear combination of $\mathbf{I},\mQ,\ldots,{\mQ}^k$.

Let $i\neq j$ and choose $\ell=\dist(i,j)-1$. Then $\bar p_\ell(\mQ)_{ij}=0$, and hence
\[
\bigl|\mQ^{-1}_{ij}\bigr|
= \bigl|\mQ^{-1}_{ij}-\bar p_\ell(\mQ)_{ij}\bigr|
\le \bigl\|\mQ^{-1}-\bar p_{\ell}(\mQ)\bigr\|_2
\le C_0 \rho^{\,\dist(i,j)}.
\]
If $i=j$, note that
\[
\bigl|\mQ^{-1}_{ii}\bigr| \le \bigl\|\mQ^{-1}\bigr\|_2 = \frac{1}{a},
\]
which is consistent with the stated bound. This completes the proof.\qed\medskip

\section{Heuristic for PRUNE}\label{app: trim heuristic}

To determine the set of relevant functions, the exact version of \texttt{PRUNE} requires comparing every pair of quadratic functions, leading to a runtime that is quadratic in $N$, where $N$ is the total number of functions of the input. 

To mitigate this quadratic cost, we provide a heuristic version, given in Algorithm~\ref{alg: trim heuristic}. This procedure processes the functions in a single pass, requiring only $N-1$ comparisons. This linear-time approach may retain some functions that are formally irrelevant, but the total number of retained functions remains consistent with the bound established in Lemma~\ref{lemma::pruned-set}. 

\begin{algorithm}[htbp]
	\caption{heuristic-\texttt{PRUNE}}\label{alg: trim heuristic}
	\textbf{Input:} $\tilde f(\valpha)\equiv [ p_i(\valpha)]\ \forall i=1,\ldots,N$ and the constant $U$. \\
	\textbf{Output: $ f^{\trim}$} 
	\begin{algorithmic}[1]
		\State $ f^{\trim} \leftarrow \left[ p_{1}\right]$ \Comment{Add $p_{1}$ to $f^{\trim}$}
		\State $j=1$
		\For{$i=2,\dots, N$}
			\State Calculate $L(p_{i},p_{j})$ from Equation~\eqref{eq: LB} \Comment{Lower bound of roots of $ p_i- p_j=0$}
			\If{$L(p_{i},p_{j}) > U$}
			\If{$d_i<d_j$}
			\State $f^{\trim} \leftarrow \texttt{delete}\left( f^{\trim},\texttt{end}\left(f^{\trim}\right)\right)$ \Comment{Remove the last function $ p_{j}$ from $ f^{\trim}$}
			\State $ f^{\trim} \leftarrow \texttt{append}\left( f^{\trim}, p_{i}\right)$ \Comment{Add $ p_{i}$ to $f^{\trim}$}
			\State $j=i$
			\EndIf
			\Else
			\State $ f^{\trim} \leftarrow \texttt{append}( f^{\trim}, p_{i})$
			\State $j=i$
			\EndIf
		\EndFor
		\State\Return $f^{\trim}$ 
		\end{algorithmic}
	\end{algorithm}	
This improvement is possible when the quadratic functions are stored where every consecutive pair of functions is $m$-similar. In this case, it suffices to check for irrelevance only among consecutive pairs of functions. When $\mQ$ has a tree decomposition that is a path (including the banded case), we note that this ordering can be maintained with no additional cost. However, in the general case, storing the functions in the desired order requires additional steps, which can increase computational cost, making the heuristic inefficient. For this reason, we rely on the version of \texttt{PRUNE} presented in Algorithm~\ref{alg: prune} for the general case.

\end{document}